\documentclass[11pt]{amsart}

\usepackage{epsf,amssymb,times,overpic,amscd}
\usepackage{color}
\usepackage[color,matrix,arrow]{xy}
\usepackage{hyperref}
\usepackage{amsfonts}
\usepackage{amsmath}
\usepackage{graphicx}
\usepackage{epsfig}
\usepackage{epstopdf}

\newcommand{\ba}{\begin{array}}
\newcommand{\ea}{\end{array}}
\newcommand{\be}{\begin{enumerate}}
\newcommand{\ee}{\end{enumerate}}

\newtheorem{thm}{Theorem}[section]
\newtheorem{prop}[thm]{Proposition}
\newtheorem{lemma}[thm]{Lemma}

\theoremstyle{definition}
\newtheorem{defn}[thm]{Definition}

\theoremstyle{remark}

\newtheorem{rmk}[thm]{Remark}
\newtheorem{example}[thm]{Example}

\newcommand{\C}{\mathcal{C}}

\newcommand{\Z}{\mathbb{Z}}

\newcommand{\N}{\mathbb{N}}

\newcommand{\deh}{\operatorname{deg}_h}
\newcommand{\dt}{\operatorname{deg}_t}
\newcommand{\dta}{\operatorname{deg}_{t_1}}
\newcommand{\dtb}{\operatorname{deg}_{t_2}}
\newcommand{\bdry}{\partial}
\newcommand{\tC}{\tilde{\mathcal{C}}}

\newcommand{\n}{\noindent}
\newcommand{\F}{\mathbb{F}_2}
\newcommand{\aoa}{A \otimes A}
\newcommand{\aor}{A \boxtimes R_n}
\newcommand{\uq}{\mathbf{U}_q(\mathfrak{sl}(1|1))}
\newcommand{\Ut}{\mathbf{U}_T(\mathfrak{sl}(1|1))}
\newcommand{\ut}{\mathbf{U}_T}
\newcommand{\utt}{\mathbf{U}_T \otimes_{\Z} \mathbf{U}_T}

\newcommand{\zt}{\Z[t^{\pm1}]}
\newcommand{\zT}{\Z[T^{\pm1}]}
\newcommand{\ztt}{\Z[T_1^{\pm1}, T_2^{\pm1}]}

\newcommand{\bt}{\boxtimes}
\newcommand{\ot}{\otimes}
\newcommand{\de}{\Delta}
\newcommand{\rn}{R_n}

\newcommand{\ra}{\rightarrow}
\newcommand{\xra}{\xrightarrow}
\newcommand{\bn}{\mathcal{B}_n}
\newcommand{\bnk}{\mathcal{B}_{n,k}}
\newcommand{\g}{\Gamma}
\newcommand{\gn}{Q_n}
\newcommand{\gnk}{Q_{n,k}}

\newcommand{\mf}{\mathbf}
\newcommand{\op}{\operatorname}
\newcommand{\cal}{\mathcal}
\newcommand{\es}{\emptyset}
\newcommand{\lan}{\langle}
\newcommand{\ran}{\rangle}

\begin{document}
%%%%%%%%%%%%%%%%%%%%%%%%%%%%%%%%%%%%%%%%%%%%%%%%%%%%%%%%%
\title{A categorification of $\Ut$ and its tensor product representations}

\author{Yin Tian}
\address{University of Southern California, Los Angeles, CA 90089}
\email{yintian@usc.edu}
%\urladdr{http://www.math.}

%\subjclass{Primary 57M50; Secondary 53C15}

\begin{abstract}
We define the Hopf superalgebra $\Ut$, which is a variant of the quantum supergroup $\uq$, and its representations $V_1^{\ot n}$ for $n>0$.
We construct families of DG algebras $A$, $B$ and $R_n$, and consider the DG categories $DGP(A)$, $DGP(B)$ and $DGP(R_n)$, which are full DG subcategories of the categories of DG $A$-, $B$- and $R_n$-modules generated by certain distinguished projective modules.
Their $0$th homology categories $HP(A)$, $HP(B)$, and $HP(R_n)$ are triangulated and give algebraic formulations of the contact categories of an annulus, a twice punctured disk, and an $n$ times punctured disk.
Their Grothendieck groups are isomorphic to $\Ut$, $\Ut \ot_{\Z} \Ut$ and $V_1^{\ot n}$, respectively.
We categorify the multiplication and comultiplication on $\Ut$ to a bifunctor $HP(A) \times HP(A) \ra HP(A)$ and a functor $HP(A) \ra HP(B)$, respectively.
The $\Ut$-action on $V_1^{\ot n}$ is lifted to a bifunctor $HP(A) \times HP(R_n) \ra HP(R_n)$.
\end{abstract}
\maketitle
%%%%%%%%%%%%%%%%%%%%%%%%%%%%%%%%%%%%%%%%%

\section{Introduction}
\subsection{Background}
This paper is a sequel to \cite{Tian} in which we categorified the algebra structure of an integral version of the quantum supergroup $\uq$.
The goal of this paper is to present a categorification of a Hopf superalgebra $\Ut$ (a variant of $\uq$) and its representations $V_1^{\ot n}$ for $n>0$, where $V_1$ is the two-dimensional fundamental representation.

In the late 1980's, Witten \cite{Wi} and Reshetikhin-Turaev \cite{RT} established a connection between quantum groups and knot invariants.
In particular, the Jones polynomial could be recovered as the Witten-Reshetikhin-Turaev invariant of the fundamental representation of $\mathbf{U}_q(\mathfrak{sl}_2)$.
For quantum supergroups, Kauffman and Saleur in \cite{KS} developed an analogous representation-theoretic approach to the Alexander polynomial, by considering the fundamental representation $V_0$ of $\uq$.
Rozansky and Saleur in \cite{RS} gave a corresponding quantum field theory description.

The connection between quantum groups and knot invariants can be lifted to the categorical level.
The existence of such a lifting process, called {\em categorification}, was conjectured by Crane and Frenkel in \cite{CF}.
In the seminal paper \cite{Kh1}, Khovanov defined a doubly graded homology, now called {\em Khovanov homology}, whose graded Euler characteristic agreed with the Jones polynomial.
Chuang and Rouquier \cite{CR} categorified locally finite $\mathfrak{sl}_2$-representations, and more generally, Rouquier \cite{Rou} constructed a $2$-category associated with a Kac-Moody algebra.
For the quantum groups themselves, Lauda \cite{La} gave a diagrammatic categorification of $\mathbf{U}_q(\mathfrak{sl}_2)$ and general cases are given by Khovanov-Lauda \cite{KL1, KL3, KL2}.
The program of categorifying Witten-Reshetikhin-Turaev invariants was brought to fruition by Webster \cite{Web1, Web2} using the diagrammatic approach.

On the other hand, the Alexander polynomial is categorified by {\em knot Floer homology}, defined independently by Ozsv\'ath-Szab\'o \cite{OS1} and Rasmussen \cite{Ra}.
Although its initial definition was through Lagrangian Floer homology, knot Floer homology admits a completely combinatorial description by Manolescu-Ozsv\'ath-Sarkar \cite{MOS}.
It is natural to ask whether there is a categorical program for $\uq$ which is analogous to that of $\mathbf{U}_q(\mathfrak{sl}_2)$ and which recovers knot Floer homology.

This paper presents another step towards such a categorical program.
We first define the Hopf superalgebra $\Ut$ as a variant of $\uq$ and the representations $V_1^{\ot n}$ of $\Ut$.
Then we categorify the multiplication and comultiplication on $\Ut$, and its representations $V_1^{\ot n}$.
In a subsequent paper \cite{Tian2}, we will categorify the action of the braid group on $V_1^{\ot n}$ which is induced by the R-matrix structure of $\Ut$.

Our motivation is from the {\em contact category} introduced by Honda \cite{Honda1} which presents an algebraic way to study $3$-dimensional contact topology.
The contact category is closely related to {\em bordered Heegaard Floer homology} defined by Lipshitz, Ozsv\'ath and Thurston \cite{LOT}.
See Section 1.3 for more detail on the contact category.
Motivated by the strands algebra in bordered Heegaard Floer homology, Khovanov in \cite{Kh2} categorified the positive part of $\mathbf{U}_q (\mathfrak{gl}(1|2))$.
A counterpart of our construction in Lie theory is developed by Sartori in \cite{Sa} using subquotient categories of $\cal{O}(\mathfrak{gl}_n)$.

\subsection{Main results}
We define the Hopf superalgebra $\Ut$ as an associative $\Z$-algebra with unit $I$, generators $E, F, T, T^{-1}$ and relations:
\begin{gather}
E^2=F^2=0, \notag\\
EF+FE=I-T,  \label{CT}\\
ET=TE,~ FT=TF, \quad TT^{-1}=T^{-1}T=I.\notag
\end{gather}
The comultiplication $\Delta: \Ut \ra \Ut \otimes_{\Z} \Ut$ is given by:
\begin{gather*}
\de(E)=E \ot I + I \ot E, \quad \de(F)=F \ot T + I \ot F, \quad \de(T)=T \ot T.
\end{gather*}

Recall from \cite{KS} that the commutator relation of $\uq$ is:
\begin{gather}
EF+FE=\frac{H - H^{-1}}{q- q^{-1}}. \label{CQ}
\end{gather}
To see the relation between $\Ut$ and $\uq$, we compare their commutator relations (\ref{CT}) and (\ref{CQ}) by setting $T=H^{-2}$.
Then the right hand side of (\ref{CQ}) is equal to that of (\ref{CT}) multiplied by $\frac{H}{q- q^{-1}}$.

Let $\ut$ denote $\Ut$ from now on.
Let $V_1$ be a free $\zt$-module spanned by $\cal{B}_1=\{|0 \ran, |1 \ran \}$ which admits an action of $\ut$ given by:
\begin{gather*}
E|0\ran=0, \quad F|0\ran=|1\ran, \\
E|1\ran=(1-t)|0 \ran, \quad F|1\ran=0,\\
T|0\ran=t|0\ran, \quad T|1\ran=t|1\ran.
\end{gather*}
Consider the $n$th tensor product representation $V_1^{\ot n}$ induced by the iterated comultiplication of $\ut$.
Note that $T \cdot v=t^n v$ for $v \in V_1^{\ot n}$ since $\Delta(T)=T\ot T$.
Our categorification of $V_1^{\ot n}$ is built on a distinguished basis $\cal{B}'_n = \cal{B}_1^{\times n}=\{\mf{a}=|a_1 \dots  a_n\ran ~|~ a_i \in \{0, 1\} \},$
where $|a_1 \dots  a_n\ran$ is the shorthand for $|a_1\ran \ot \dots \ot |a_n\ran$.
The followings are the main results of this paper:
\begin{thm} [Categorification of the multiplication on $\ut$] \label{thm-ut}
There exist a triangulated category $HP(A)$ whose Grothendieck group is $\ut$ and an exact bifunctor $$\mathcal{M}: HP(A) \times HP(A) \ra HP(A)$$ whose induced map $K_0(\cal{M}): \ut \times \ut \ra \ut$ on the Grothendieck groups
agrees with the multiplication on $\ut$.
\end{thm}

\begin{thm} [Categorification of the comultiplication on $\ut$] \label{thm-ut-com}
There exist a triangulated category $HP(B)$ whose Grothendieck group is $\utt$ and an exact functor $$\delta: HP(A) \ra HP(B)$$ whose induced map $K_0(\delta): \ut \ra \utt$ on the Grothendieck groups
agrees with the comultiplication on $\ut$.
\end{thm}

\begin{thm} [Categorification of the $\ut$-module $V_1^{\ot n}$] \label{thm-utvn}
For each $n>0$, there exist a triangulated category $HP(H(R_n))$ whose Grothendieck group is $V_1^{\ot n}$ and an exact bifunctor $$\mathcal{M}_n: HP(A) \times HP(H(R_n)) \ra HP(H(R_n))$$ whose induced map $K_0(\cal{M}_n)$ on the Grothendieck groups
agrees with the action $\ut \times V_1^{\ot n} \rightarrow V_1^{\ot n}$.
\end{thm}

The topological motivation for categorifying the multiplication and comultiplication is completely different.
In the corresponding algebraic formulations, we use $HP(A) \times HP(A)$ for the multiplication and use $HP(B)$ for the comultiplication.
See Figures \ref{1-1} and \ref{1-2} for more detail about the topological motivation.

\subsection{Motivation from contact topology}
The contact category $\C(\Sigma, F)$ of $(\Sigma, F)$ is an additive category associated to an oriented surface $\Sigma$ and a finite subset $F$ of $\bdry\Sigma$.
The objects of $\C(\Sigma, F)$ are formal direct sums of isotopy classes of {\em dividing sets} on $\Sigma$ whose restrictions to $\bdry\Sigma$ agree with $F$.
A {\em dividing set} $\g$ on $\Sigma$ is a properly embedded 1-manifold, possibly disconnected and possibly with boundary, which divides $\Sigma$ into positive and negative regions.
The morphism $\op{Hom}_{\C(\Sigma, F)}(\g_0,\g_1)$ is an $\F$-vector space spanned by isotopy classes of {\em tight} contact structures on $\Sigma \times [0,1]$ with the dividing sets $\g_i$ on $\Sigma \times \{i\}$ for $i=0,1$.
The composition is given by vertically stacking contact structures.
Any dividing set with a contractible component is isomorphic to the zero object since there is no tight contact structure in a neighborhood of the dividing set by a criterion of Giroux \cite{Gi}.
As basic blocks of morphisms, {\em bypass attachments} introduced by Honda \cite{Honda2} locally change dividing sets as in Figure \ref{1-0}.
Honda-Kazez-Mati\'c \cite{HKM1} gave a criterion for the addition of a collection of disjoint bypasses to be tight.
\begin{figure}[h]
\begin{overpic}
[scale=0.2]{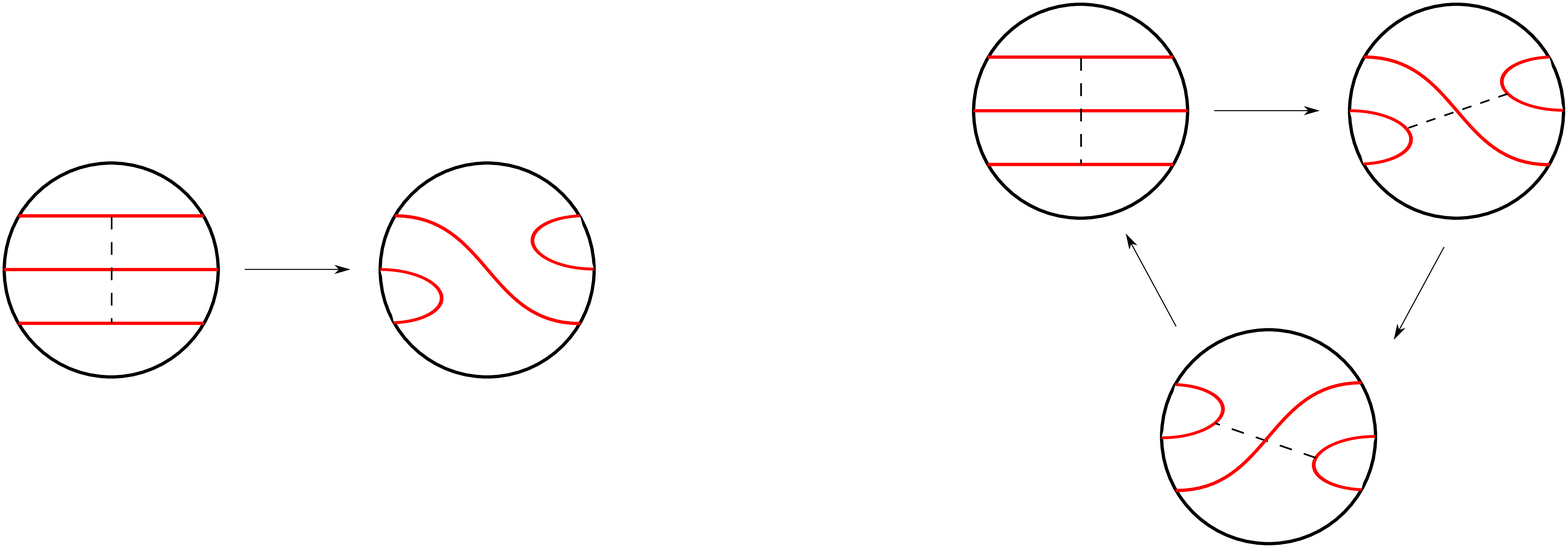}
\end{overpic}
\caption{The picture on the left is a bypass attachment along the dashed arc; the one on the right is a distinguished triangle given by a triple of bypass attachments.}
\label{1-0}
\end{figure}

The connection between $3$-dimensional contact topology and {\em Heegaard Floer homology} was established by Ozsv\'ath and Szab\'o \cite{OS2} in the closed case.
Honda-Kazez-Mati\'c generalized it to the case of a contact 3-manifold with {\em convex} boundary in \cite{HKM3} and formulated it in the framework of TQFT in \cite{HKM2}.
The combinatorial properties of this TQFT were studied by Mathews in the case of disks \cite{Ma1} and annuli \cite{Ma2}.
The connection on the categorical level is observed by Zarev \cite{Za2}.

There is a refined version, called the {\em universal cover} $\tC(\Sigma, F)$ of the contact category $\C(\Sigma,F)$ given as follows.
Choose a dividing set $\g_0$ as a base point.
The basic objects of $\tC(\Sigma,F)$ are pairs $(\g, [\zeta])$, where $\g$ is an isotopy class of dividing sets on $(\Sigma, F)$, and $[\zeta]$ is a homotopy class of a $2$-plane field $\zeta$ on $\Sigma \times [0,1]$ which is contact near $\Sigma \times \{0,1\}$ with the dividing sets $\g_0$ on $\Sigma \times \{0\}$ and $\g$ on $\Sigma \times \{1\}$.
The morphism set $\op{Hom}_{\tC(\Sigma,F)}((\g_1, [\zeta_1]), (\g_2, [\zeta_2]))$ is spanned by tight contact structures $\{\xi\}$ such that $[\zeta_2]=[\xi \cup \zeta_1]$, where $\xi \cup \zeta_1$ denotes a concatenation of the $2$-plane fields $\xi$ and $\zeta_1$.
In other words, the component $[\zeta]$ gives a grading $\op{gr}(\Sigma)$ on the objects of $\tC(\Sigma,F)$ which takes values in homotopy classes of $2$-plane fields.
Equivalently, the grading $\op{gr}(\Sigma)$ is given by a central extension by $\Z$ of the homology group $H_1(\Sigma)$, i.e., there is a short exact sequence:
$0 \ra \Z \ra \op{gr}(\Sigma) \ra H_1(\Sigma) \ra 0.$
Note that a similar grading appears in bordered Heegaard Floer homology \cite[Section 3.3]{LOT}.
The main feature of the universal cover $\tC(\Sigma,F)$ is the existence of distinguished triangles given by a triple of bypass attachments as in Figure \ref{1-0}.
The subgroup $\Z$ of the grading $\op{gr}(\Sigma)$ is related to the shift functor in a triangulated category.
In particular, Huang \cite{Huang} showed that a triple of bypass attachments changes the $\Z$ component by $1$.

This paper provides an algebraic formulation of the universal covers of the contact categories of an annulus, a twice punctured disk and an $n$ times punctured disk.
Let $\tC _o$ be the universal cover of $\C(S_o, F_o)$, where $S_{o}$ is an annulus and $F_o$ consists of two points on each boundary component.
Then $\tC _o$ is a monoidal category with a bifunctor $\cal{M}: \tC _o \times \tC _o \ra \tC _o$ defined by stacking two dividing sets along their common boundaries of two annuli for objects and gluing two contact structures for morphisms. See Figure \ref{1-1}.
The Grothendieck group $K_0(\tC _o)$ is isomorphic to $\ut$, where the multiplication on $\ut$ is lifted to the monoidal functor $\cal{M}$.
A $\zT$-basis of $K_0(\tC _o)$ is given by classes of dividing sets in $\cal{B}=\{I, E, F, EF\}$, where $EF$ is the stacking of $E$ and $F$ under $\cal{M}$.
The generator of $H_1(S_o)$ in the grading $\op{gr}(S_o)$ corresponds to the central element $T \in \ut$.
\begin{figure}[h]
\begin{overpic}
[scale=0.25]{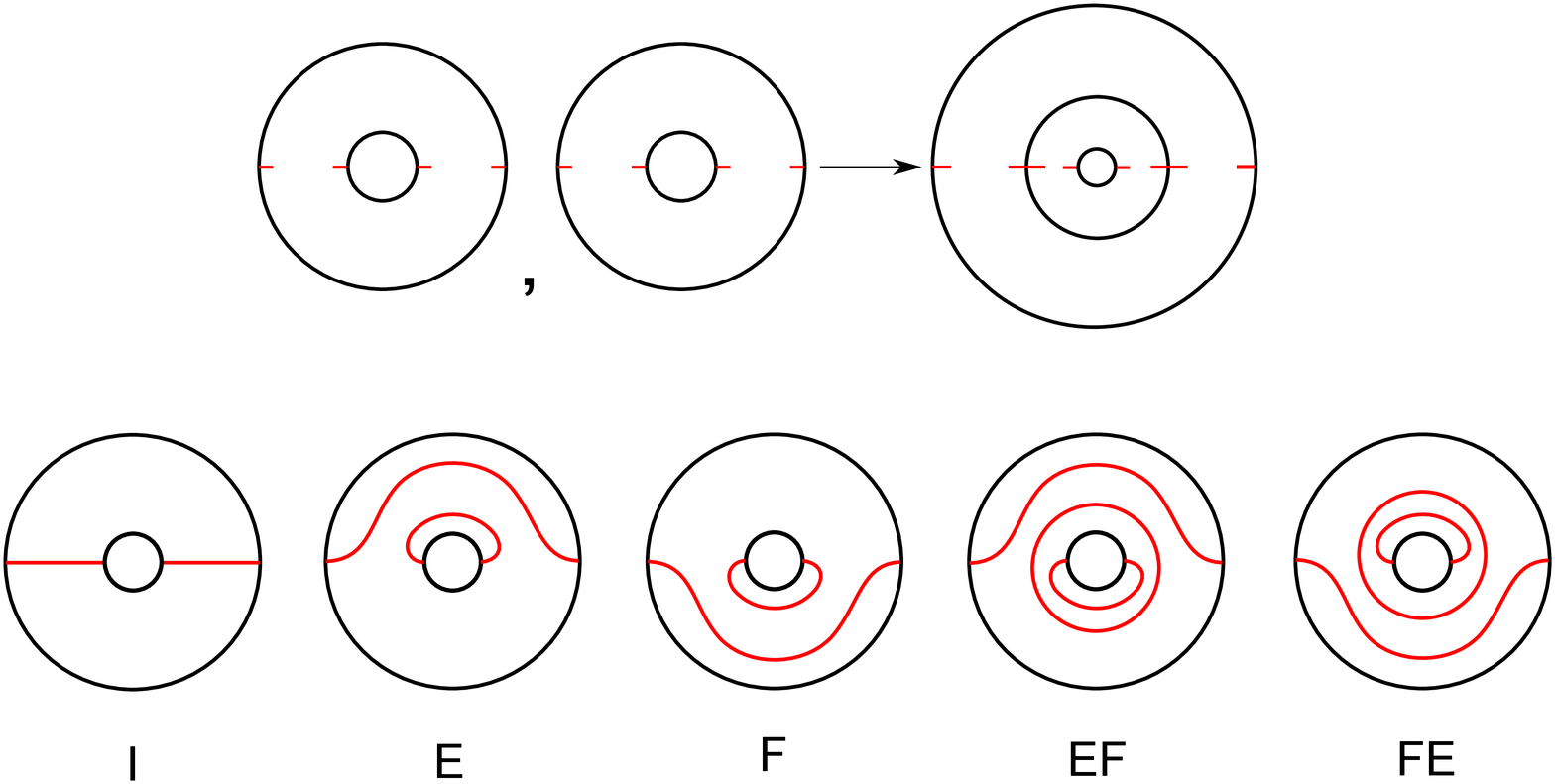}
\put(23,43){$\g_1$}
\put(42,43){$\g_2$}
\put(53,42){$\cal{M}$}
\put(70,47){${\scriptstyle \g_1}$}
\put(70,41.5){${\scriptscriptstyle \g_2}$}
\end{overpic}
\caption{The upper picture describes the monoidal functor $\cal{M}$ on objects; the lower one consists of the distinguished basis of $K_0(\tC _o)$ and the dividing set $FE$.}
\label{1-1}
\end{figure}
The commutator relation (\ref{CT}) is lifted to two distinguished triangles in $\tC _o$:
$I \ra EF \ra K^{-1}$ and $K^{-1} \ra I \ra FE$ as in Figure \ref{2-1}, where the homotopy gradings are ignored.
\begin{figure}[h]
\begin{overpic}
[scale=0.2]{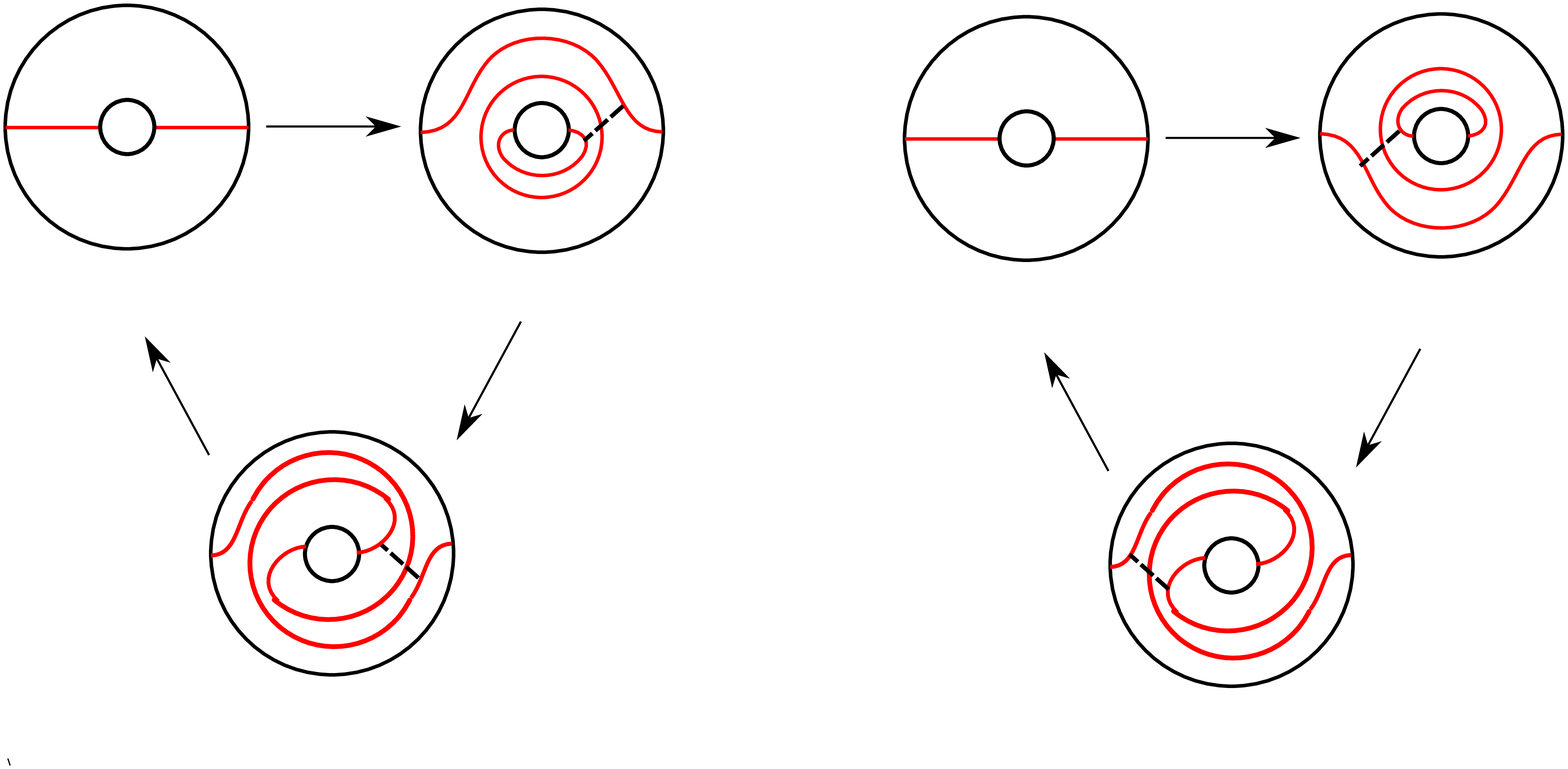}
\put(8,29){$I$}
\put(31,29){$EF$}
\put(65,28){$I$}
\put(89,28){$FE$}
\put(19,0){$K^{-1}$}
\put(77,0){$K^{-1}$}
\end{overpic}
\caption{Two distinguished triangles lift the commutator relation (\ref{CT}).}
\label{2-1}
\end{figure}

To categorify $\utt$ in the comultiplication, consider $\tC _{oo}$ as the universal cover of $\C(S_{oo}, F_{oo})$, where $S_{oo}$ is a twice punctured disk and $F_{oo}$ consists of two points on each boundary component.
A distinguished basis of the Grothendieck group $K_0(\tC _{oo})$ is given by classes of dividing sets in $\{\g_1 \ot \g_2 ~|~ \g_1,\g_2 \in \cal{B}\}$ as in Figure \ref{1-2}.
There are two generators $t_1, t_2 \in H_1(S_{oo})$ in the grading $\op{gr}(S_{oo})$ given by the two loops.
They correspond to the central elements $T\ot I, I\ot T \in \utt$.
Hence $K_0(\tC _{oo})$ is isomorphic to $\utt$.
To categorify the comultiplication $\de: \ut \ra \utt$, define a functor $\delta: \tC _{o} \ra \tC _{oo}$ on objects by stacking dividing sets $\g \in \tC _{o}$ with the specific dividing set $I\ot I \in \tC _{oo}$ along the outmost boundary of $S_{oo}$, on morphisms by gluing contact structures in $S_{o} \times [0,1]$ with the {\em $I$-invariant} contact structure of $I \ot I$ in $S_{oo} \times [0,1]$.
An $I$-invariant contact structure is the trivial contact structure which contains no nontrivial bypass attachments.
Then the decategorification $K_0(\delta): \ut \ra \utt$ agrees with the comultiplication $\de$.
For instance, $\de(E)=E\ot I + I\ot E$ is lifted to a distinguished triangle: $I\ot E \ra \delta(E) \ra E\ot I$.
\begin{figure}[h]
\begin{overpic}
[scale=0.25]{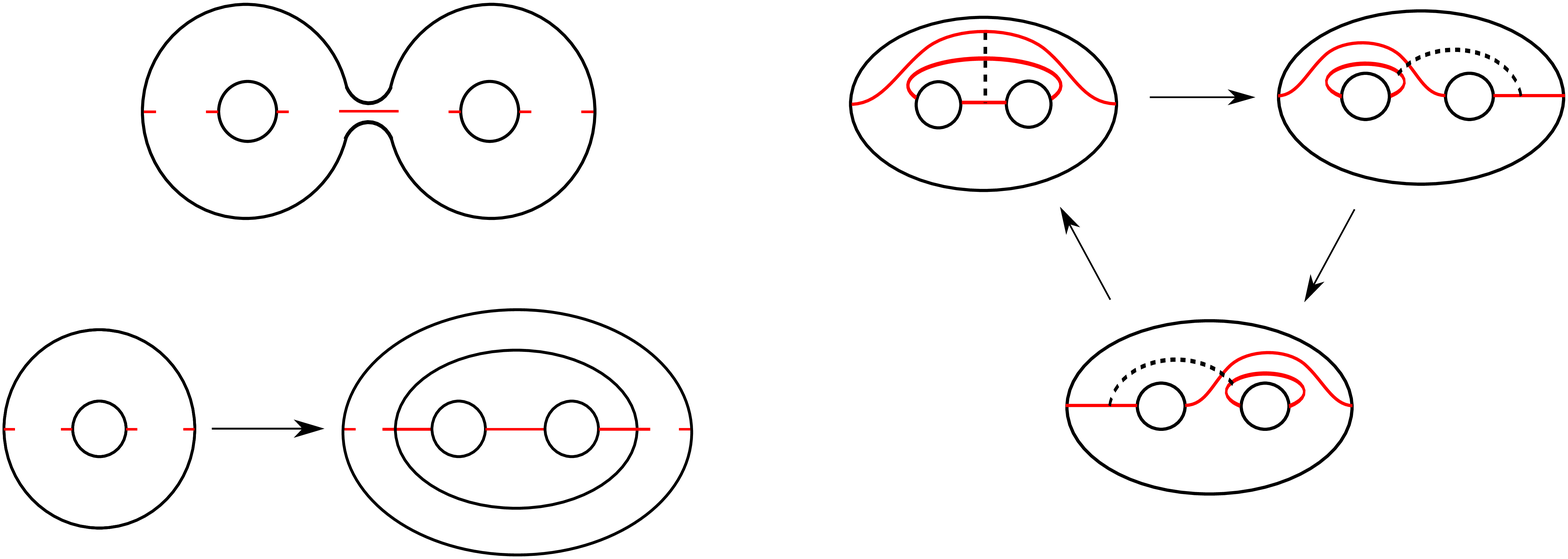}
\end{overpic}
\put(-300,112){$\g_1$}
\put(-248,112){$\g_2$}
\put(-300,35){$\delta$}
\put(-333,38){$\g$}
\put(-242,48){${\scriptstyle \g}$}
\put(-145,70){$\delta(E)$}
\put(-40,70){$E\ot I$}
\put(-100,0){$I\ot E$}
\caption{The upper picture on the left describes the basis $\{\g_1\ot \g_2\}$ of $K_0(\tC _{oo})$; the lower picture on the left gives the comultiplication $\delta$; the picture on the right is the triangle $I\ot E \ra \delta(E) \ra E\ot I$ in $\tC _{oo}$.}
\label{1-2}
\end{figure}

To categorify $V_1^{\ot n}$, consider $\tC _n$ as the universal cover of $\C(\Sigma_n, F_n)$, where $\Sigma_n$ is an $n$ times punctured disk and $F_n$ contains two marked points on the outermost boundary and no points on the other boundary components\footnote{Since there is no marked point on interior boundary components $\bdry\Sigma_n'$, the boundary restriction of contact structures in $\op{Hom}_{\tC _n}(\g_1, \g_2)$ is a collection of dividing sets $(\bdry\Sigma_n' \times \{1/2\} \cup \g_1 \times \{0\} \cup \g_2 \times \{1\}) \subset \bdry(\Sigma \times [0,1])$.}.
The Grothendieck group $K_0(\tC _n)$ is a free module over $\Z[t_1^{\pm1},t_2^{\pm1},\cdots,t_n^{\pm1}]$, where $t_i$ is the generator in $H_1(\Sigma_n)$ corresponding to the $i$th loop.
A quotient of $K_0(\tC _n)$ by the relations $t_1=t_2=\cdots=t_n=t$ is isomorphic to the $\ut$-module $V_1^{\ot n}$.
A distinguished collection of dividing sets in $\tC _n$ is obtained by lifting the basis $\bn'$ of $V_1^{\ot n}$.
See Figure \ref{1-3} for $\tC _1$ and $\tC _2$.
Note that $\tC _1$ and $\tC _{o}$ have the same underlying surface but with different boundary conditions.
The $\ut$ action on $V_1^{\ot n}$ is lifted to a functor $\cal{M}_n: \tC _{o} \times \tC _n \ra \tC _n$ given by stacking dividing sets on the annulus $S_{o}$ with those on $\Sigma_n$ along the outermost boundary of $\Sigma_n$.

We give some morphism sets in $\tC_1$ and $\tC_2$ as follows.
There is a unique tight contact structure $e_{\es} \in \op{Hom}_{\tC _1}(|0\ran, |0\ran)$ which is $I$-invariant.
In particular, $e_{\es}$ is an idempotent under the composition: $e_{\es}\cdot e_{\es}=e_{\es}$.
There are exactly two tight contact structures $e_1,\rho_1 \in \op{Hom}_{\tC _1}(|1\ran, |1\ran)$, where $e_1$ is $I$-invariant and $\rho_1$ is nilpotent, i.e., the composition $\rho_1 \cdot \rho_1$ is not tight.
\begin{gather*}
\op{Hom}_{\tC _1}(|0\ran, |0\ran)=\lan e_{\es} \ran; \quad \op{Hom}_{\tC _1}(|1\ran, |1\ran)=\lan e_1,\rho_1 ~|~ \rho_1^2=0 \ran.
\end{gather*}
The nonzero morphism sets in $\tC _2$ are the followings:
\begin{gather*}
\op{Hom}_{\tC _2}(|00\ran, |00\ran)=\lan e_{\es} \ran; \\
\op{Hom}_{\tC _2}(|01\ran, |01\ran)=\lan e_{2}, \rho_2 ~|~ \rho_2^2=0 \ran, \\
\op{Hom}_{\tC _2}(|10\ran, |10\ran)=\lan e_{1}, \rho_1 ~|~ \rho_1^2=0 \ran, \\
\op{Hom}_{\tC _2}(|01\ran, |10\ran)=\lan r,~ \rho_2 \cdot r,~ r \cdot \rho_1,~ \rho_2 \cdot r \cdot \rho_1 \ran; \\
\op{Hom}_{\tC _2}(|11\ran, |11\ran)=\lan e_{1,2},~ \rho_1,~ \rho_2,~ \rho_1 \cdot \rho_2 ~|~ \rho_1^2=\rho_2^2=0, ~\rho_1 \cdot \rho_2 = \rho_2 \cdot \rho_1 \ran.
\end{gather*}
There is one nilpotent endomorphism of $\mf{x} \in \cal{B}_2'$ associated to each factor $|1\ran$ in $\mf{x}$.
The two nilpotent endomorphisms of $|11\ran$ commutes.
There is one tight contact structure $r \in \op{Hom}_{\tC_2}(|01\ran, |10\ran)$ given by a single bypass attachment as in Figure \ref{1-3}.
Note that $\op{Hom}_{\tC _2}(|10\ran, |01\ran)=0$ which is related to the non-decreasing restriction on diagrams in strands algebras. (See Definition \ref{strand} for more detail.)
\begin{figure}[h]
\begin{overpic}
[scale=0.25]{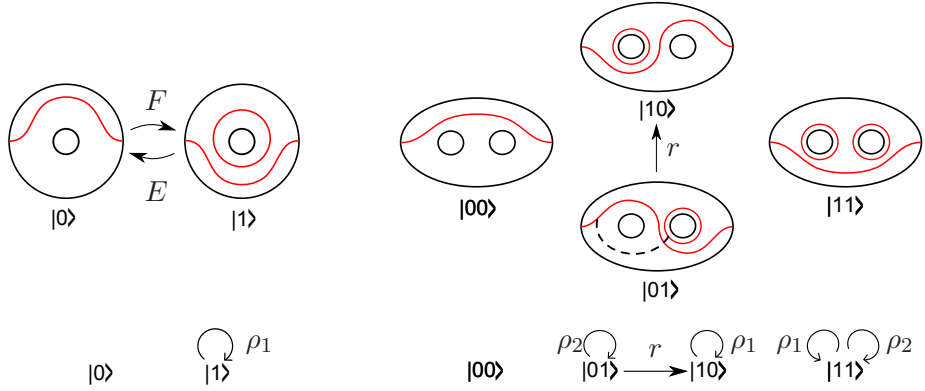}
\put(15,30){$F$}
\put(15,20){$E$}
\put(26,4){$\rho_1$}
\put(60,4){$\rho_2$}
\put(72,25){$r$}
\put(70,3){$r$}
\put(79,4){$\rho_1$}
\put(84,4){$\rho_1$}
\put(96,4){$\rho_2$}
\end{overpic}
\caption{The upper picture on the left describes actions of $E$ and $F$ exchanging $|0\ran$ and $|1\ran$ in $\tC _1$.
The upper picture on the right gives the distinguished collections of objects in $\tC _2$, where the arrow $r$ is a bypass attachment along the dashed line in $\op{Hom}_{\tC_2}(|01\ran, |10\ran)$. The quivers in the lower picture give morphisms in $\tC _1$ and $\tC _2$.}
\label{1-3}
\end{figure}

\subsection{Algebraic formulations}
At this point we pass to algebra\footnote{In fact, the rest of this paper is just algebra which is motivated by the contact category.}.
An algebraic formulation of $\tC _n$ is given as follows.
We construct a quiver $\g_n$ whose vertex set $V(\g_n)$ is the basis $\bn'$ of $V_1^{\ot n}$.
The arrow set $A(\g_n)$ is given by morphisms between the objects of $\tC _n$ in the distinguished collection which lifts the basis $\bn'$.
Consider the path algebra $\F \g_n$ of the quiver $\g_n$ with an additional $t$-grading.
Finally, we define a $t$-graded DG algebra $R_n$ by adding a nontrivial differential on a quotient of $\F \g_n$.
We prove that $R_n$ is formal, i.e., it is quasi-isomorphic to its cohomology $H(R_n)$.
Similarly, we define $t$-graded DG algebras $A$ and $B$ from the contact categories $\tC _o$ and $\tC _{oo}$.

The DG algebra $R_n$ is closely related to the strands algebra for an $n$ times punctured disk.
In general, the strands algebra of any surface with boundary was defined by Zarev \cite{Za}.
Motivated by the {\em rook monoid} \cite{So} and its diagrammatic presentation {\em rook diagrams} \cite{FHH}, we describe $R_n$ in terms of {\em decorated rook diagrams}.
The rook diagram is used to study the Alexander and Jones polynomials by Bigelow-Ramos-Yi \cite{BRY}, and tensor representations of $\mathfrak{gl}(1|1)$ by Benkart-Moon \cite{BM}.

Consider the DG category $DG(\rn)$ of $t$-graded DG $\rn$-modules and its full subcategory $DGP(\rn)$ generated by some distinguished projective DG $\rn$-modules.
We model $\tC _n$ by the $0$th homology category of $DGP(R_n)$ which is denoted by $HP(R_n)$.
Let $HP(H(R_n))$ denote the $0$th homology category of $DGP(H(R_n))$.
Then $HP(R_n)$ and $HP(H(R_n))$ are equivalent as triangulated categories since $R_n$ is formal.
Their Grothendieck groups are isomorphic to free $\zt$-modules over $\bn'$:
$K_0(HP(R_n)) \cong K_0(HP(H(R_n))) \cong \zt\lan \bn' \ran \cong V_1^{\ot n}.$
Similarly, $\tC _{o}$ and $\tC _{oo}$ are modeled by triangulated categories $HP(A)$ and $HP(B)$ whose Grothendieck groups are isomorphic to $\ut$ and $\utt$, respectively.

In order to categorify the $\ut$-action on $V_1^{\ot n}$, we define a DG algebra $A \bt R_n$ by adding a differential to $A \ot R_n$.
Consider the $0$th homology category $HP(A \bt R_n)$ of $DGP(A \bt R_n)$ whose Grothendieck group is isomorphic to a quotient $\ut\ot_{\{T=t^n\}}V_1^{\ot n}$ of $\ut \times V_1^{\ot n}$ by the relation $(T,v)=(I,t^nv)$.
Motivated from stacking of dividing sets in the contact categories, we define a DG $(H(R_n), A \bt R_n)$-bimodule $C_n$ which is the key construction in our categorification.
A functor defined by tensoring with $C_n$ over $A \bt R_n$: $DGP(A \bt R_n) \xra{C_n \ot -} DGP(H(R_n))$ induces an exact functor $\eta_n$ between their $0$th homology categories.
The decategorification $K_0(\eta_n)$ on the Grothendieck groups agrees with the $\ut$-action on $V_1^{\ot n}$: $\ut \ot_{\{T=t^n\}} V_1^{\ot n} \ra V_1^{\ot n}$.
Similarly, we construct functors $\eta: HP(A \ot A) \xra{N \ot -} HP(A)$ by tensoring with a DG $(A, A \ot A)$-bimodule $N$, and $\delta: HP(A) \xra{S \ot -} HP(B)$ by tensoring with a DG $(B, A)$-bimodule $S$.
We show that $\eta$ and $\delta$ categorify the multiplication and comultiplication on $\ut$, respectively.
$$
\xymatrix{
\eta_n: & HP(A \bt R_n) \ar[r]^{C_n \ot -} \ar[d]_{K_{0}} & HP(H(R_n)) \ar[d]_{K_{0}} \\
K_0(\eta_n): & \ut \ot_{\{T=t^n\}} V_1^{\ot n} \ar[r]  & V_1^{\ot n}.
}$$

\noindent {\em The organization of the paper.}
Since many algebraic constructions are quite technical, we will try to give the motivation from contact topology in remarks after the algebraic definitions.
\begin{itemize}
\item In Section 2 we define the Hopf superalgebra $\ut$ and categorify its multiplication via the DG $(A, A \ot A)$-bimodule $N$.
\item In Section 3 we categorify the comultiplication on $\ut$ via the DG $(B, A)$-bimodule $S$.
\item In Section 4 we define the representations $V_1$ and $V_1^{\ot n}$ of $\ut$.
\item In Section 5 we define the {\em decorated rook diagrams} and the $t$-graded DG algebras $R_n$, $A \bt R_n$ and show that they are formal as DG algebras.
\item In Section 6 we define the DG $(H(R_n), A \bt R_n)$-bimodule $C_n$.
\item In Section 7 we give a categorification of the $\ut$ action on $V_1^{\ot n}$.
\end{itemize}

\noindent {\em Acknowledgements.}
I am very grateful to Ko Honda for suggesting this project and introducing me to the contact category.
I would like to thank Aaron Lauda for teaching me a great deal about categorification, especially the diagrammatic approach, Ciprian Manolescu for helpful conversations on the strands algebra, and Stephen Bigelow for explaining the connection between the Alexander polynomial and the rook diagram.
Also thank David Rose and Fan Yang for useful discussions.

\section{$\Ut$ and the categorification of its multiplication}
In Section 2.1 we define the Hopf superalgebra $\ut$.
In Section 2.2 we define the $t$-graded DG algebra $A$ and the triangulated category $HP(A)$ whose Grothendieck group is isomorphic to $\ut$.
In Section 2.3 we define the triangulated category $HP(A \ot A)$ whose Grothendieck group is isomorphic to $\ut \ot_{\zT} \ut$.
In Section 2.4 we construct the $t$-graded DG $(A, A \ot A)$-bimodule $N$.
In Section 2.5 we categorify the multiplication $\ut \ot_{\zT} \ut \ra \ut$ to the exact functor $\eta: HP(A \ot A) \xra{N \ot_{A \ot A} -} HP(A)$.

\subsection{The Hopf superalgebra $\ut$}
\begin{defn} \label{ut}
Define the Hopf superalgebra $\{\ut, m, \op{p}, \Delta, \epsilon, S\}$ over $\Z$ as follows:

\n(1) The multiplication $m$ makes $\ut$ into an associative $\Z$-algebra with unit $I$, generators $E, F, T, T^{-1}$ and relations:
\begin{gather*}
E^2=F^2=0; \quad EF+FE=I-T; \\
ET=TE,~ FT=TF; \quad TT^{-1}=T^{-1}T=I.
\end{gather*}

\n(2) The parity $\op{p}$ is a $\Z$-grading\footnote{In fact, this is a ``categorical parity" rather than the usual parity taking values in $\F$.} defined by: $\op{p}(E)=-1, ~ \op{p}(F)=1, ~ \op{p}(I)=\op{p}(T)=0.$

\n(3) The comultiplication $\Delta: \ut \ra \utt$ is an algebra map defined on the generators by:
\begin{gather*}
\de(E)=E \ot I + I \ot E, \quad \de(F)=F \ot T + I \ot F, \quad \de(T)=T \ot T.
\end{gather*}

\n(3) The counit $\epsilon: \ut \ra \Z$ is an algebra map defined by:
$\epsilon(E)=\epsilon(F)=0, ~ \epsilon(I)=\epsilon(T)=1.$

\n(4) The antipode $S: \ut \ra \ut$ is an anti-homomorphism of superalgebras, i.e., $S(ab)=(-1)^{\op{p}(a)\op{p}(b)}S(b)S(a),$ defined by: $S(T)=T^{-1}, ~ S(E)=-E, ~ S(F)=-FT^{-1}.$
\end{defn}

\begin{rmk}
(1) Since $T$ is a central element, $\ut$ can be viewed as a free $\Z[T^{\pm}]$-module over the basis $\cal{B}=\{F; I, EF; E\}$.

\n (2) The parity $\op{p}$ actually comes from the Euler number of a dividing set.
Recall a dividing set divides the surface into positive and negative regions.
Then the {\em Euler number} is the Euler characteristic of the positive region minus the Euler characteristic of the negative region.

\n (3) The multiplication on $\utt$ is graded: $(a \ot b)\cdot (c \ot d)=(-1)^{\op{p}(b)\op{p}(c)}ac \ot bd.$

\n (4) The counit corresponds to a functor $\tC _{o} \ra \tC(S^2)$ between the contact categories of an annulus $S_{o}$ and a sphere $S^2$ which is given by capping each component of $\bdry S_{o}$ off with a disk.

\n (5) The antipode corresponds to a functor $\tC _{o} \ra \tC _{o}$ given by an inversion about the core of the annulus.
\end{rmk}

\begin{lemma} \label{ut lemma}
The definition above gives a Hopf superalgebra $\{\ut, m, \op{p}, \Delta, \epsilon, S\}$:
\be
\item $\de$ is an algebra map.
\item $\de$ is coassociative: $(\de \ot id)\circ \de=(id \ot \de)\circ \de$.
\item $S$ is an antipode: $m\circ (S \ot id) \circ \de(a)=m\circ (id \ot S) \circ \de(a)=\epsilon(a)I$ for all $a \in \ut$.
\ee
\end{lemma}

\proof
We verify (1) and leave (2) and (3) to the reader.
\begin{align*}
\de(E) \de(E)& =(E \ot I + I \ot E)(E \ot I + I \ot E)  \\
& = E^2 \ot I + I \ot E^2 + (E \ot I)(I \ot E) + (I \ot E)(E \ot I) \\
& = E \ot E - E \ot E = 0.
\end{align*}
Similarly, $\de(F) \de(F)=0.$
\begin{align*}
& \de(E)\de(F)+\de(F)\de(E) \\
= &  (E \ot I + I \ot E)(F \ot T + I \ot F) + (F \ot T + I \ot F)(E \ot I + I \ot E) \\
= & (EF \ot T - F \ot ET + E \ot F + I \ot EF) + (FE \ot T + F \ot TE - E \ot F + I \ot FE) \\
= & (EF+FE) \ot T + I \ot (EF+FE) \\
= & I \ot I - T \ot T =  \de(I-T) = \de(EF+FE). \qed
\end{align*}

\subsection{The $t$-graded DG algebra $A$}
We refer to \cite[Section 10]{BL} for an introduction to DG algebras, DG modules and {\em projective} DG modules, and to \cite{Ke} for an introduction to DG categories and their homology categories.
A $t$-graded DG algebra $R$ is a DG algebra with an additional $t$-grading.
Let $DG(R)$ denote the DG category of $t$-graded DG left $R$-modules.
We refer to \cite{Tian} for more detail.
\begin{defn}
Let $A$ be a $t$-graded DG $\F$-algebra with idempotents $e(\g)$ for $\g \in \cal{B}=\{F; I, EF; E\}$, generators $\rho(I,EF), \rho(EF,I)$ and relations:
\begin{gather*}
e(\g) \cdot e(\g')=\delta_{\g,\g'}e(\g) ~\mbox{for}~ \g,\g' \in \cal{B};\\
e(I) \cdot \rho(I, EF) =\rho(I, EF) \cdot e(EF)=\rho(I, EF);\\
e(EF) \cdot \rho(EF, I) =\rho(EF, I) \cdot e(I)=\rho(EF, I);\\
\rho(I,EF) \cdot \rho(EF,I)=0.
\end{gather*}
The differential on $A$ is trivial.
The grading $\op{deg}=(\op{deg}_h, \op{deg}_t)$ is defined by:
$$\op{deg}(a)= \left\{
\begin{array}{cc}
(1,1) & \mbox{if}~ a=\rho(EF,I), \\
(0,0) & \mbox{otherwise},
\end{array}
\right.$$
where $\deh$ is the cohomological grading and $\dt$ is the $t$-grading.
\end{defn}

\begin{rmk} \label{Qa}
(1) The algebra $A \cong \bigoplus\limits_{\g_1, \g_2 \in \cal{B}}\op{Hom}_{\tC_o}(\g_1, \g_2)$ describes all tight contact structures between the dividing sets in $\cal{B}$.
Each of $e(\g)$ is the $I$-invariant contact structure associated to $\g \in \cal{B}$.
Each of $\rho(I,EF) \in \op{Hom}_{\tC_o}(I, EF)$ and $\rho(EF,I) \in \op{Hom}_{\tC_o}(EF, I)$ is given by a single nontrivial bypass attachment.

\n(2) The stacking $\rho(I,EF)\cdot \rho(EF,I)=0 \in \op{Hom}_{\tC_o}(I, I)$ is a non-tight contact structure.
On the other hand, the stacking $\rho(EF,I)\cdot \rho(I,EF) \in \op{Hom}_{\tC_o}(EF, EF)$ is tight and nonzero.
This nonzero product compared to $\rho(I,EF)\cdot \rho(EF,I)=0$ reflects the differences between $I$ and $EF$ as dividing sets, where $I$ is the identity in categorical actions but $EF$ is not the identity.

\n(3) The algebra $A$ is a quotient of the path algebra $\F Q_A$ of a quiver $Q_A$:
$$\xymatrix{
F; & I \ar@<1ex>[r] & EF; \ar[l]  &E.
}$$
\end{rmk}

There is a decomposition: $A=A_1 \oplus A_0 \oplus A_{-1}$, where $A_1=e(F)Ae(F)$, $A_{-1}=e(E)Ae(E)$, and $A_0=e(I)Ae(I)\oplus e(EF)Ae(EF)$.

\begin{defn} \label{pa}
Define a {\em parity} $\op{p}: A \ra \Z$ on $A$ by $\op{p}(a)=i$ for $a \in A_i$.
Note that $\op{p}$ is not a grading with respect to the multiplication on $A$.
\end{defn}

Consider a collection of projective DG $A$-modules $\{P(\g)=A\cdot e(\g) ~|~ \g \in \cal{B}\}$.
As a left $A$-module, $P(\g)$ is generated by the idempotent $e(\g)$.
To distinguish from $e(\g) \in A$, let $m(\g)$ denote the generator of $P(\g)$.
The grading on $P(\g)$ is inherited from $A$, i.e., $\op{deg}(m(\g))=(0,0)$.

\begin{defn} \label{pg}
Define a {\em parity} $\op{p}: \bigsqcup\limits_{\g \in \cal{B}}P(\g) \ra \Z$ by $\op{p}(m)=\op{p}(\g)$
for all $m \neq 0 \in P(\g)$ and $\g \in \cal{B} \subset \ut$, where $\op{p}(\g)$ is the parity of $\g$ in $\ut$ from Definition \ref{ut}.
\end{defn}

For $a\in A$ and $m \in P(\g)$, $\op{p}(a\cdot m)=\op{p}(m)=\op{p}(a)$ if $m$ and $a \cdot m$ are nonzero.
In particular, $\op{p}$ on $P(\g)$ is not a grading with respect to the $A$ action on $P(\g)$.
The two parities in Definitions \ref{pa} and \ref{pg} will be used to define twisted gradings on $A\ot A$ in Definition \ref{aoa} and on $P(\g_1) \ot P(\g_2)$ in Definition \ref{mom}.

\begin{defn}
Let $DGP(A)$ be the smallest full subcategory of $DG(A)$ which contains the projective DG $A$-modules $\{P(\g)=A\cdot e(\g) ~|~ \g \in \cal{B}\}$ and is closed under the cohomological grading shift functor $[1]$, the $t$-grading shift functor $\{1\}$ and taking mapping cones.
\end{defn}

The $0$th homology category $HP(A)$ of $DGP(A)$ is the homotopy category of $t$-graded DG projective $A$-modules generated by $\{P(\g) ~|~ \g \in \cal{B}\}$.
It is a triangulated category and the Grothendieck group $K_0(HP(A))$ has a $\zT$-basis $\{[P(\g)] ~|~ \g \in \cal{B}\}$, where the multiplication by $T$ is induced by the $t$-grading shift: $[M\{1\}]=T[M] \in K_0(HP(A))$ for $M \in HP(A)$.
\begin{lemma} \label{K0 mod}
There is an isomorphism $K_0(HP(A)) \cong \ut$ of free $\zT$-modules.
\end{lemma}

\subsection{The $t$-graded DG algebra $A\ot A$}
\begin{defn} \label{aoa}
Let $A \ot A$ be the tensor product of two $A$'s over $\F$ as an algebra.
The differential is trivial.
The grading $\op{deg}=(\deh, \dt)$ is defined for generators $a, b \in A$ by:
\begin{align*}
 \dt(a\ot b)&=\dt(a)+\dt(b), \\
 \deh(a \ot b)&= \deh(a) + \deh(b) + 2\dt(a)\op{p}(b),
\end{align*}
\end{defn}

\begin{rmk}
The cohomological grading $\deh$ of $A \ot A$ is the sum of two $\deh$'s twisted by the $t$-grading $\dt$ and the parity $\op{p}$.
The topological meaning of $\deh$ is the framing of links in $S_o \times [0,1]$.
In general, $\deh(a\ot b)\neq \deh(a)+\deh(b)$ when two links are nontrivially linked.
\end{rmk}

\begin{lemma}
The grading $\op{deg}$ is well-defined: $\op{deg}(ac \ot bd)=\op{deg}(a\ot b)+\op{deg}(c \ot d)$.
\end{lemma}
\begin{proof}
If $ac \ot bd\neq 0$ for generators $a,b,c,d \in A$, then $\op{p}(b)=\op{p}(d)=\op{p}(bd)$ and $\dt(ac)=\dt(a)+\dt(c)$.
By definition
\begin{align*}
\deh(ac \ot bd)=&\deh(ac)+\deh(bd)+2\dt(ac)\op{p}(bd) \\
=&\deh(a)+\deh(c)+\deh(b)+\deh(d)+2\dt(a)\op{p}(b)+2\dt(c)\op{p}(d)\\
=&\deh(a\ot b)+\deh(c \ot d).
\end{align*}
The equation for the $t$-component is obvious.
\end{proof}

\begin{defn}
Let $DGP(A\ot A)$ be the smallest full subcategory of $DG(A \ot A)$ which contains the projective DG $A \ot A$-modules $\{P(\g, \g')=(A \ot A)\cdot (e(\g)\ot e(\g')) ~|~ \g, \g' \in \cal{B}\}$ and is closed under the cohomological grading shift functor $[1]$, the $t$-grading shift functor $\{1\}$ and taking mapping cones.
\end{defn}

The $0$th homology category $HP(A \ot A)$ of $DGP(A \ot A)$ is the homotopy category of $t$-graded DG projective $A \ot A$-modules generated by $\{P(\g, \g') ~|~ \g, \g' \in \cal{B}\}$.

\begin{defn} \label{mom}
Define a tensor product functor
$$
\begin{array}{cccccc}
 \chi: & HP(A) & \times & HP(A) & \ra & HP(\aoa) \\
       &   M_1    &     ,   &   M_2    & \mapsto & M_1 \ot_{\F} M_2,
\end{array}
$$
where the grading on $M_1 \ot M_2$ is given for $m_1 \in M_1$ and $m_2 \in M_2$ by:
\begin{align*}
 \dt(m_1\ot m_2)&=\dt(m_1)+\dt(m_2), \\
 \deh(m_1 \ot m_2)&= \deh(m_1) + \deh(m_2) + 2\dt(m_1)\op{p}(m_2),
\end{align*}
\end{defn}

\begin{rmk}\label{grading}
The grading on $M_1 \ot M_2$ is compatible with the grading on $A \ot A$.
We have
$$\chi(P(\g)\{n\}, P(\g')\{n'\})= P(\g, \g')\{n+n'\}[2n\op{p}(\g')].$$
Note that the twisting $[2n\op{p}(\g')]$ cannot be seen on the level of Grothendieck groups.
\end{rmk}

Since $K_0(HP(\aoa))$ has a $\zT$-basis $\cal{B} \times \cal{B}$, we have the following:
\begin{lemma}
There is an isomorphism $K_0(HP(\aoa)) \cong \ut \ot_{\Z[T^{\pm}]} \ut$ of free $\zT$-modules.
Moreover, the functor $\chi$ induces a tensor product on their Grothendieck groups:
$$K_0(\chi): \ut \times \ut \xra{\ot_{\Z[T^{\pm}]}} \ut \ot_{\Z[T^{\pm}]} \ut.$$
\end{lemma}

\subsection{The $t$-graded DG $(A, A\ot A)$-bimodule $N$}
To define a functor $\eta: DGP(\aoa) \ra DGP(A)$ lifting the multiplication on $\ut$, we construct a DG $(A, A \ot A)$-bimodule $N$ in two steps: a left DG $A$-module $N$ in Section 2.4.1 and a right DG $\aoa$-module $N$ in Section 2.4.2.

In practice, the functor $\eta$ is obtained by ``reverse-engineering": we have all the essential information about $\eta$ from the contact topology and we construct the bimodule $N$ to realize the functor.
More precisely, we first figure out the behavior of $\eta$ on the objects $P(\g_1,\g_2)$ and set
$$N=\bigoplus\limits_{\g_1,\g_2\in \cal{B}}N(\g_1,\g_2)=\bigoplus\limits_{\g_1,\g_2\in \cal{B}}\eta(P(\g_1,\g_2)) \in DGP(A),$$
as left DG $A$-modules.
We then determine the right $A\ot A$-module structure on $N$ by considering morphism sets $\op{Hom}(P(\g_1,\g_2), P(\g_1',\g_2'))$ in $DGP(\aoa)$.
For instance, the right multiplication by $e(F)\ot \rho(I,EF)$ in $\aoa$ defines a morphism:
$$
\begin{array}{cccc}
f: & P(F, I) & \ra & P(F, EF) \\
         & m & \mapsto & m \cdot (e(F) \ot \rho(I,EF)).
\end{array}
$$
Then the right multiplication on $N$ by $e(F) \ot \rho(I,EF)$ is given by the morphism $$\eta(f): \eta(P(F,I)) \ra \eta(P(F,EF))$$ in $DGP(A)$, where $\eta(P(F,I))$ and $\eta(P(F,EF))$ are viewed as left $A$-submodules of $N$.
This technique will be used to construct various bimodules in the paper.

\subsubsection{The left DG $A$-module $N$}
\begin{defn} \label{Def T}
Define a left DG $A$-module
$$N=\bigoplus \limits_{\g_1, \g_2 \in \cal{B}} N(\g_1, \g_2),$$
where $(N(\g_1, \g_2), d(\g_1, \g_2)) \in DGP(A)$ is defined on a case-by-case basis as follows:
\begin{gather*}
N(E, E)=N(E, EF)=N(F, F)=N(EF, F)=0, \\
N(I, \g)=N(\g, I)=P(\g) ~\mbox{for all}~ \g \in \cal{B}, \\
N(E, F)=P(EF), \\
N(F, EF)=P(F) \oplus P(F)\{1\}[1],\\
N(EF, E)=P(E) \oplus P(E)\{1\}[-1], \\
N(EF, EF)=P(EF) \oplus P(EF)\{1\}[1], \\
N(F, E)=P(I)\oplus P(EF)[-1] \oplus P(I)\{1\}[-1].
\end{gather*}
$d(\g_1, \g_2)=0$ for all $(\g_1, \g_2)\neq (F, E)$ and $d(F,E)$ is a map of left $A$-modules defined on generators of $N(F, E)$ by
$$
\begin{array}{cccc}
d(F, E): & N(F, E) & \ra & N(F, E) \\
         & m_{F,E}(I) & \mapsto & \rho(I ,EF)\cdot m_{F,E}(EF), \\
         & m_{F,E}(EF) & \mapsto & \rho(EF, I)\cdot m_{F,E}'(I), \\
         & m_{F,E}'(I) & \mapsto & 0,
\end{array}
$$
where $m_{F,E}(I) \in P(I), m_{F,E}(EF) \in P(EF)[-1]$ and $m'_{F,E}(I) \in P(I)\{1\}[-1]$.
\end{defn}

\begin{rmk} \label{Def T rmk}
(1) The left DG $A$-module $N(\g_1, \g_2)$ is supposed to be the categorical multiplication of two DG $A$-modules $P(\g_1)$ and $P(\g_2)$.
In particular, the class $[N(\g_1, \g_2)] \in \ut$ is the multiplication $\g_1 \cdot \g_2 \in \ut$ under the isomorphism in Lemma \ref{K0 mod}.

\n (2) In the contact category $\tC _0$, the stacking $EF \cdot E$ of dividing sets is the union of $E$ and a pair of loops.
The pair of loops corresponds to tensoring with $\Z^2$ up to grading.
Correspondingly, the left $A$-module $N(EF,E)$ is a direct sum of two $P(E)$'s.

\n (3) The definition of $N(F,E)$ is motivated from the two distinguished triangles in $\tC _0$:
$I \ra EF \ra K^{-1}$ and $K^{-1} \ra I \ra FE$ as in Figure \ref{2-1}, where the gradings are ignored.
The isomorphisms $FE \cong (K^{-1} \ra I)$ and $K^{-1} \cong (I \ra EF)$ give the isomorphism $FE \cong (I \ra EF \ra I)$ as in the definition of $N(F,E)$.
\end{rmk}

\begin{lemma}
$(N(F, E), d(F, E))$ is a $t$-graded DG $A$-module.
\end{lemma}
\proof
It suffices to prove that $d=d(F,E)$ is of degree $(1, 0)$ such that $d^2=0$.
We verify that
\begin{align*}
d^2(m_{F,E}(I))&=d(\rho(I ,EF)\cdot m_{F,E}(EF))=\rho(I ,EF)\cdot \rho(EF, I)\cdot m_{F,E}'(I)=0.
\end{align*}
The degrees of the generators of $N(F,E)$ are as follows:
$$\op{deg}(m_{F,E}(I))=(0,0), \quad \op{deg}(m_{F,E}(EF))=(1,0), \quad \op{deg}(m_{F,E}'(I))=(1,-1).$$
Hence, the differential $d$ is of degree $(1,0)$:
\begin{gather*}
\op{deg}(d(m_{F,E}(I)))=\op{deg}(\rho(I ,EF))+\op{deg}(m_{F,E}(EF))=(1,0)=\op{deg}(m_{F,E}(I))+(1,0); \\
\op{deg}(d(m_{F,E}(EF)))=\op{deg}(\rho(EF, I))+\op{deg}(m_{F,E}'(I))=(2,0)=\op{deg}(m_{F,E}(EF))+(1,0). \qed
\end{gather*}

\subsubsection{The right $\aoa$-module structure on $N$}
In this subsection we describe the right $\aoa$-module structure on $N$.
Let $m \times (a \ot b)$ denote the right multiplication for $m \in N, a \ot b \in \aoa$ and let $m \cdot a$ denote the multiplication in $A$ for $m \in P(\g) \subset A, a \in A$.

We fix the notation for the generators of $N(\g_1, \g_2)$:
$$
\begin{array}{c}
m_{\g,I}(\g)\in P(\g)=N(\g, I) ~\mbox{for all}~ \g \in \cal{B}; \\
m_{I,\g}(\g)\in P(\g)=N(I, \g) ~\mbox{for all}~ \g \in \cal{B}; \\
m_{E,F}(EF)\in P(EF)=N(E, F); \\
m_{EF,E}(E)\in P(E), \quad m_{EF,E}'(E)\in P(E)\{1\}[-1] \quad \mbox{in}~ N(EF, E); \\
m_{F,EF}(F)\in P(F), \quad m_{F,EF}'(F)\in P(F)\{1\}[1] \quad \mbox{in}~ N(F, EF); \\
m_{EF,EF}(EF)\in P(EF), \quad m_{EF,EF}'(EF)\in P(EF)\{1\}[1] \quad \mbox{in}~ N(EF, EF);\\
m_{F,E}(I) \in P(I), \quad m_{F,E}(EF) \in P(EF)[-1], \quad m'_{F,E}(I) \in P(I)\{1\}[-1] \quad \mbox{in}~ N(F,E).
\end{array}
$$

We define right multiplications by the generators of $\aoa$ on a case-by-case basis.
Each right multiplication is a map of left $A$-modules defined on the generators of $N$ as follows:

\n (1) For an idempotent $e(\g_1) \ot e(\g_2)$, define
$$
\ba{cccc}
\times (e(\g_1) \ot e(\g_2)): & N(\g_1',\g_2') &  \ra & N(\g_1',\g_2') \\
                        &     m      & \mapsto & \delta_{\g_1,\g_1'} \delta_{\g_2,\g_2'} ~ m
\ea
$$

\vspace{0.2cm}
\n (2) For generators $\rho(EF, I) \ot e(E)$ and $\rho(I, EF) \ot e(E)$, define
$$
\ba{cccc}
\times (\rho(EF,I) \ot e(E)): & N(EF, E) & \ra & N(I, E) \\
       & m_{EF,E}(E)      & \mapsto  & 0 \\
       & m_{EF,E}'(E)  & \mapsto  & m_{I,E}(E),
\ea
$$
$$
\ba{cccc}
\times  (\rho(I,EF) \ot e(E)): & N(I, E)& \ra & N(EF, E) \\
       & m_{I,E}(E)    & \mapsto  &  m_{EF,E}(E).
\ea
$$

\vspace{0.2cm}
\n (3) For generators $e(F) \ot \rho(EF, I)$ and $e(F) \ot \rho(I, EF)$, define
$$
\ba{cccc}
\times (e(F) \ot \rho(EF,I)): & N(F, EF) & \ra & N(F, I) \\
       & m_{F,EF}(F)      & \mapsto  & 0 \\
       & m_{F,EF}'(F)  & \mapsto  & m_{F,I}(F),
\ea
$$
$$
\ba{cccc}
\times (e(F) \ot \rho(I,EF)): & N(F,I) & \ra & N(F,EF) \\
       & m_{F,I}(F)  & \mapsto  & m_{F,EF}(F).
\ea
$$

\vspace{0.2cm}
\n (4) For generators $e(I) \ot a$ and $a \ot e(I)$, where $a \in \{\rho(I, EF), \rho(EF, I)\}$, define
$$
\ba{cccc}
\times (\rho(EF,I) \ot e(I)): & N(EF,I) & \ra & N(I, I) \\
       & m_{EF,I}(EF)  & \mapsto  & \rho(EF,I) \cdot m_{I,I}(I),
\ea
$$
$$
\ba{cccc}
\times (\rho(I,EF) \ot e(I)): & N(I,I) & \ra & N(EF, I) \\
       & m_{I,I}(I)  & \mapsto  & \rho(I,EF) \cdot m_{EF,I}(EF),
\ea
$$$$
\ba{cccc}
\times (e(I) \ot \rho(EF,I)): & N(I,EF) & \ra & N(I, I) \\
       & m_{I,EF}(EF)  & \mapsto  & \rho(EF,I) \cdot m_{I,I}(I),
\ea
$$$$
\ba{cccc}
\times (e(I) \ot \rho(I,EF)): & N(I,I) & \ra & N(I,EF) \\
       & m_{I,I}(I)  & \mapsto  & \rho(I,EF) \cdot m_{I,EF}(EF).
\ea
$$

\vspace{0.2cm}
\n (5) For generators $e(EF) \ot a$ and $a \ot e(EF)$, where $a \in \{\rho(I, EF), \rho(EF, I)\}$, define
$$
\ba{cccc}
\times (\rho(EF,I) \ot e(EF)): & N(EF, EF) & \ra & N(I, EF) \\
       & m_{EF,EF}(EF)      & \mapsto  & 0 \\
       & m_{EF,EF}'(EF)  & \mapsto  & m_{I,EF}(EF),
\ea
$$
$$
\ba{cccc}
\times (e(EF) \ot \rho(EF,I)): & N(EF, EF) & \ra & N(EF, I) \\
       & m_{EF,EF}(EF)      & \mapsto  & 0 \\
       & m_{EF,EF}'(EF)  & \mapsto  & m_{EF,I}(EF),
\ea
$$
$$
\ba{cccc}
\times  (\rho(I,EF) \ot e(EF)): & N(I, EF)& \ra & N(EF, EF) \\
       & m_{I,EF}(EF)    & \mapsto  &  m_{EF,EF}(EF),
\ea
$$
$$
\ba{cccc}
\times (e(EF) \ot \rho(I,EF)): & N(EF,I) & \ra & N(EF,EF) \\
       & m_{EF,I}(EF)  & \mapsto  & m_{EF,EF}(EF).
\ea
$$

\vspace{0.2cm}
\n (6) For generators $e(E) \ot \rho(EF, I), e(E) \ot \rho(I, EF), \rho(EF, I) \ot e(F)$ and $\rho(I, EF) \ot e(F)$, define the right multiplication to be the zero map since the corresponding domains or ranges are trivial from Definition \ref{Def T}.

\vspace{0.2cm}
This concludes the definition of the right multiplications by the generators of $\aoa$.
In general, define $m \times (r_1\cdot r_2):=(m \times r_1) \times r_2$ for $r_1 ,r_2 \in \aoa$ and $m \in N$.

\begin{rmk}
(1) The definition is motivated from studying tight contact structures on the gluing of two $\tC_o$'s in Figure \ref{1-1}.
For instance, the right multiplication with $\rho(EF,I) \ot e(E)$ is determined by the corresponding tight contact structure in $\op{Hom}_{\tC_o}(EF\cdot E, I\cdot E)$.

\n(2) The right multiplication with $r$ on $N(F,E)$ is nonzero only if $r=e(F) \ot e(E)$.
In that case the right multiplication is the identity from Case (1).
\end{rmk}

\begin{lemma}
$N$ is a $t$-graded right $\aoa$-module:
\be
\item $(m \times r_1) \times r_2=(m \times r_1') \times r_2'$, for $r_1\cdot r_2=r_1'\cdot r_2' \in \aoa$.
\item $\op{deg}(m \times r)=\op{deg}(m)+\op{deg}(r).$
\ee
\end{lemma}
\proof
We verify (1) for $r_1=\rho(I,EF) \ot e(I), r_2=\rho(EF,I) \ot e(I)$ and $m=m_{I,I}(I)$ in Case (4).
It suffices to show that $(m \times r_1) \times r_2=0$ since $r_1\cdot r_2=0$. We have
\begin{align*}
(m \times r_1) \times r_2= \rho(I,EF) \cdot m_{EF,I}(EF) \times r_2 = \rho(I,EF) \cdot \rho(EF,I) \cdot m_{I,I}(I) = 0,
\end{align*}
since $\rho(I,EF) \cdot \rho(EF,I)=0 \in A$.

The proofs for other cases are similar and we leave them to the reader.

\vspace{.2cm}
\n For (2), the only nontrivial case is $m\times r=m_{EF,E}'(E) \times (\rho(EF,I)\ot e(E))=m_{I,E}(E)$ in Case (2), where the gradings are given in the following:
$$\op{deg}(m_{EF,E}'(E))+\op{deg}(\rho(EF,I)\ot e(E))=(1,-1)+(-1,1)=(0,0)=\op{deg}(m_{I,E}(E)). \qed$$

Since the right multiplications are the maps of left $A$-modules, we have
$a \cdot (m \times r)=(a \cdot m) \times r,$
for $a \in A, r \in \aoa$ and $m \in N$.
Hence $N$ is a $t$-graded DG $(A, \aoa)$-bimodule.

\subsection{The categorification of the multiplication on $\ut$}
In this section, we categorify the multiplication on $\ut$, i.e., prove Theorem \ref{thm-ut}.
Let $\eta: DGP(\aoa) \xra{N \otimes_{\aoa} -} DGP(A)$ be a functor given by tensoring with the DG $(A, \aoa)$-bimodule $N$ over $\aoa$.

\begin{lemma} \label{eta}
The functor $\eta$ maps $P(\Gamma_1,\Gamma_2)$ to $N(\Gamma_1,\Gamma_2) \in DGP(A)$ for all $\Gamma_1,\Gamma_2 \in \mathcal{B}$.
\end{lemma}
\proof
Since $N = \bigoplus\limits_{\g_1', \g_2' \in \cal{B}} N(\g_1', \g_2')$ as left DG $A$-modules, $N \otimes P(\g_1, \g_2)$ is the quotient of $\bigoplus\limits_{\g_1', \g_2' \in \cal{B}} (N(\g_1', \g_2') \times P(\g_1, \g_2))$ by the relations
$$\{(m \times r , ~e(\g_1)\ot e(\g_2))=(m , ~r \cdot (e(\g_1) \ot e(\g_2)))~|~ m \in N(\g_1', \g_2'), r \in \aoa \}.$$
Since $\{(m ,~ r \cdot (e(\g_1) \ot e(\g_2)))~|~ m \in N(\g_1', \g_2'),~ r \cdot e(\g_1) \ot e(\g_2) \neq 0\}$ spans $N(\g_1', \g_2') \times P(\g_1, \g_2)$, $N \otimes P(\g_1, \g_2)$ is spanned by
$$\{(m \times r ,~ e(\g_1) \ot e(\g_2))~|~ m \in N(\g_1', \g_2'),~ r \cdot (e(\g_1) \ot e(\g_2)) \neq 0 \} \cong N(\g_1, \g_2) \in DGP(A).\qed$$

There is an induced exact functor $\eta: HP(\aoa) \xra{N \otimes_{\aoa} -} HP(A)$ between the $0$th homology categories.
Let $\cal{M}=\eta \circ \chi$ be the composition: $\cal{M}: HP(A) \times HP(A) \xra{\chi} HP(\aoa) \xra{\eta} HP(A).$

\proof [Proof of Theorem \ref{thm-ut}]
We compute the multiplication
$$K_0(\cal{M}): K_0(HP(A)) \times K_0(HP(A)) \ra K_0(HP(A)).$$

\n (1) By Lemma \ref{eta}, $\cal{M}(P(\g), P(\g'))=\eta(P(\g,\g'))=N(\g, \g'),$ for $\g, \g' \in \cal{B}$.
Its class $[N(\g, \g')]$ agrees with $\g \cdot \g' \in \ut$ by Remark \ref{Def T rmk}.

\vspace{0.2cm}
\n (2) The class $[P(I)]$ is a unit of $K_0(HP(A))$, since $P(I)$ is a unit under $\cal{M}$:
\begin{gather*}
\cal{M}(P(\g), P(I))=\eta(P(\g,I))=N(\g, I)=P(\g), \\
\cal{M}(P(I), P(\g))=\eta(P(I,\g))=N(I,\g)=P(\g).
\end{gather*}

\vspace{0.2cm}
\n (3) By Remark \ref{grading}, $\cal{M}(P(\g), P(I)\{1\})=\eta(P(\g,I)\{1\})=P(\g)\{1\};$
\begin{gather*}
\cal{M}(P(I)\{1\}, P(\g))=\eta(P(I,\g)\{1\}[2 \op{p}(\g)])=P(\g)\{1\}[2 \op{p}(\g)].
\end{gather*}
Although $\cal{M}(P(\g), P(I)\{1\})$ and $\cal{M}(P(I)\{1\}, P(\g))$ differ by $2 \op{p}(\g)$ in their cohomological gradings, their classes agree in $K_0(HP(A))$.
Hence $[P(I)\{1\}]$ corresponds to the variable $T$ in the $\zT$-algebra $K_0(HP(A))$.

\vspace{0.2cm}
(1), (2) and (3) together imply that the following map is an isomorphism of $\zT$-algebras:
$$
\begin{array}{ccc}
\ut & \ra & K_0(HP(A)) \\
\g & \mapsto & [P(\g)] , \\
T & \mapsto & \qquad [P(I)\{1\}]. \qed
\end{array}
$$

\begin{rmk} \label{rmk monoidal}
It is natural to ask whether $\cal{M}$ is a monoidal functor, i.e., the following diagram commutes up to equivalence:
$$
\xymatrix{
HP(A) \times HP(A) \times HP(A) \ar[r]^-{id \times \cal{M}} \ar[d]^{\cal{M} \times id} & HP(A) \times HP(A) \ar[d]^{\cal{M}} \\
HP(A) \times HP(A) \ar[r]^{\cal{M}} & HP(A).
}$$
We believe that the answer is positive and it could be done by verifying some associativity relation on various DG bimodules.
\end{rmk}

\section{The categorification of the comultiplication on $\Ut$}
To categorify the comultiplication $\de: \ut \ra \utt$, we define the $(t_1,t_2)$-graded DG algebra $B$ and the triangulated category $HP(B)$ whose Grothendieck group is isomorphic to $\utt$.
Then we construct the $(t_1,t_2)$-graded DG $(B, A)$-bimodule $S$ to give an exact functor $\delta: HP(A) \xra{S \ot_{A} -} HP(B)$.
The decategorification $K_0(\delta)$ agrees with the comultiplication on $\ut$.

\subsection{The $(t_1,t_2)$-graded DG algebra $B$}
We define the algebra $B$ via a quiver $Q_B$.
\begin{defn} [Quiver $Q_B=(V(Q_B),A(Q_B))$]
Let $V(Q_B)=\cal{B} \times \cal{B}$ be the set of vertices.
Let $\g_1 \ot \g_2$ denote a vertex of $Q_B$ for $\g_1, \g_2 \in \cal{B}$.
Let $A(Q_B)$ be the set of arrows given as follows:
$$\xymatrix{
E\ot EF \ar@<1ex>[d]       & EF\ot E \ar@<1ex>[d]        & &  E\ot F \ar@[red][r]& I\ot EF \ar@<1ex>[r] \ar@<1ex>[d] & I\ot I \ar@<1ex>[l] \ar@<1ex>[d] \\
E\ot I \ar@<1ex>[u] \ar@[red][r] & I\ot E \ar@<1ex>[u]          & &               & EF\ot EF \ar@<1ex>[u] \ar@<1ex>[r] & EF\ot I \ar@<1ex>[l] \ar@<1ex>[u] \ar@[red][d] \\
EF\ot F \ar@<1ex>[d]       & F\ot EF \ar@<1ex>[d]           & &               &                                    & F\ot E \\
I\ot F \ar@<1ex>[u] \ar@[red][r] & F\ot I \ar@<1ex>[u]
}$$
\end{defn}
%\begin{align*}
%A(Q_{B,-1})=\{& E \ot I \ra E \ot EF, ~E\ot EF \ra E\ot I, ~E\ot I \ra I \ot E,~I \ot E \ra EF \ot E,\\
%& EF \ot E \ra I \ot E\}; \\
%A(Q_{B,0})=\{& E \ot F \ra I \ot EF, ~I\ot EF \ra I\ot I, ~I\ot I \ra I \ot EF, ~I \ot EF \ra EF \ot EF,\\
%& EF \ot EF \ra EF \ot I, ~EF \ot I \ra I \ot I, ~I \ot I \ra EF \ot I, ~EF\ot I \ra F \ot E, \\
%&  EF \ot I \ra EF \ot EF, ~EF \ot EF \ra I \ot EF\}; \\
%A(Q_{B,1})=\{& I \ot F \ra EF \ot F, ~EF\ot F \ra I\ot F, ~I\ot F \ra F \ot I,~F \ot I \ra F \ot EF,\\
%& F \ot EF \ra F \ot I\}.
%\end{align*}

\begin{rmk}
(1) The quiver has $5$ components $Q_{B}=\bigsqcup\limits_{i=-2}^{2}Q_{B,i}$ where a vertex $\g_1 \ot \g_2$ is in $Q_{B,i}$ if $\op{p}(\g_1)+\op{p}(\g_2)=i$.
The diagrams on the left are the components $Q_{B,-1}$ and $Q_{B,1}$.
The diagram on the right is the component $Q_{B,0}$.
There are no arrows in $Q_{B,-2}$ and $Q_{B,2}$.

\n(2) The arrows give all tight contact structures between the corresponding dividing sets in $S_{oo}\times I$.
Most of the arrows are inherited from $Q_A \times Q_A$, where $Q_A$ is the quiver for the algebra $A$ in Remark \ref{Qa} (3). The extra $4$ arrows in red are given by some tight contact structures which do not exist on $(S_o \sqcup S_o) \times I$. For instance, see Figure \ref{1-2} for the bypass attachment $E\ot I \ra I\ot E$.
\end{rmk}

We define the $(t_1,t_2)$-graded algebra $B=\bigoplus\limits_{i=-2}^{2}B_i$, where $B_i$ is a quotient of the path algebra $\F Q_{B,i}$ of the component $Q_{B,i}$.
\begin{defn}
$B$ is an associative $(t_1,t_2)$-graded $\F$-algebra with a trivial differential and a grading $\op{deg}=(\deh; \dta, \dtb) \in \Z^3$.

\n (1) $B$ has idempotents $e(\g_1 \ot \g_2)$ for all vertices $\g_1 \ot \g_2 \in \cal{B} \times \cal{B}$, generators $\rho(\g_1 \ot \g_2,\g_1' \ot \g_2')$ for all arrows $\g_1 \ot \g_2 \ra \g_1' \ot \g_2'$ in $Q_B$.
The relations consists of $4$ groups:

\n (i) idempotents:
\begin{gather*}
e(\g_1 \ot \g_2) \cdot e(\g_1' \ot \g_2')=\delta_{\g_1,\g_1'}\delta_{\g_2,\g_2'}~e(\g_1 \ot \g_2) ~\mbox{for}~ \g_1,\g_2,\g_1',\g_2' \in \cal{B},\\
e(\g_1 \ot \g_2) \cdot \rho(\g_1 \ot \g_2,\g_1' \ot \g_2') =\rho(\g_1 \ot \g_2,\g_1' \ot \g_2') \cdot e(\g_1' \ot \g_2')=\rho(\g_1 \ot \g_2,\g_1' \ot \g_2');
\end{gather*}

\vspace{.1cm}
\n (ii) relations in $B_{-1}$ of $2$ groups:

(A) relations from the algebra $A$:
\begin{gather*}
\rho(E\ot I,E\ot EF)\cdot \rho(E\ot EF, E\ot I)=0,\\
\rho(I\ot E,EF\ot E)\cdot \rho(EF\ot E, I\ot E)=0;
\end{gather*}

(B) relations for $\rho(E\ot I, I\ot E)$:
\begin{gather*}
\rho(E\ot EF,E\ot I)\cdot \rho(E\ot I, I\ot E)=0,\\
\rho(E\ot I,I\ot E)\cdot \rho(I\ot E, EF\ot E)=0;
\end{gather*}

\vspace{.1cm}
\n (iii) relations in $B_0$ of $3$ groups:

(A) relations from the algebra $A$:
\begin{gather*}
\rho(I\ot I,I\ot EF)\cdot \rho(I\ot EF, I\ot I)=0,\\
\rho(I\ot I,EF\ot I)\cdot \rho(EF\ot I, I\ot I)=0,\\
\rho(I\ot EF,EF\ot EF)\cdot \rho(EF\ot EF, I\ot EF)=0,\\
\rho(EF\ot I,EF\ot EF)\cdot \rho(EF\ot EF, EF\ot I)=0;
\end{gather*}

(B) commutativity relations:
\begin{gather*}
\rho(I\ot I,I\ot EF)\cdot \rho(I\ot EF, EF\ot EF)=\rho(I\ot I,EF\ot I)\cdot \rho(EF\ot I, EF\ot EF),\\
\rho(I\ot EF,I\ot I)\cdot \rho(I\ot I, EF\ot I)=\rho(I\ot EF,EF\ot EF)\cdot \rho(EF\ot EF, EF\ot I),\\
\rho(EF\ot I,I\ot I)\cdot \rho(I\ot I, I\ot EF)=\rho(EF\ot I,EF\ot EF)\cdot \rho(EF\ot EF, I\ot EF),\\
\rho(EF\ot EF,I\ot EF)\cdot \rho(I\ot EF, I\ot I)=\rho(EF\ot EF,EF\ot I)\cdot \rho(EF\ot I, I\ot I);
\end{gather*}

(C) relations for $E\ot F$ and $F\ot E$:
\begin{gather*}
\rho(E\ot F,I\ot EF)\cdot \rho(I\ot EF, EF\ot EF)=0,\\
\rho(EF\ot EF,EF\ot I)\cdot \rho(EF\ot I, F\ot E)=0;
\end{gather*}

\vspace{.1cm}
\n (iv) relations in $B_1$ of $2$ groups:

(A) relations from the algebra $A$:
\begin{gather*}
\rho(I\ot F,EF\ot F)\cdot \rho(EF\ot F, I\ot F)=0,\\
\rho(F\ot I,F\ot EF)\cdot \rho(F\ot EF, F\ot I)=0;
\end{gather*}

(B) relations for $\rho(I\ot F, F\ot I)$:
\begin{gather*}
\rho(EF\ot F,I\ot F)\cdot \rho(I\ot F, F\ot I)=0,\\
\rho(I\ot F,F\ot I)\cdot \rho(F\ot I, F\ot EF)=0.
\end{gather*}

\vspace{.2cm}
\n (2) The grading $\op{deg}=(\deh; \dta, \dtb)$ is defined on the generators by:
$$\op{deg}(a)= \left\{
\begin{array}{cl}
(1;0,0) & \mbox{if}~ a=\rho(E\ot I,I\ot E), ~\rho(E\ot F,I\ot EF),\\
(1;1,0) & \mbox{if}~ a=\rho(EF\ot \g,I\ot \g) ~\mbox{for all}~\g \in \cal{B}, \\
(1;0,1) & \mbox{if}~ a=\rho(I \ot F, F\ot I), ~\rho(\g \ot EF, \g \ot I) ~\mbox{for all}~\g \in \cal{B}, \\
(0;0,0) & \mbox{otherwise},
\end{array}
\right.$$
where $\deh$ is the cohomological grading and $(\dta,\dtb)$ is the $(t_1,t_2)$-grading.
\end{defn}

\begin{rmk}
(1) Relations (ii-A),(iii-A) and (iv-A) come from the relation $\rho(I,EF)\cdot \rho(EF,I)=0$ in $A$.
Relations (iii-B) come from certain isotopies of tight contact structures.
Other relations come from the fact that stackings of the corresponding contact structures are not tight.

\vspace{.1cm}
\n (2) The generators in $A \ot A$ can be viewed as generators in $B$: $\rho(I,EF) \ot e(I) \in A \ot A$ corresponds to $\rho(I \ot I, EF \ot I) \in B$ for instance. But the gradings on $A \ot A$ and $B$ are quite different.

\vspace{.1cm}
\n (3) The algebra $B$ is actually the homology of the strands algebra for a specific handle decomposition of a twice punctured disk.
We refer to Section 5.1 for more detail.
\end{rmk}

\begin{defn}
Let $DGP(B)$ be the smallest full subcategory of $DG(B)$ which contains the projective DG $B$-modules $\{P(\g_1 \ot \g_2)=B\cdot e(\g_1 \ot \g_2) ~|~ \g_1,\g_2 \in \cal{B}\}$ and is closed under the cohomological grading shift functor $[1]$, two $(t_1,t_2)$-grading shift functors $\{1,0\}$ and $\{0,1\}$, and taking mapping cones.
\end{defn}

Let $m(\g_1 \ot \g_2) \in P(\g_1 \ot \g_2)$ denote the generator with $\op{deg}(m(\g_1 \ot \g_2))=(0;0,0)$.
The $0$th homology category $HP(B)$ of $DGP(B)$ is a triangulated category and the Grothendieck group $K_0(HP(B))$ has a $\ztt$-basis: $\{P(\g_1 \ot \g_2) ~|~ \g_1,\g_2 \in \cal{B}\} \cong \cal{B} \times \cal{B},$
where the multiplication by $T_1$ and $T_2$ are induced by the $(t_1,t_2)$-grading shifts:
$$[M\{1,0\}]=T_1[M], \quad [M\{0,1\}]=T_2[M] \in K_0(HP(B)), \quad \mbox{for}~~M \in HP(B).$$
\begin{lemma} \label{K0 mod B}
There is an isomorphism of free $\ztt$-modules: $K_0(HP(B)) \cong \utt,$ where $T_1$ and $T_2$ act on $\utt$ by multiplying $T\ot I$ and $I\ot T$, respectively.
\end{lemma}

\subsection{The $(t_1,t_2)$-graded DG $(B,A)$-bimodule $S$}
To define a functor $\delta: DGP(A) \ra DGP(B)$ lifting the comultiplication on $\ut$, we construct a $(t_1,t_2)$-graded DG $(B, A)$-bimodule $S$ in two steps: a left DG $B$-module $S$ in Section 3.2.1 and a right $A$-module $S$ in Section 3.2.2.

\subsubsection{The left $B$-module $S$}
\begin{defn} \label{Def S}
Define a $(t_1,t_2)$-graded left DG $B$-module
$S=\bigoplus \limits_{\g \in \cal{B}} S(\g),$
where $(S(\g), d(\g))$ in $DGP(B)$ is defined on a case-by-case basis as follows:

\vspace{.1cm}
\n (1) $S(I)=P(I \ot I); d(I)=0$.

\vspace{.1cm}
\n (2) $S(E)=P(E \ot I) \oplus P(I \ot E)$;
$d(E)$ is a map of left $B$-modules defined on the generators by
$$
\begin{array}{cccc}
d(E): & S(E) & \ra & S(E) \\
         & m(E \ot I) & \mapsto & \rho(E\ot I ,I\ot E)\cdot m(I\ot E), \\
         & m(I \ot E) & \mapsto & 0.
\end{array}
$$

\n (3) $S(F)=P(I \ot F) \oplus P(F \ot I)\{0,1\}$;
$d(F)$ is a map of left $B$-modules defined by
$$
\begin{array}{cccc}
d(F): & S(F) & \ra & S(F) \\
         & m(I \ot F) & \mapsto & \rho(I\ot F ,F\ot I)\cdot m(F\ot I), \\
         & m(F \ot I) & \mapsto & 0.
\end{array}
$$

\n (4) $S(EF)=P(E \ot F)\oplus P(I \ot EF) \oplus P(EF \ot I)\{0,1\} \oplus P(F \ot E)\{0,1\}[-1]$;
$d(EF)$ is a map of left $B$-modules defined by
$$
\begin{array}{cccc}
d(EF): & S(EF) & \ra & S(EF) \\
         & m(E\ot F) & \mapsto & \rho(E\ot F ,I\ot EF)\cdot m(I\ot EF), \\
          & m(I \ot EF) & \mapsto & \rho(I\ot EF ,EF\ot EF)\cdot \rho(EF\ot EF, EF\ot I) \cdot m(EF\ot I), \\
           & m(EF \ot I) & \mapsto & \rho(EF\ot I ,F\ot E)\cdot m(F\ot E), \\
         & m(F \ot E) & \mapsto & 0.
\end{array}
$$
\end{defn}

\begin{rmk} \label{Def S rmk}
(1) The DG $B$-modules $S(\g)$ are supposed to be the categorical comultiplication of the DG $A$-modules $P(\g)$, for all $\g \in \cal{B}$.
In particular, the classes $[S(\g)] \in K_0(HP(B))$ agree with the comultiplication $\de(\g) \in \utt$ under the isomorphism in Lemma \ref{K0 mod B}.

\n (2) The definition of $S(E)$ is motivated from an isomorphism $\delta(E) \cong (E\ot I \ra I\ot E)$ in the contact category $\tC _{oo}$ as in Figure \ref{1-2} in Section 1.3. The other definitions have similar motivations.
\end{rmk}

\begin{lemma}
$(S(\g), d(\g))$ is a $(t_1,t_2)$-graded DG $B$-module.
\end{lemma}
\proof
It suffices to prove that $d(\g)$ is of degree $(1;0,0)$ such that $d(\g)^2=0$.
We verify it for $\g=EF$ and leave other cases to the reader.
\begin{align*}
d^2(m(E\ot F))=&\rho(E\ot F ,I\ot EF)\cdot\rho(I\ot EF ,EF\ot EF)\cdot \rho(EF\ot EF, EF\ot I)\cdot m(EF\ot I)\\
              =&0\cdot \rho(EF\ot EF, EF\ot I)\cdot m(EF\ot I)=0, \\
d^2(m(I\ot EF))=&\rho(I\ot EF ,EF\ot EF)\cdot \rho(EF\ot EF, EF\ot I)\cdot \rho(EF\ot I ,F\ot E)\cdot m(F\ot E)\\
              =&\rho(I\ot EF ,EF\ot EF)\cdot0\cdot m(F\ot E)=0,
\end{align*}
from Relation (iii-C) in Definition \ref{Def S}. $d^2=0$ is obvious for the other two generators of $S(EF)$.

The degrees of the generators of $S(EF)$ are as follows:
\begin{gather*}
\op{deg}(m(E\ot F))=\op{deg}(m(I\ot EF))=(0;0,0), \\
\op{deg}(m(EF\ot I))=(0;0,-1), \quad \op{deg}(m(F\ot E))=(1;0,-1).
\end{gather*}
Then $d$ is of degree $(1;0,0)$ since
\begin{align*}
\op{deg}(d(m(E\ot F)))=&\op{deg}(\rho(E\ot F ,I\ot EF))+\op{deg}(m(I\ot EF))\\
=&(1;0,0)=\op{deg}(m(E\ot F))+(1;0,0), \\
\op{deg}(d(m(EF\ot I)))=&\op{deg}(\rho(EF\ot I ,F\ot E))+\op{deg}(m(F\ot E))\\
=&(1;0,-1)=\op{deg}(m(E\ot F))+(1;0,0),\\
\op{deg}(d(m(I\ot EF)))=&\op{deg}(\rho(I\ot EF ,EF\ot EF))+\op{deg}(\rho(EF\ot EF, EF\ot I))\\
&+\op{deg}(m(EF\ot I))=(1;0,0)=\op{deg}(m(I\ot EF))+(1;0,0). \qed
\end{align*}

\subsubsection{The right $A$-module structure on $S$}
In this subsection we describe the right $A$-module structure on $S$.
Let $m \times a$ denote the right multiplication for $m \in S, a\in A$ and let $m \cdot b$ denote the multiplication in $B$ for $m \in P(\g_1 \ot \g_2) \subset B, b \in B$.

\begin{defn} \label{rt grading S}
For $m \in S$ and $a\in A$, define the grading of the right multiplication $m \times a$ by:
$$\op{deg}(m \times a)=\op{deg}(m) + (\deh(a);\dt(a),\dt(a)),$$
where $\op{deg}$ is the grading in $B$, $\deh$ and $\dt$ are the gradings in $A$.
\end{defn}
\begin{rmk}
This definition is related to the categorification of $\de(T)=T \ot T$ in the proof of Theorem \ref{thm-ut-com} in Section 3.3.
Topologically, the definition comes from the fact that the generator $t \in H_1(S_o)$ is mapped to $t_1+t_2 \in H_1(S_{oo})$ under $\delta$ as in Figure \ref{1-2} in Section 1.3.
\end{rmk}

The right multiplication is a map of left DG $B$-modules defined on generators as follows:

\n (1) For an idempotent $e(\g)$, define
$$
\ba{cccc}
\times e(\g): & S(\g') &  \ra & S(\g') \\
                        &     m      & \mapsto & \delta_{\g,\g'}  ~ m.
\ea
$$

\vspace{0.2cm}
\n (2) For the generator $\rho(I, EF)$, define
$$
\ba{cccc}
\times \rho(I,EF): & S(I) & \ra & S(EF) \\
       & m(I\ot I)      & \mapsto  & \rho(I\ot I,I\ot EF) \cdot m(I\ot EF).
\ea
$$

\vspace{0.2cm}
\n (3) For the generator $\rho(EF, I)$, define
$$
\ba{cccc}
\times \rho(EF,I): & S(EF) & \ra & S(I) \\
       & m(EF\ot I)      & \mapsto  & \rho(EF\ot I,I\ot I) \cdot m(I\ot I),\\
       & m(E\ot F)       & \mapsto  & 0,\\
       & m(I\ot EF)       & \mapsto  & 0,\\
       & m(F\ot E)       & \mapsto  & 0.\\
\ea
$$

In general, define $m\times(a_1a_2):=(m\times a_1) \times a_2$ for $a_1, a_2 \in A$ and $m \in S$.

\begin{lemma}
The definition above gives a right DG $A$-module $S$:
\be
\item $(m \times a_1) \times a_2=(m \times a_1') \times a_2'$ for $a_1\cdot a_2=a_1'\cdot a_2' \in A$ and $m \in S$.
\item $d(m \times a)=d(m) \times a$ for $a \in A$ and $m \in M$.
\item The right multiplication is compatible with the grading in Definition \ref{rt grading S}.
\ee
\end{lemma}
\begin{proof}
For (1), since the only non-trivial relation in $A$ is $\rho(I,EF)) \cdot \rho(EF,I)=0$, it suffices to prove that $(m \times \rho(I,EF)) \times \rho(EF,I)=0,$ which follows from the definition.

\vspace{.1cm}
For (2), we verify the following from Relations (iii-A) and (iii-B) in Definition \ref{Def S}:
\begin{align*}
d(m(I\ot I) \times \rho(I,EF))=&\rho(I\ot I,I\ot EF) \cdot d(m(I\ot EF))\\
=&\rho(I\ot I,I\ot EF) \cdot \rho(I\ot EF,I\ot I)\cdot \rho(I\ot I, EF\ot I)\\
=&0=d(m(I\ot I)) \times \rho(I,EF),
\end{align*}
.Similarly, $d(m(I\ot EF)) \times \rho(EF,I)=0=d(m(I\ot EF) \times \rho(EF,I)).$

\vspace{.1cm}
For (3), we verify that
$$\op{deg}(m(EF\ot I) \times \rho(EF, I))=(0;0,-1)+(1;1,1)=\op{deg}(\rho(EF\ot I,I\ot I)\cdot m(I\ot I)).$$
Similarly, $\op{deg}(m(I\ot I) \times \rho(I,EF))=\op{deg}(\rho(I\ot I,I\ot EF) \cdot m(I\ot EF))$.
\end{proof}

Since the right multiplication is a map of left $A$-modules, we have:
$b \cdot (m \times a)=(b\cdot m) \times a,$
for $a \in A, b \in B$ and $m \in S$.
Hence $S$ is a $(t_1,t_2)$-graded DG $(B, A)$-bimodule.

\subsection{The categorification of the comultiplication on $\ut$}
In this section, we use the bimodule $S$ to categorify the comultiplication on $\ut$, i.e., prove Theorem \ref{thm-ut-com}.

Let $\delta: DGP(A) \xra{S \otimes_{A} -} DGP(B)$ given by tensoring with the DG $(B, A)$-bimodule $S$ over $A$.
\begin{lemma} \label{tau}
The functor $\delta$ maps $P(\g)$ to $S(\g) \in DGP(B)$ for all $\g \in \cal{B}$.
\end{lemma}
\begin{proof}
The proof is similar to that of Lemma \ref{eta}.
\end{proof}

There is an induced exact functor $\delta: HP(A) \xra{S \otimes_{A} -} HP(B)$ between the homology categories.

\begin{proof} [Proof of Theorem \ref{thm-ut-com}]
We compute the map on the Grothendieck groups:
$$K_0(\delta): K_0(HP(A)) \ra K_0(HP(B)).$$

\n (1) By Lemma \ref{tau}, $\delta(P(\g))=S(\g),$ for $\g \in \cal{B}=\{I,E,F,EF\}$.
Hence by Remark \ref{Def S rmk},
$$K_0(\delta)[P(\g)]=[S(\g)]=\de(\g) \in \utt.$$

\n (2) By the grading in Definition \ref{rt grading S}, $\delta(P(\g)\{n\})=S(\g)\{n,n\}$ for $n \in \Z$.
Hence,
$$K_0(\delta)(T^{n}[P(\g)])=K_0(\delta)([P(\g)\{n\}])=[S(\g)\{n,n\}]=T_1^n T_2^n[S(\g)]$$.

(1) and (2) together imply that $K_0(\delta)=\de:\ut \ra \utt$ since the $\Z$-linear maps $K_0(\delta)$ and $\de$ agree on the $\Z$-basis $\{T^n\g ~|~ \g \in \cal{B}, n \in \Z\}$ of $\ut$.
\end{proof}

\begin{rmk}
It is interesting to ask whether the properties of the comultiplication in Lemma \ref{ut lemma}, such as coassociativity, can be lifted to the categorical level.
We believe that the answer is positive since on the topological side $(\de \ot id)\circ \de$ and $(id \ot \de)\circ \de$ are both categorified to a functor $\tC(S_o, F_o) \ra \tC(S_{ooo}, F_{ooo})$, where $S_{ooo}$ is a triple punctured disk with $F_{ooo}$ consisting of two points on each boundary component of $\bdry S_{ooo}$.
\end{rmk}

\section{The linear action of $\ut$ on $V_1^{\ot n}$}
In this section, we give a distinguished basis $\bn$ of $V_1^{\ot n}$ and express the action in terms of $\bn$.

\subsection{The representations $V_1$ and $V_2$}
Let $V_1$ be a free $\zt$-module with a basis $\cal{B}_1=\{|0 \ran, |1 \ran \}.$
A {\em parity} $\op{p}$ is a $\Z$-grading on $V_1$ given by $\op{p}(|0 \ran)=0, \op{p}(|1 \ran)=1$.
Define an action of $\ut$ on $V_1$ by:
\begin{gather*}
E|0\ran=0, \quad F|0\ran=|1\ran, \\
E|1\ran=(1-t)|0 \ran, \quad F|1\ran=0,\\
T|0\ran=t|0\ran, \quad T|1\ran=t(-1)^{2}|1\ran.
\end{gather*}
%\begin{figure}[h]
%\begin{overpic}
%[scale=0.2]{3-1.eps}
%\put(12,-3){$E$}
%\put(12,6){$F$}
%\put(61,-3){$E$}
%\put(61,6){$F$}
%\put(86,-3){$E$}
%\put(86,6){$F$}
%\end{overpic}
%\caption{The pictures describe the representation $V_1$ and its tensor product $V_2$.}
%\label{3-1}
%\end{figure}
Note that the factor $(-1)^{2}$ in $T|1\ran$ is the shadow of a cohomological grading shift by $2$ on the categorical level.
The parities on $\ut$ and $V_1$ are compatible with respect to the action.
More precisely, the operators $E$ and $F$ change the parity by $-1$ and $1$, respectively:
\begin{gather*}
\op{p}(F|0\ran)=\op{p}(|1\ran)=1=\op{p}(F)+\op{p}(|0\ran), \quad \op{p}(E|1\ran)=\op{p}(|0\ran)=0=\op{p}(E)+\op{p}(|1\ran).
\end{gather*}

Let $V_2=V_1 \ot_{\zt} V_1$ be a free $\zt$-module with a basis $\cal{B}'_2=\cal{B}_1 \times \cal{B}_1=\{|00\ran, |01 \ran, |10 \ran, |11 \ran \}.$
The action of $\ut$ on $V_2$ is induced by the comultiplication $\de: a\cdot(v\ot w)=\de(a)(v \ot w),$
for $a \in \ut, v,w \in V_1$.
Note that the action of $\utt$ on $V_1 \ot V_1$ is the graded tensor product:
$(a_1 \ot a_2)(v \ot w)=(-1)^{\op{p}(a_2)\op{p}(v)}a_1v \ot a_2w.$

\begin{lemma}
The action of $\ut$ on $V_2$ is given in the basis $\cal{B}_2$ as follows:
\begin{gather*}
E|00\ran=0, \quad F|00\ran=|01\ran + t|10\ran, \\
E|01\ran=(1-t)|00\ran, \quad F|01\ran=t|11\ran,\\
E|10\ran=(1-t)|00\ran, \quad F|10\ran=-|11\ran,\\
E|11\ran=(1-t)|01\ran-(1-t)|10\ran, \quad F|11\ran=0,\\
T(v)=t^2  v ~\mbox{for all}~ v=v_1 \ot v_2 \in \cal{B}_2, ~\mbox{where}~ v_1, v_2 \in \cal{B}_1.
\end{gather*}
\end{lemma}
\proof
We verify some of the formulas and leave others to the reader:
\begin{gather*}
T(v)=\de(T)(v_1 \ot v_2)=(T \ot T)(v_1 \ot v_2)=T(v_1) \ot T(v_2)=t^2  v, \\
F|00\ran=\de(F)|00\ran=(1\ot F + F \ot T)|00\ran=|01\ran + t|10\ran, \\
F|10\ran=\de(F)|10\ran=(1\ot F + F \ot T)|10\ran=(1\ot F)|10\ran=-|11\ran. \qed
\end{gather*}

\subsection{The representations $V_1^{\ot n}=V_1^{\ot n}$}
There is an action of $\ut$ on the $n$-th tensor product $V_1^{\ot n}=V_1^{\ot n}$ induced by iterated comultiplication.
Consider a $\zt$-basis $\cal{B}'_n$ of $V_1^{\ot n}$:
$$\cal{B}'_n = \cal{B}_1^{\times n}=\{\mf{a}=|a_1 \dots  a_n\ran ~|~ a_i \in \{0, 1\} \}.$$
We call $\cal{B}'_n$ the tensor product basis of $V_1^{\ot n}$.
Consider another presentation of the basis:
$$\cal{B}_n=\{\mf{x}=(x_1, \dots ,x_k) ~|~ 1\leq x_1 < \dots < x_k \leq n, 1\leq k \leq n\} \sqcup \{\es\}.$$
There is a one-to-one correspondence between $\cal{B}_n$ and $\cal{B}'_n$:
$$
\ba{ccc}
\cal{B}_n & \ra & \cal{B}'_n \\
\es & \mapsto & a=|0\dots 0 \ran \\
\mf{x}=(x_1, \dots, x_k) & \mapsto & a=|a_1 \dots a_n\ran
\ea
$$
where
$$a_i= \left\{
\begin{array}{cc}
1 & \mbox{if}~ i=x_l, ~\mbox{for some}~ 1\leq l \leq k;\\
0 & \mbox{otherwise}.
\end{array}
\right.$$
There is a partition: $\bn=\bigsqcup\limits_{k=0}^{n}\bnk$, where $\bnk=\{\mf{x}=(x_1, \dots, x_k) ~|~ 1\leq x_1 < \dots < x_k \leq n \}$ for $1\leq k \leq n$ and $\cal{B}_{n,0}=\{\es\}$.
Let $V_1^{\ot n} = \bigoplus\limits_{k=0}^{n} V_{n,k}$ be the corresponding decomposition of $V_1^{\ot n}$, where $V_{n,k}$ is spanned by the basis $\bnk$ for $0 \leq k \leq n$.

In the $\ut$-action, $F$ converts a state from $|0\ran$ to $|1\ran$ for one factor of the state in $\bn'$.
In particular, $F$ increases the number of $|1\ran$ states by $1$; similarly, $E$ decreases the number of $|1\ran$ states by $1$:
\begin{gather*}
F: V_{n,k} \ra V_{n,k+1} \quad E: V_{n,k} \ra V_{n,k-1}
\end{gather*}
We introduce some notations to describe the action in terms of $\bn$.
For $\mf{x}=(x_1,\dots,x_k) \in \bnk$, let $\mf{\bar{x}}=(\bar{x}_1, \dots, \bar{x}_{n-k})\in \cal{B}_{n,n-k}$ be the increasing sequence consisting of the complement $\{1,\dots,n\} \backslash \{x_1,\dots,x_k\}$ of $\mf{x}$ in $\{1,\dots,n\}$.
Define
$$
\beta(\mf{x}, \bar{x}_j)=\left|\{l \in \{1,\dots,k\} ~|~ x_l < \bar{x}_j\}\right| + 2 \left|\{l \in \{1,\dots,k\} ~|~ x_l > \bar{x}_j\}\right|.
$$
Let $\mf{x}\sqcup\{\bar{x}_j\}$ be an increasing sequence obtained by adjoining $\bar{x}_j$ to $\mf{x}$ and $\mf{x} \backslash \{x_i\}$ be an increasing sequence obtained by removing $x_i$ from $\mf{x}$.

\begin{lemma} \label{utvn}
The $\ut$-action on $V_1^{\ot n}$ is given for $\mf{x} \in \bnk$ by:
\begin{gather*}
I(\mf{x})=\mf{x}, \qquad T(\mf{x})=t^n(-1)^{2n}\mf{x}, \\
F(\mf{x})=\sum_{j=1}^{n-k}t^{n-\bar{x}_j}(-1)^{\beta(\mf{x}, \bar{x}_j)}\mf{x}\sqcup\{\bar{x}_j\}, \\
E(\mf{x})=\sum_{i=1}^{k}\left((-1)^{1-i}\mf{x} \backslash \{x_i\} + ~ t (-1)^{2-i} \mf{x} \backslash \{x_i\}\right).
\end{gather*}
\end{lemma}
\begin{proof}
We only check the action of $F$:
\begin{align*}
F(\mf{x})= \de^n(F)(\mf{x}) =  \sum_{j=1}^{n} (1 \ot \cdots 1 \ot \underset{j-\mbox{\footnotesize{th}}}{F} \ot T \cdots \ot T)(\mf{x})
=  \sum_{j=1}^{n-k} t^{n-\bar{x}_j}(-1)^{\beta(\mf{x}, \bar{x}_j)}\mf{x}\sqcup\{\bar{x}_j\},
\end{align*}
where the exponent of $t$ comes from $(n-\bar{x}_j)$'s $T$ in the $j$-th term of $\de^n(F)$; the exponent of $-1$ comes from the graded tensor product and the action of $T$ on the state $|1\ran$: $T|1\ran=t(-1)^{2}|1\ran.$
\end{proof}

The exponents of $(-1)$ in the expressions including $\beta(\mf{x}, \bar{x}_j)$ in $F(\mf{x})$ and $2n$ in $T(\mf{x})$ will be used as cohomological grading shifts in the bimodule $C_n$ in Section 6.1 which categorifies the $\ut$-action.

\section{The $t$-graded DG algebra $R_n$ through the quiver $\gn$}
We define a family of $t$-graded DG algebras $R_n$ which are closely related to the strands algebras associated to an $n$ times punctured disk in bordered Heegaard Floer homology.
In Section 5.1 we briefly review the definition of the strands algebras and introduce the {\em decorated rook diagrams} as a variant of {\em rook diagrams}.
In Section 5.2 we construct the quivers $Q_n$ whose arrows are given by the decorated rook diagrams.
In Section 5.3 we define $R_n$ as a quotient of the path algebra $\F Q_n$ whose differential is given by resolutions of crossings and markings for the decorated rook diagrams.
We show that $R_n$ is formal and categorify $V_1^{\ot n}$ through $DGP(R_n)$ generated by some projective DG $R_n$-modules.
In Section 5.4 we define a variant $A \bt R_n$ of $A \ot R_n$ which will be used in the construction of the $(H(R_n), A\bt R_n)$-bimodule $C_n$ in Section 6.

\subsection{Background on the strands algebras and the rook monoid}

\subsubsection{The strands algebra}
Lipshitz, Ozsv\'ath and Thurston in \cite{LOT} introduced the {\em strands algebra} associated to a connected surface with one closed boundary component.
Later on, the strands algebra of any surface with boundary was defined by Zarev \cite{Za}.
We refer to \cite[Section 2]{Za} for more detail on the strands algebras associated to an {\em arc diagram}.
\begin{defn}
An {\em arc diagram} $\cal{Z}=(\mf{Z},\mf{a},M)$ is a triple consisting of a collection $\mf{Z}=\{Z_1,\dots,Z_l\}$ of line segments, a collection $\mf{a}=\{a_1, \dots, a_{2k}\}$ of distinct $2k$ points in $\mf{Z}$, and a matching $M$ of $\mf{a}$ into $k$ pairs of points.
\end{defn}
A surface $F(\cal{Z})$ can be constructed from an arc diagram $\cal{Z}$ by starting with a collection of rectangles $[0,1] \times Z_j$ for $j=1,\dots,l$, and attaching a $1$-handle with endpoints on $M^{-1}(i) \times \{0\}$ for each $i=1,\dots,k$.
An $n$ times punctured disk $\Sigma_n$ can be parametrized by $\cal{Z}(2n)=(\mf{Z},\mf{a},M)$:
\begin{itemize}
\item $\mf{Z}=\{Z\}$ is a single vertical line segment;
\item $\mf{a}=\{1,\dots, 2n\}$ is a collection of $2n$ points in $Z$ ordered from bottom to top;
\item $M$ matches $\mf{a}$ into $n$ pairs of adjacent points $\{2i-1, 2i\}$ for $i=1,\dots,n$.
\end{itemize}
We fix the arc diagram $\cal{Z}_n=\cal{Z}(2n)=(\mf{Z},\mf{a},M)$ throughout this paper.
See the diagram on the left of Figure \ref{4-1-0}.
The associated strands algebra is generated by {\em strands diagrams}.

\begin{defn} \label{strand}
Given the arc diagram $\cal{Z}_n$, a {\em strands diagram with k strands} is a triple $(S,T,\phi)$, where $S, T$ are $k$-element subsets of $\mf{a}$ and $\phi: S\ra T$ is a bijection with $i \leq \phi(i)$ for all $i \in S$.
\end{defn}

Geometrically, a strands diagram $(S,T,\phi)$ with $k$-strands is an isotopy class of a set of $k$ strands with a minimal number of crossings which connect the $k$ points in $S$ as a subset of the $2n$ points on the left to the $k$ points in $T$ as a subset of the $2n$ points on the right.
The restriction that $\phi$ is non-decreasing means that strands stay horizontal or move up when read from left to right.
\begin{figure}[h]
\begin{overpic}
[scale=0.18]{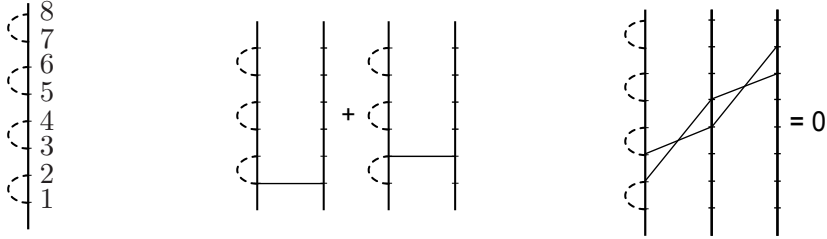}
\put(4,3.5){$1$}
\put(4,6.5){$2$}
\put(4,10){$3$}
\put(4,13){$4$}
\put(4,16.5){$5$}
\put(4,20){$6$}
\put(4,23){$7$}
\put(4,26.5){$8$}
\end{overpic}
\caption{The diagram on the left is the arc diagram $\cal{Z}_4$; the diagram in the middle gives a primitive idempotent in $\cal{A}(\cal{Z}_3,1)$ according to the idempotent constraint; the diagram on the right describes a double crossing.}
\label{4-1-0}
\end{figure}

The associated {\em strands algebra} $\cal{A}(\cal{Z}_n)=\bigoplus\limits_{k=0}^{n}\cal{A}(\cal{Z}_n,k)$, where $\cal{A}(\cal{Z}_n,k)$ is $\F$-vector space generated by strands diagrams with $k$ strands with the following two constraints.
The first constraint on a strands diagram $(S,T,\phi)$ is that $|S \cap \{2i-1, 2i\}|\leq 1$ and $|T \cap \{2i-1, 2i\}|\leq 1$ for $i=1,\dots,n$, i.e., the number of intersection points of the strands diagram with any pair $\{2i-1, 2i\}$ is at most $1$.
We call it the {\em $1$-handle constraint}.
The second constraint is about horizontal strands: if any generator $r \in \cal{A}(\cal{Z}_n)$ contains $(S_1,T_1,\phi_1)$ as a summand where $\phi_1(i)=i$ for some $i$, then $r$ must contain another summand $(S_2, T_2, \phi_2)$, where $S_2=S_1\backslash\{i\}\cup\{j\}, T_2=T_1\backslash\{i\}\cup\{j\}$ and $\phi_2|_{S_2\backslash\{j\}}=\phi_1|_{S_1\backslash\{i\}}$, $\phi_2(j)=j$ for $j=i-(-1)^i$.
Note that $i$ and $j$ form a pair in $\cal{Z}_n$ so that $j \notin S_1, T_1$ according to the $1$-handle constraint.
As in the middle diagram in Figure \ref{4-1-0}, a primitive idempotent is a sum of two horizontal strands: $(S_1,T_1,\phi_1) + (S_2, T_2, \phi_2)$, where $S_1=T_1=\{1\}, S_2=T_2=\{2\}$ and $\phi_1, \phi_2$ are the identities.
In particular, each summand $(S_1,T_1,\phi_1)$ or $(S_2, T_2, \phi_2)$ is not a generator of $\cal{A}(\cal{Z}_3,1)$.
Since horizontal strands represent idempotents in the algebra, we call it the {\em idempotent constraint}.

The product $a\cdot b$ of two strands diagrams is set to be zero if the right side of $a$ does not match the left side of $b$; otherwise, the product is the horizontal juxtaposition of $a$ and $b$.
If two strands cross each other twice in the juxtaposition, the product is set to be zero.
It is called the {\em double crossing relation}.
See the right diagram in Figure \ref{4-1-0}.

The differential of a strand diagram is the sum over all ways of resolving one crossing of the diagram.
See Figure \ref{4-1-00} for an example.
\begin{figure}[h]
\begin{overpic}
[scale=0.2]{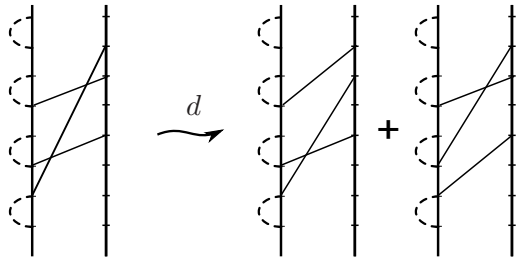}
\put(35,28){$d$}
\end{overpic}
\caption{Differential given by resolving crossings}
\label{4-1-00}
\end{figure}

\subsubsection{The generalized rook diagrams}
Since the matching in the arc diagram $\cal{Z}_n$ simply identifies $2i-1$ and $2i$ in each pair, we introduce a generalization of {\em rook diagrams} to describe the strands algebra $\cal{A}(\cal{Z}_n)$.
We first recall the definition of {\em rook monoid} from \cite{So}.
\begin{defn}
Let $n$ be a positive integer and $\mf{n}=\{1,\dots,n\}$.
The {\em rook monoid} $\cal{R}_n$ is the set of all one-to-one maps $\sigma$ with domain $I(\sigma) \subset \mf{n}$ and range $J(\sigma) \subset \mf{n}$.
The multiplication on $\cal{R}_n$ is given by composition of maps.
\end{defn}
There is a diagrammatic presentation of the rook monoid, called {\em rook diagrams}, given in \cite{FHH}.
A rook diagram associated to an element $\sigma \in \cal{R}_n$ is a graph on two rows of $n$ vertices such that vertex $i$ in the bottom row is connected to vertex $j$ in the top row if and only if $\sigma(i)=j$.
The multiplication is given by vertical concatenation of two rook diagrams.

We define the {\em generalized rook diagrams} by adding a new type of diagrams with loops attached at vertices to the rook diagrams.
The strand diagrams in $\cal{A}(\cal{Z}_n)$ corresponding to the loops and the multiplication rule on the loops will be given in Section 5.1.3.
\subsubsection{From strands diagrams to generalized rook diagrams}
We describe the translation from the strands diagrams in $\cal{A}(\cal{Z}_n)$ to the generalized rook diagrams on some generators of $\cal{A}(\cal{Z}_n)$.

In the left diagram in Figure \ref{4-1-1}, an idempotent in $\cal{A}(\cal{Z}_n)$ as a sum of two horizontal strands is translated to a single vertical rook diagram $id$ from the state $|001\ran$ to itself.
The translation consists of three steps: (1) rotate a strand diagram counterclockwise by $\frac{\pi}{2}$;
(2) replace the identified points $\{2i-1, 2i\}$ of $i$th pair in the strand diagram by the $i$th vertex from the right in a rook diagram; and (3) combine two horizontal strands in the strand diagram into a single vertical rook diagram.

In the middle diagram in Figure \ref{4-1-1}, an upward strand connecting two points in a pair of $\cal{Z}_n$ is translated to a loop $\rho$ attached at the corresponding state $|1\ran$.
Note that the loop $\rho$ is nilpotent in $\cal{A}(\cal{Z}_n)$: $\rho^2=0$.
Correspondingly, the square of any loop is defined as zero.
In the right diagram in Figure \ref{4-1-1}, a strand connecting two points in different pairs is translated to a left-veering rook diagram.
In general, a strand diagram with $k$ strands is translated to a generalized rook diagram as a superposition of the corresponding $k$ generalized rook diagrams.
\begin{figure}[h]
\begin{overpic}
[scale=0.2]{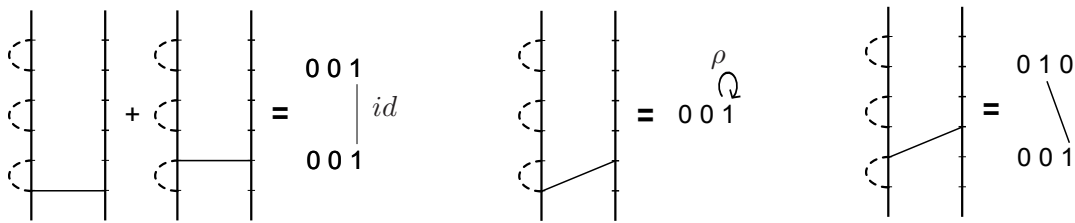}
\end{overpic}
\put(-265,40){$id$}
\put(-137,60){$\rho$}
\caption{The translation for $\cal{A}(\cal{Z}_3,1)$: the left one is the idempotent $id$; the middle one is the nilpotent element $\rho$; the right one is a left-veering rook diagram.}
\label{4-1-1}
\end{figure}

Since the strands diagrams stay horizontal or move up, the corresponding rook diagrams always have negative or infinity slopes, i.e., they stay vertical or move to the left when read from bottom to top.
We call these generalized rook diagrams as {\em left-veering rook diagrams}.

The differential on $\cal{A}(\cal{Z}_n)$ induces a differential $d_1$ on the left-veering rook diagrams.
As in Figure \ref{4-1-2} the resolution of such a crossing contains two left-veering rook diagrams with loops, since a vertical strand in a left-veering rook diagram corresponds to a sum of two terms in $\cal{A}(\cal{Z}_n)$.

%We first rotate a strand diagram counterclockwise by $\frac{\pi}{2}$. Then we replace the identified points $\{a_{2i-1}, a_{2i}\}$ of $i$-th pair in a strand diagram by the $i$-th vertex in a rook diagram.Therefore, the $n$ pairs of $2n$ points are replaced by a row of $n$ vertices. In a strand diagram $(S,T,\phi)$, there is at most one strand connecting to each pair of identified points from the $1$-handle constraint on strand diagrams.Hence there is at most one strand connecting to any vertex in the corresponding rook diagram.We label the $i$-th vertex in the bottom row by the state $|1\ran$ if $|S \cap \{a_{2i-1}, a_{2i}\}| = 1$; otherwise, we label it by the state $|0\ran$.Similarly, we label the $i$-th vertex in the top row by the state $|1\ran$ if $|T \cap \{a_{2i-1}, a_{2i}\}| = 1$; otherwise, we label it by the state $|0\ran$.Finally, we get a rook diagram with decorations in the states $\{|0\ran, |1\ran\}$.
\begin{figure}[h]
\begin{overpic}
[scale=0.16]{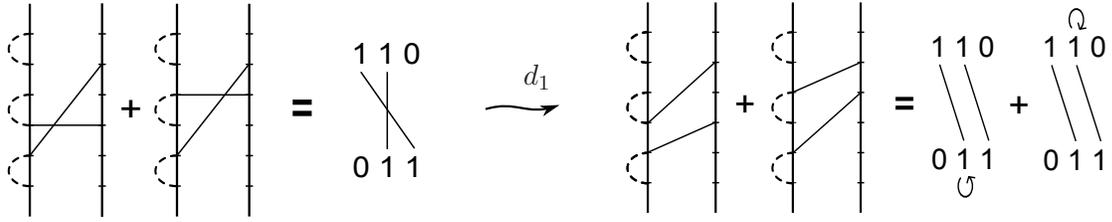}
\end{overpic}
\put(-220,50){$d_1$}
\caption{The translation for a differential of a crossing}
%\rho((2,3) \xra{1} (2,3)) \cdot r((2,3) \xra{1,0,(0)} (1,3)) \cdot r((1,3) \xra{2,0,(0)} (1,2))+ r((2,3) \xra{1,0,(0)} (1,3)) \cdot r((1,3) \xra{2,0,(0)} (1,2)) \cdot \rho((1,2) \xra{2} (1,2))
\label{4-1-2}
\end{figure}

\subsubsection{New ingredients}
We define the {\em decorated rook diagrams} as the left-veering rook diagrams, possibly added with some {\em markings}.

The motivation of introducing diagrams with markings is given as follows.
There is a relation in $\cal{A}(\cal{Z}_n): (S,T,\phi)\cdot(S',T',\phi')=0,$ if $T \cap \{2i-1, 2i\} = \{2i-1\}$ and $S' \cap \{2i-1, 2i\} = \{2i\}$ for some $i$ since their endpoints do not match.
See the left diagram in Figure \ref{4-1-3}.
But the endpoints of the corresponding decorated rook diagrams do match.
We introduce a new rook diagram with a marking at the position corresponding to the pair $\{2i-1, 2i\}$.
We deform the relation in $\cal{A}(\cal{Z}_n)$ to be a differential $d_0$ of this new diagram with the marking.
In general, the differential $d_0$ of a strand diagram is the sum over all ways of resolving one marking of the diagram.

We define a differential $d=d_0+d_1$ on the decorated rook diagrams.
In other words, $d$ is a combination of the resolutions of crossings and those of markings as in the right part of Figure \ref{4-1-3}.
For decorated rook diagrams with no crossings and markings, the differential $d$ is defined to be zero.
\begin{figure}[h]
\begin{overpic}
[scale=0.18]{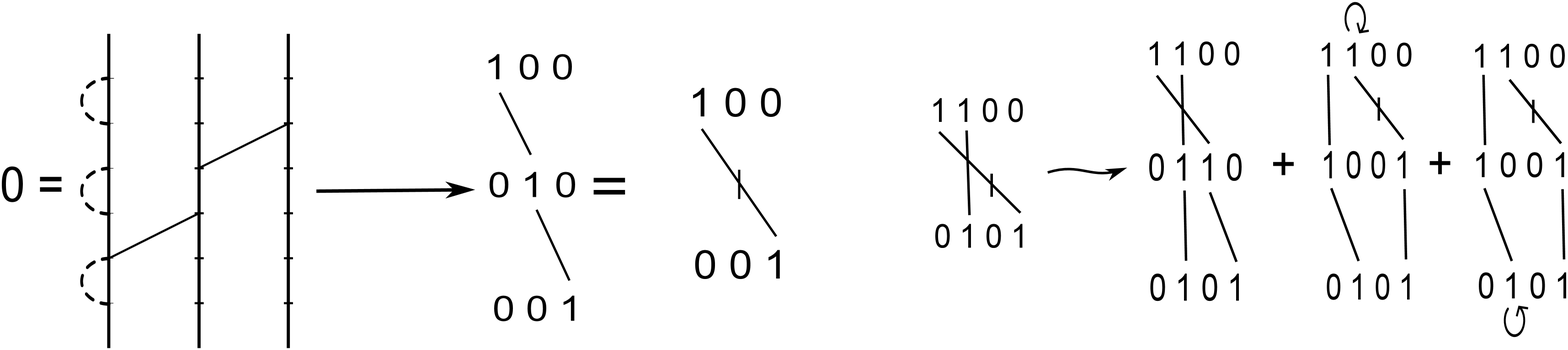}
\put(9,-1){$\phi$}
\put(14,-1){$\phi'$}
\put(69,12){$d$}
\put(42,10){$d_0$}
\put(20,11){{\small deform}}
\end{overpic}
\caption{The deformation of a relation in $\cal{A}(\cal{Z}_n)$ is on the left; a differential of a decorated rook diagram with one marking and one crossing is on the right.}
\label{4-1-3}
\end{figure}
\begin{rmk}
There are two types of resolutions $d_0, d_1$ for the decorated rook diagrams.
On the one hand, $d_1$ is inherited from the strands algebras via studying moduli spaces of holomorphic curves.
On the other hand, the author do not know a moduli space or contact topological interpretation of $d_0$.
The definition is only used for the algebraic construction of the bimodule $C_n$ in Section 6.
\end{rmk}

We introduce the notion of {\em elementary decorated rook diagrams} which cannot be decomposed as a concatenation of two nontrivial pieces.
The elementary decorated rook diagrams will give generators of the algebra $R_n$ in Definition \ref{rnk}.
Notice that $n$-tuples of $\{|0\ran, |1\ran\}$ are elements in the basis $\bn$ of the representation $V_1^{\ot n}$.
Hence, a decorated rook diagram can be viewed as a map from one element of $\bn$ in the bottom row to the other element of $\bn$ in the top row.

\begin{defn}
{\em Elementary decorated rook diagrams} consist of two types:

\n(1) a loop $\mf{x} \xra{i} \mf{x}$ attached at $x_i$ for $i=1,\dots, k$ and $\mf{x} \in \bnk$;

\n(2) a decorated rook diagram $\mf{x} \xra{i,s_1} \mf{y}$ with $s_1$ crossings and $\sum\limits_{j=i}^{i+s_1}(x_j-y_j-1)$ markings associated to a matching $\sigma: \mf{x}=(x_1, \dots, x_k) \ra \mf{y}=(y_1, \dots, y_k)$, where $i \in\{1,\dots,k\}$, $s_1\geq 0$ such that
\begin{gather*}
\sigma(x_{i+s_1})=y_i, \\
\sigma(x_{j})=y_{j+1}=x_{j} ~\mbox{for}~ j \in \{i,\dots,i+s_1-1\}, \\
\sigma(x_j)=y_j=x_j ~\mbox{for} ~ j \notin \{i,i+1,\dots,i+s_1\}.
\end{gather*}
\end{defn}

The algebraic definition above is technical while the corresponding rook diagrams are easier to follow as in Figure \ref{4-1-5}.
Given an elementary decorated rook diagram $\mf{x} \xra{i,s_1} \mf{y}$, define $$\mf{v}=(x_i-y_i-1, \dots, x_{i+s_1}-y_{i+s_1}-1) \in \N^{s_1+1}.$$
The vector $\mf{v}$ counts the numbers of $|0\ran$ states between the $j$th $|1\ran$ states $\{x_j\}$ in $\mf{x}$ and $\{y_j\}$ in $\mf{y}$ for $j=i,\dots,i+s_1$.
Let $s_0(\mf{v})= \sum_{l=0}^{s_1} v_l$ denote the total number of $|0\ran$ states between $x_{i+s_1}$ and $y_{i}$.
Then $\mf{x}$ and $\mf{y}$ only differ at two positions $x_{i+s_1}$ and $y_{i}$ which are connected by a left-veering strand.
On this strand, there are $s_1$ crossings with vertical strands and $s_0(\mf{v})$ markings, i.e., all possible crossings and markings must be on the strand.
\begin{figure}[h]
\begin{overpic}
[scale=0.22]{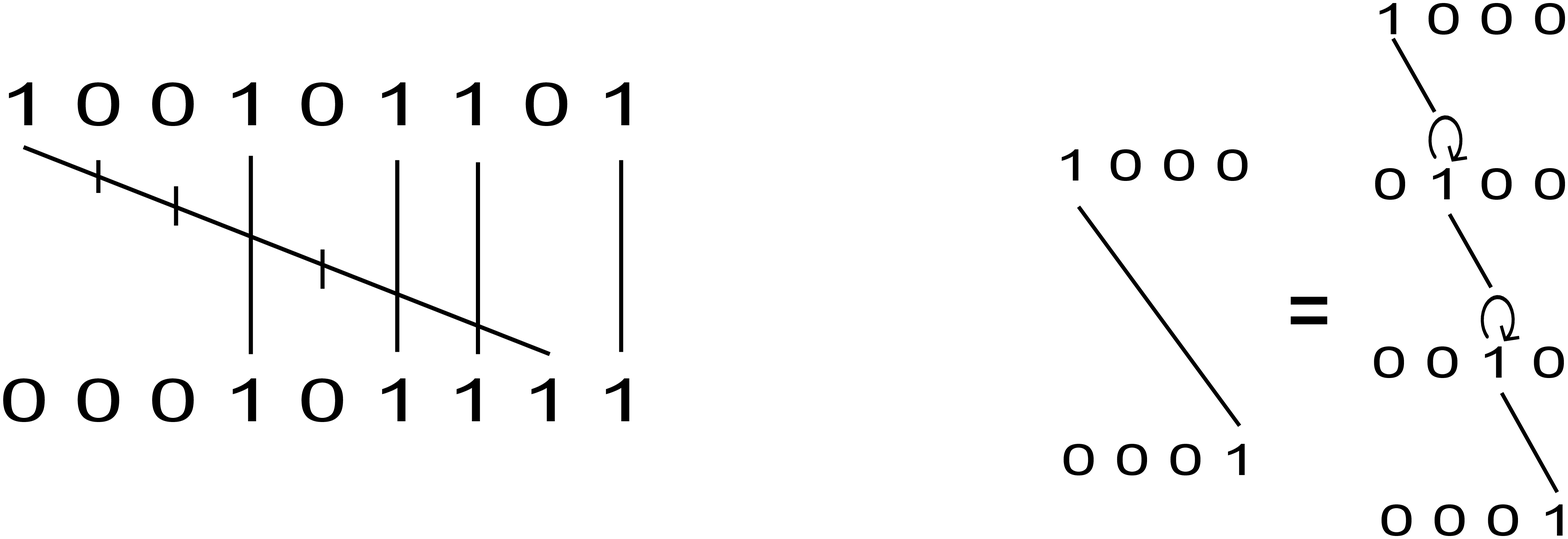}
\put(-7,8){$\mf{x}:$}
\put(-7,27){$\mf{y}:$}
\put(14,4){$x_1$}
\put(34,4){$x_{4}$}
\put(0,32){$y_1$}
\put(29,32){$y_{4}$}
\end{overpic}
\caption{The left-hand diagram is an elementary rook diagram $\mf{x} \xra{i,s_1,\mf{v}} \mf{y}$, where $i=1, s_1=3, \mf{v}=(2,1,0,0)$; in the right-hand diagram, we use the left-hand side to denote the composition of elementary diagrams on the right-hand side.}
\label{4-1-5}
\end{figure}
Although $\mf{v}$ is determined by $\mf{x} \xra{i,s_1} \mf{y}$, we will still write the decorated rook diagram as $\mf{x} \xra{i,s_1,\mf{v}} \mf{y}$.

Relations for concatenation of decorated rook diagrams are quite different from those for the strands diagrams as in Figure \ref{4-1-6}:
\begin{itemize}
\item The double crossing in decorated rook diagrams is not zero.
\item An isotopy of a crossing does not give the same decorated rook diagram.
\end{itemize}
For more detail about the relations, refer to Definition \ref{rnk} of the algebra $R_n$.
\begin{figure}[h]
\begin{overpic}
[scale=0.20]{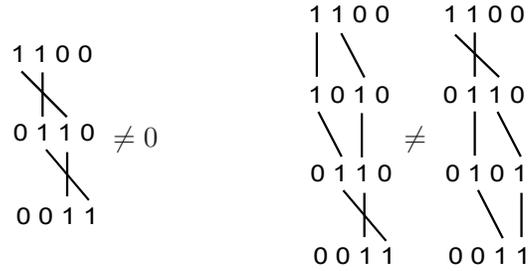}
\end{overpic}
\put(-48,45){$\neq$}
\put(-158,45){$\neq 0$}
\caption{Some relations for decorated rook diagrams}
\label{4-1-6}
\end{figure}

\subsection{The quiver $\gn$}
In this section, we construct the quiver $\gn=\sqcup_{k=0}^{n} \gnk$ for $n>0$.
\begin{defn} [Quiver $\gnk=(V(\gnk), A(\gnk))$]
\n \be
\item
Let $V(\gnk)=\bnk$ be the set of vertices.
\item
Let $A(\gnk)$ be the set of arrows consisting of two types:
\begin{align*}
\mbox{Loops:}~ & \{\mf{x} \xra{i} \mf{x} ~|~ i=1,\dots,k;~ \mf{x} \in \bnk\}, \\
\mbox{Arrows:}~ & \{\mbox{elementary decorated rook diagrams}~ \mf{x} \xra{i,s_1,\mf{v}} \mf{y} \}
\end{align*}
\ee
\end{defn}

\begin{example} [Quiver $\g_{4,2}$]
$V(\g_{4,2})=\{(3,4), (2,4), (1,4), (2,3), (1,3), (1,2)\}$.
For $\mf{x}=(x_1,x_2)$, there exist two loops, one for each $x_i$.
There are $6$ arrows without crossings or markings,
\begin{gather*}
\{(3,4)\xra{1,0,(0)}(2,4),~(2,4)\xra{1,0,(0)}(1,4),~(2,4)\xra{2,0,(0)}(2,3),\\
(1,4)\xra{2,0,(0)}(1,3),~(2,3)\xra{1,0,(0)}(1,3),~(1,3)\xra{2,0,(0)}(1,2)\}
\end{gather*}
and $6$ arrows with crossings or markings:
\begin{gather*}
\{(3,4)\xra{1,0,(1)}(1,4),~(3,4)\xra{1,1,(0,0)}(2,3),~(3,4)\xra{1,1,(1,0)}(1,3),\\
(1,4)\xra{2,0,(1)}(1,2),~(2,3)\xra{1,1,(0,0)}(1,2),~(2,4)\xra{1,1,(0,1)}(1,2)\}
\end{gather*}
\begin{figure}[h]
\begin{overpic}
[scale=0.25]{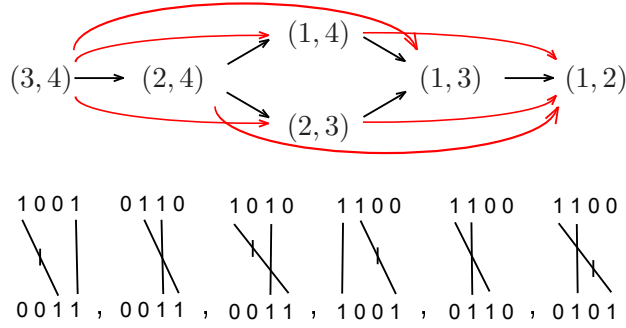}
\end{overpic}
\put(-230,88){$(3,4)$}
\put(-180,88){$(2,4)$}
\put(-75,88){$(1,3)$}
\put(-20,88){$(1,2)$}
\put(-125,105){$(1,4)$}
\put(-125,70){$(2,3)$}
\caption{The top diagram describes the quiver $\g_{4,2}$, where black lines denote arrows without crossings or markings and red lines denote arrows with crossings or markings which are represented in the bottom diagram.}
\label{4-2}
\end{figure}
\end{example}

\subsection{The $t$-graded DG algebra $R_n$}
We define the $t$-graded DG algebra $R_n=\bigoplus\limits_{k=0}^{n} R_{n,k}$, where $R_{n,k}=\F \gnk / \sim$ is a quotient of the path algebra $\F \gnk$ of the quiver $\gnk$ with a differential.
\begin{defn} [$t$-graded DG algebra $R_n$] \label{rnk}
$R_n$ is an associative $t$-graded $\F$-algebra with a differential $d$ and a grading $\op{deg}=(\deh, \dt) \in \Z^2$.

\vspace{.1cm}
\n (A) $R_{n}$ has idempotents $e(\mf{x})$ for each vertex $\mf{x}$ in $\gn$, generators $\rho(\mf{x} \xra{i} \mf{x})$ for each loop $\mf{x} \xra{i} \mf{x}$ and $r(\mf{x} \xra{i,s_1,\mf{v}} \mf{y})$ for each arrow $\mf{x} \xra{i,s_1,\mf{v}} \mf{y}$ in $\gn$.
The relations consist of $4$ groups:

\n (i) idempotents:
\begin{gather*}
e(\mf{x}) \cdot e(\mf{y})=\delta_{\mf{x},\mf{y}}\cdot e(\mf{x}) ~\mbox{for all}~ \mf{x}, \mf{y}, \\
e(\mf{x}) \cdot \rho(\mf{x} \xra{i} \mf{x})= \rho(\mf{x} \xra{i} \mf{x}) \cdot e(\mf{x})=\rho(\mf{x} \xra{i} \mf{x}) ~\mbox{for all}~ \rho(\mf{x} \xra{i} \mf{x}), \\
e(\mf{x}) \cdot r(\mf{x} \xra{i,s_1,\mf{v}} \mf{y})=r(\mf{x} \xra{i,s_1,\mf{v}} \mf{y}) \cdot e(\mf{y})= r(\mf{x} \xra{i,s_1,\mf{v}} \mf{y}) ~\mbox{for all}~ r(\mf{x} \xra{i,s_1,\mf{v}} \mf{y});
\end{gather*}
\n (ii) nilpotent loops:
\begin{gather*}
\rho(\mf{x} \xra{i} \mf{x}) \cdot \rho(\mf{x} \xra{i} \mf{x})=0 ~\mbox{for all}~ \rho(\mf{x} \xra{i} \mf{x});
\end{gather*}
\n (iii) commutativity for disjoint diagrams:
\begin{gather*}
\rho(\mf{x} \xra{i'} \mf{x}) \cdot \rho(\mf{x} \xra{i} \mf{x})=\rho(\mf{x} \xra{i} \mf{x}) \cdot \rho(\mf{x} \xra{i'} \mf{x}) ~\mbox{if}~ i'\neq i, \\
\rho(\mf{x} \xra{i'} \mf{x}) \cdot r(\mf{x} \xra{i,s_1,\mf{v}} \mf{y})=r(\mf{x} \xra{i,s_1,\mf{v}} \mf{y}) \cdot \rho(\mf{y} \xra{i'} \mf{y}) ~\mbox{if}~ i' \notin \{i,\dots,i+s_1\},
\end{gather*}
$$r(\mf{x} \xra{i,s_1,\mf{v}} \mf{y}) \cdot r(\mf{y} \xra{i',s_1',\mf{v'}} \mf{z})=r(\mf{x} \xra{i',s_1',\mf{v'}} \mf{w}) \cdot r(\mf{w} \xra{i,s_1,\mf{v}} \mf{z}) ~\mbox{if}~ x_{i+s_1}<z_{i'};$$
\n (iv) sliding over a crossing:
\begin{gather*}
\rho(\mf{x} \xra{i'} \mf{x}) \cdot r(\mf{x} \xra{i,s_1,\mf{v}} \mf{y})=r(\mf{x} \xra{i,s_1,\mf{v}} \mf{y}) \cdot \rho(\mf{y} \xra{i'+1} \mf{y}) ~\mbox{if}~ i' \in \{i,\dots,i+s_1-1\}, s_1>0.
\end{gather*}

\vspace{.1cm}
\n (B) The differential $d=d_0+d_1$ is defined on the generators by the resolutions of crossings and markings on the corresponding decorated rook diagrams.
%\begin{align*}
%d(r(\mf{x} \xra{i,s_1,\mf{v}} \mf{y}))  = & \sum\limits_{s=0}^{s_1-1} ( r(\mf{x} \xra{i,s,\mf{v}^1(s)} \mf{z}^{s}) \cdot r(\mf{z}^{s} \xra{i+s+1,s_1-s-1,\overline{\mf{v}^1(s)}} \mf{y}) \cdot \rho(\mf{y} \xra{i+s+1} \mf{y}) \\
%& + \rho(\mf{x} \xra{i+s} \mf{x}) \cdot r(\mf{x} \xra{i,s,\mf{v}^1(s)} \mf{z}^{s}) \cdot r(\mf{z}^{s} \xra{i+s+1,s_1-s-1,\overline{\mf{v}^1(s)}} \mf{y}) ) \\
%& + \sum\limits_{t=1}^{s_0} r(\mf{x} \xra{i+f(t),s_1-f(t),\overline{\mf{v}^0(t)}} \mf{w}^{t}) \cdot r(\mf{w}^{t} \xra{i,f(t),\mf{v}^0(t)} \mf{y})
%\end{align*}
In general, the differential is extended by Leibniz's rule: $d(r_1\cdot r_2)=dr_1 \cdot r_2 + r_1 \cdot dr_2$ for $r_1, r_2 \in R_n$.

\vspace{.1cm}
\n (C) The grading $\op{deg}=(\deh, \dt)$ is defined on generators by:
\begin{gather*}
\op{deg}(e(\mf{x}))=(0,0), \\
\op{deg}(\rho(\mf{x} \xra{i} \mf{x}))=(-1,-1), \\
\op{deg}(r(\mf{x} \xra{i,s_1,\mf{v}} \mf{y}))=(1-s_1,1+s_0).
\end{gather*}
\end{defn}

\begin{rmk}
Relations (iii) are from isotopies of stackings of disjoint decorated rook diagrams.
\end{rmk}
\begin{figure}[h]
\begin{overpic}
[scale=0.2]{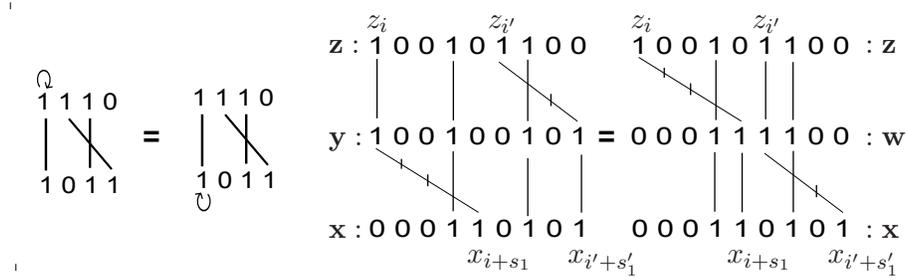}
\put(38,4){$\mf{x}:$}
\put(38,15){$\mf{y}:$}
\put(38,26){$\mf{z}:$}
\put(102,4){$:\mf{x}$}
\put(102,15){$:\mf{w}$}
\put(102,26){$:\mf{z}$}
\put(54.5,1){$x_{i+s_1}$}
\put(66.5,1){$x_{i'+s_1'}$}
\put(85.5,1){$x_{i+s_1}$}
\put(97.5,1){$x_{i'+s_1'}$}
\put(57,29){$z_{i'}$}
\put(42.5,29){$z_{i}$}
\put(88.5,29){$z_{i'}$}
\put(74,29){$z_{i}$}
\end{overpic}
\caption{The diagrams for the second and third relations in (iii) on the left and right, respectively.}
\label{4-3-0}
\end{figure}

\begin{lemma}
$d$ is well-defined and is a differential on $R_n$.
\end{lemma}
\begin{proof}
We use the geometric description of $d$ in terms of resolving crossings and markings.

\n (1) Well-definition.
We show that $d$ is well-defined under the relations of $R_n$.
Since the commutativity relations correspond to isotopies of stackings of disjoint rook diagrams, their resolutions commute as well.
The pictorial proof of the invariance of the differential under the sliding relation (Relation (iv)) is given in Figure \ref{4-3-1}.
Resolutions of both sides are obviously the same under the sliding relation except for the resolution of the crossing over which the loop slides.
\begin{figure}[h]
\begin{overpic}
[scale=0.25]{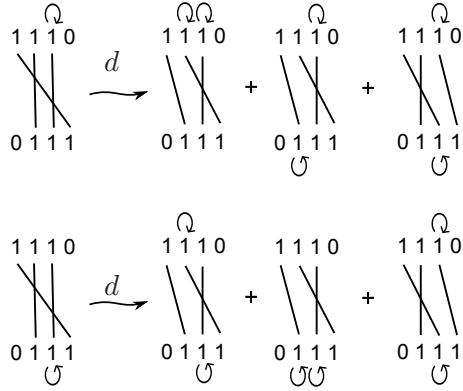}
\end{overpic}
\put(-135,120){$d$}
\put(-135,35){$d$}
\caption{Differentials of both sides in the sliding relation}
\label{4-3-1}
\end{figure}
%\begin{align*}
%& d(\rho(\mf{x} \xra{i'} \mf{x}) \cdot r(\mf{x} \xra{i,s_1,\mf{v}} \mf{y}))+ d(r(\mf{x} \xra{i,s_1,\mf{v}} \mf{y}) \cdot \rho(\mf{y} \xra{i'+1} \mf{y}))\\
%= & \rho(\mf{x} \xra{i'} \mf{x}) \cdot d(r(\mf{x} \xra{i,s_1,\mf{v}} \mf{y})) + d(r(\mf{x} \xra{i,s_1,\mf{v}} \mf{y})) \cdot \rho(\mf{y} \xra{i'+1} \mf{y}) \\
%= & \rho(\mf{x} \xra{i'} \mf{x}) \cdot r(\mf{x} \xra{i,i'-i,\mf{v}^1(i'-i)} \mf{z}^{i'-i}) \cdot r(\mf{z}^{i'-i} %\xra{i+i'-i+1,s_1-i'+i-1,\overline{\mf{v}^1(i'-i)}} \mf{y}) \cdot \rho(\mf{y} \xra{i+i'-i+1} \mf{y}) \\
%& + \rho(\mf{x} \xra{i'} \mf{x}) \cdot \rho(\mf{x} \xra{i+i'-i} \mf{x}) \cdot r(\mf{x} \xra{i,i'-i,\mf{v}^1(i'-i)} \mf{z}^{i'-i}) \cdot r(\mf{z}^{i'-i} %\xra{i+i'-i+1,s_1-i'+i-1,\overline{\mf{v}^1(i'-i)}} \mf{y}) \\
%& + r(\mf{x} \xra{i,i'-i,\mf{v}^1(i'-i)} \mf{z}^{i'-i}) \cdot r(\mf{z}^{i'-i} \xra{i+i'-i+1,s_1-i'+i-1,\overline{\mf{v}^1(i'-i)}} \mf{y}) \cdot \rho(\mf{y} %\xra{i+i'-i+1} \mf{y}) \cdot \rho(\mf{y} \xra{i'+1} \mf{y}) \\
%& + \rho(\mf{x} \xra{i+i'-i} \mf{x}) \cdot r(\mf{x} \xra{i,i'-i,\mf{v}^1(i'-i)} \mf{z}^{i'-i}) \cdot r(\mf{z}^{i'-i} %\xra{i+i'-i+1,s_1-i'+i-1,\overline{\mf{v}^1(i'-i)}} \mf{y}) \cdot \rho(\mf{y} \xra{i'+1} \mf{y}) \\
%= & 2~ \rho(\mf{x} \xra{i'} \mf{x}) \cdot r(\mf{x} \xra{i,i'-i,\mf{v}^1(i'-i)} \mf{z}^{i'-i}) \cdot r(\mf{z}^{i'-i} %\xra{i+i'-i+1,s_1-i'+i-1,\overline{\mf{v}^1(i'-i)}} \mf{y}) \cdot \rho(\mf{y} \xra{i'+1} \mf{y}) \\
%= & 0.
%\end{align*}

\n (2) Verification that $d$ is a differential.
It is easy to check that $d$ is of degree $(1,0)$.
We have to show that $d^2(r)=0$ for any generator $r \in R_n$.
In the expansion of $d^2(r)$, any term comes from a resolution of two of crossings and markings.
If $r$ has at least two crossings or markings, then the coefficient of each term is even since there are two ways to resolve them depending on the different orders of resolutions.
Hence, $d^2=0$ since we are working in $\F$.
If $r$ has only one crossing or marking, then $dr$ has no crossings or markings and $d(dr)=0$.
\end{proof}

\begin{lemma} \label{H(R_n) proof}
The cohomology $H(R_n)$ is generated by idempotents $e(\mf{x})$, loops $\rho(\mf{x} \xra{i} \mf{x})$ and arrows $r(\mf{x} \xra{i,s_1,\mf{v}} \mf{y})$ without crossings or markings, i.e., $s_1 = s_0(\mf{v})=0$.
\end{lemma}
\begin{proof}
It is easy to see that $R_n=\bigoplus\limits_{\mf{x},\mf{y} \in \bn}R_n(\mf{x},\mf{y})$, where $R_n(\mf{x},\mf{y})$ is the subspace of $R_n$ generated by all the arrows from $\mf{x}$ to $\mf{y}$.
It suffices to prove the lemma for $R_n(\mf{x},\mf{y})$.
Since the differential $d=d_0+d_1$ can be decomposed into two differentials, we have a double complex
$C=\bigoplus\limits_{p,q} R_n(\mf{x},\mf{y})_{p,q},$
where $R_n(\mf{x},\mf{y})_{p,q}$ is the subspace of $R_n(\mf{x},\mf{y})$ generated by all $r(\mf{x} \xra{i,s_1,\mf{v}} \mf{y})$ with $s_1=p, s_0=q$, and the horizontal and vertical differentials are $d_1$ and $d_0$, respectively.
$$\xymatrix{
R_n(\mf{x},\mf{y})_{0,2} \ar[d]^{d_0} & R_n(\mf{x},\mf{y})_{1,2} \ar[l]_{d_1} \ar[d]^{d_0} & \cdots \ar[l]_-{d_1} \ar[d]^{d_0} \\
R_n(\mf{x},\mf{y})_{0,1} \ar[d]^{d_0} & R_n(\mf{x},\mf{y})_{1,1} \ar[l]_{d_1} \ar[d]^{d_0} & \cdots \ar[l]_-{d_1} \ar[d]^{d_0}\\
R_n(\mf{x},\mf{y})_{0,0} & R_n(\mf{x},\mf{y})_{1,0} \ar[l]_{d_1} & \cdots \ar[l]_-{d_1}
}$$
Then the differential in the total complex $Tot(C)$ is $d=d_0+d_1$ in $R_n$.
Note that the double complex $C$ is finite since there are at most $n$ crossings and markings in any arrow.
Consider the two spectral sequences of $C$ from the two filtrations which converge to the homology of $Tot(C)$ \cite[Section 5.6]{Weibel}.
Let $E^1_{p,q}=H_q^v(C_{p,*})$ and $'E^1_{p,q}=H_p^h(C_{*,q})$ be the first pages by taking the homology of the vertical differential $d_0$ and the horizontal differential $d_1$, respectively.
We will show that $$E^1_{p,q}=0 ~\mbox{for}~ q>0; \qquad 'E^1_{p,q}=0 ~\mbox{for}~ p>0.$$
Therefore, $H_{p+q}Tot(C)=0$ for $p+q>0$, i.e., $H(R_n)$ is generated by loops and arrows without crossings or markings.

\vspace{.2cm}
For $E^1_{pq}$ with $q>0$, suppose $d_0r=0$, where $r=\sum\limits_{i}r_i \in C_{p,q}$ of a finite sum and each $r_i$ is a product of elementary decorated rook diagrams given by a path $\phi_i$ in $Q_n$:
\begin{gather*}
\mf{x} \ra \mf{z}^{i_1} \ra \cdots \ra \mf{z}^{i_{j(i)}} \ra \mf{y}.
\end{gather*}
Assume further that $r$ is primitive, i.e., any nontrivial partial sum of $\sum\limits_{i}r_i$ is non-closed.
A path consisting of elementary decorated rook diagrams is called a {\em composition path}.

The key observation is that $d_0$ only locally changes a decorated rook diagram by replacing the marking $\alpha$ with the gap $\beta$, as shown in Figure \ref{4-3-2}.
For each $r_i(s)$ in $d_0r_i=\sum\limits_{s}r_i(s)$, the corresponding path $\phi_{i(s)}$ is given by
inserting one vertex of $Q_n$ to $\phi_i$.
A composition path $\phi$ is called a {\em quotient-path} of $\phi'$ if the set of vertices $V(\phi)$ in $\phi$ is a proper subset of $V(\phi')$.
In particular, $\phi_i$ is a quotient-path of $\phi_{i(s)}$ for all $s$.
The goal is to find a universal path $\phi_r$ for $r$ which is a quotient-path of $\phi_i$ for all $i$.

We start from $r_1(1)$ which must be equal to $r_j(s)$ as a summand of $d_0r_j$ for some $j\neq 1$ since $r$ is $d_0$-closed. Without loss of generality, assume $r_1(1)=r_2(1)$.
The composition path for any element in $R_n$ is uniquely determined up to isotopies of disjoint diagrams and slides of loops over crossings.
We choose a composition path $\phi_{1(1)}$ for $r_1(1)$ which in turn determines the composition pathes $\phi_1$ for $r_1$ and $\phi_{1(s)}$ for $r_1(s)$ for $s>1$.
Since $r_2(1)=r_1(1)$, we choose $\phi_{2(1)}=\phi_{1(1)}$ as the composition path for $r_2(1)$ which determines the composition pathes $\phi_2$ for $r_2$ and $\phi_{2(s)}$ for $r_2(s)$ for $s>1$..
Then both $\phi_1$ and $\phi_2$ are quotient-pathes of $\phi_{1(1)}=\phi_{2(1)}$.
Let $\phi_{1,2}$ be the unique composition path such that $V(\phi_{1,2})=V(\phi_1) \cap V(\phi_2)$.

We repeat the same procedure for $r_1(2)$. If $r_1(2)=r_2(s)$ for some $s$, then we move on to $r_1(3)$.
Otherwise, assume $r_1(2)=r_3(1)$.
We choose $\phi_{1(2)}$ as the composition path for $r_3(1)$.
Then $\phi_3$ and $\phi_{1,2}$ are both quotient-pathes of $\phi_{1(2)}$ since $\phi_{1}$ is a quotient path of $\phi_{1(2)}$.
Let $\phi_{1,2,3}$ be the unique composition path such that $V(\phi_{1,2,3})=V(\phi_{1,2}) \cap V(\phi_3)$.
Since we assume $r$ is primitive, by iterating this procedure we can finally find a universal path $\phi_{r}$ for $r$ such that $\phi_r$ is a quotient-path of $\phi_i$ for all $i$.
In other words, there exists a decorated rook diagram $\g(r)$ corresponding to $\phi_r$ and a finite set $\cal{I}$ indexed by all markings in $\g(r)$ such that each $r_i$ is represented by a diagram $\g(r_i)$ which is obtained from $\g(r)$ by changing some of the markings to gaps.
\begin{figure}[h]
\begin{overpic}
[scale=0.2]{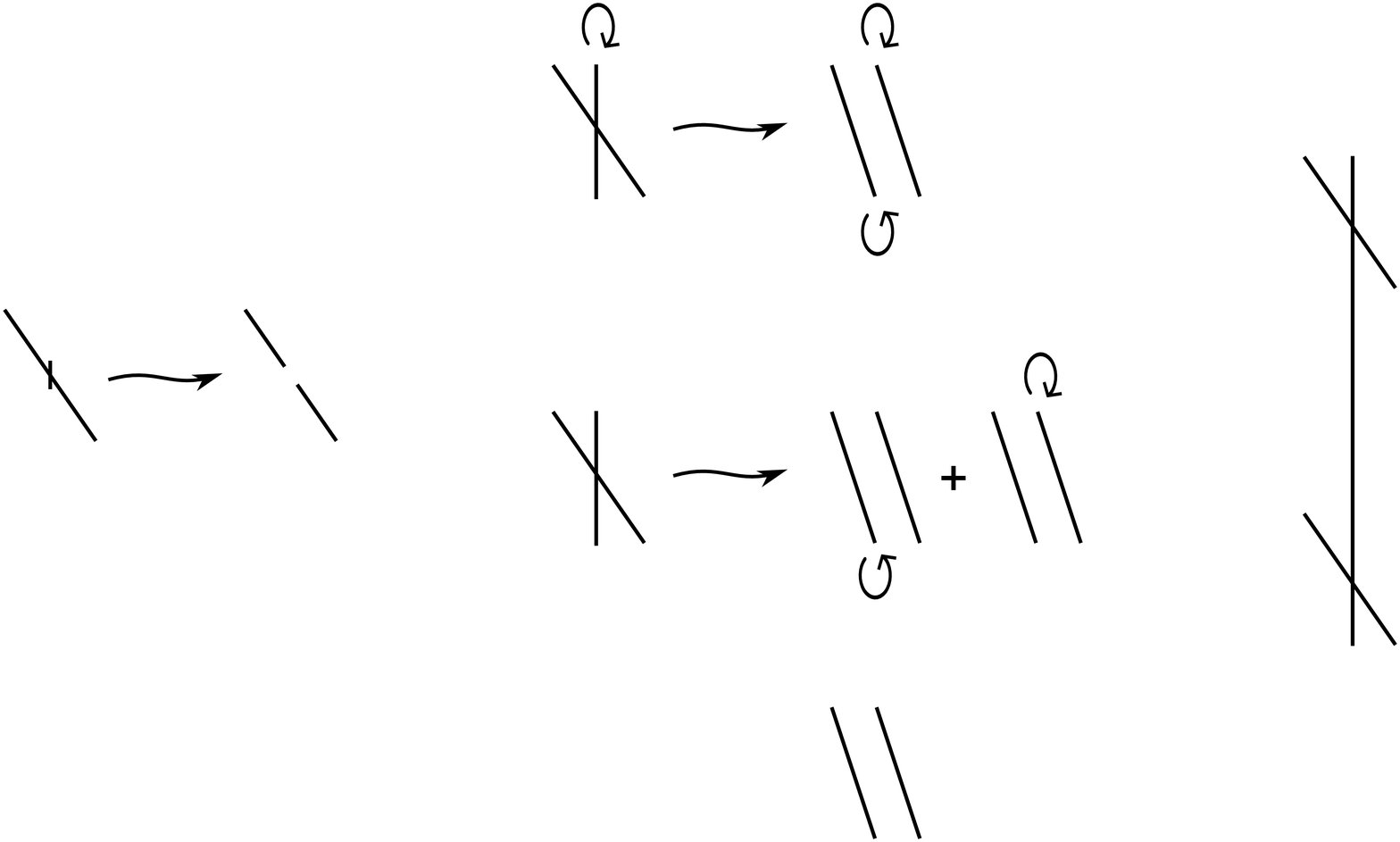}
\end{overpic}
\put(-200,80){$d_0$}
\put(-215,50){$\alpha$}
\put(-175,50){$\beta$}
\put(-110,115){$d_1$}
\put(-130,85){$\alpha_1$}
\put(-80,85){$\beta_1$}
\put(-110,65){$d_1$}
\put(-130,30){$\alpha_2$}
\put(-85,30){$\beta_2$}
\put(-55,30){$\beta_3$}
\put(-85,-10){$\beta_4$}
\caption{The left picture is $W_0$; the middle is $W_1$; the right is $W_1 \ot_S W_1$.}
\label{4-3-2}
\end{figure}

We construct a {\em chain complex of markings} $C_0(r)$ for $r \in R_n$ such that $d_0(r)=0$ as follows.
Let $$W_0=\lan \alpha\ran \xra{d_0} \lan \beta \ran$$ be a chain complex of $\F$-vector spaces generated by a marking $\alpha$ and a gap $\beta$, where $\lan \alpha \ran$ and $\lan \beta \ran$ are in degree $1$ and $0$, respectively.
The differential $d_0$ resolves a marking $\alpha$ and yields a gap $\beta$.
The homology of $(W_0, d_0)$ is zero.
Note that $(W_0, d_0)$ is the local model for one marking in $\g(r)$.
Define the chain complex of markings $C_0(r)$ by $|\cal{I}|$-th tensor product $W_0^{\ot |\cal{I}|}$ of $W_0$ over $\F$.
In other words, $C_0(r)$ encodes the information of all markings in $\g(r)$.
Then each $r_i \in R_n$ corresponds to a generator in the chain complex.
The differential $d_0$ in $R_n$ corresponds to the differential in $W_0^{\ot |\cal{I}|}$.

We compute the homology of $W_0^{\ot |\cal{I}|}$ in the following.
Recall the K\"unneth formula from \cite[Theorem 3.6.3]{Weibel}:
If $P$ and $Q$ are right and left complexes of $R$-modules such that $P_n$ and $d(P_n)$ are flat for each $n$, then there is an exact sequence
$$0 \ra \bigoplus\limits_{p+q=n}H_p(P)\ot H_q(Q) \ra H_n(P \ot_{R} Q) \ra \bigoplus\limits_{p+q=n-1}Tor_1^{R}(H_p(P),H_q(Q)) \ra 0.$$
The homology of $(W_0^{\ot |\cal{I}|}, d_0)$ is zero by taking $R=\F$ and $P=Q=W_0$.
It is easy to see that $r=\sum\limits_{i}r_i$ corresponds a closed element in $W^{\ot |\cal{I}|}$.
Then there exists another element $w \in W^{\ot |\cal{I}|}$ such that $d_0(w)=r$ since the homology of $(W_0^{\ot |\cal{I}|}, d_0)$ is zero.
Hence, there exists a corresponding element $w$ in $R_n$ such that $d_0(w)=r \in R_n$.

\vspace{.2cm}
For $'E^1_{pq}$, the proof is similar to that for $E^1_{pq}$.
A key difference is that the collection of local diagrams consists of $6$ patterns $\{\alpha_i; \beta_j ~|~ i=1,2; j=1,2,3,4\}$ as in Figure \ref{4-3-2}.
Let $(W_1, d_1)$ be the chain complex of $\F$-vector spaces given by locally resolving a crossing:
$$\ba{ccc}
\lan \alpha_1, \alpha_2 \ran & \xra{d_1} & \lan \beta_1, \beta_2, \beta_3, \beta_4 \ran \\
\alpha_1 & \mapsto & \beta_1, \\
\alpha_2 & \mapsto & \beta_2+\beta_3,
\ea$$
where $\alpha_i$'s are in degree $1$ and $\beta_j$'s are in degree $0$.
Then $(W_1, d_1)$ is the local model for one crossing.
Notice that $(W_1, d_1)$ can be viewed as a chain complex of $O$-bimodules, where the ring $O=\lan 1, \rho ~|~ \rho^2=0 \ran$ acts on $W_1$ by adding a loop $\rho$ to the decorated rook diagrams.

Similarly, we construct a {\em chain complex of crossings} $C_1(r')$ as a tensor product of $W_1$'s for any element $r' \in R_n$ such that $d_1(r')=0$.
The chain complex $C_1(r')$ is supposed to encode the information of all crossings in $\g'$ associated to $r'$.
There are two types of tensor products in $C_1(r')$.
The first one is a tensor product of two $W_1$'s over $O$ if the corresponding two crossings can be connected by a vertical strand as in Figure \ref{4-3-2} since the loop $\rho$ could slide over a crossing and along a vertical strand.
Otherwise, we use a tensor product over $\F$.

It is easy to verify the conditions for $W_1 \ot_O W_1$ and $W_1 \ot_{\F} W_1$ in the K\"unneth formula.
The homology $H_1(W_1)$ is zero at degree $1$ and $H_0(W_1)$ is isomorphic to the ring $O$.
Hence, $H_0(W_1)$ is free as left and right $K$ modules and the Tor group in the K\"unneth formula vanishes.
We have $$H_n(W_1 \ot_R W_1) \cong \bigoplus\limits_{p+q=n}H_p(W_1)\ot_R H_q(W_1),$$ where $R$ is either $O$ or $\F$ depending on the type of the tensor product.
It follows that the homology of $C_1(r')$ is zero at degree greater than $0$.
Hence, $'E^1_{pq}=0$ for $p>0$ and we conclude the proof.
\end{proof}

\begin{rmk}
The first page $E^1_{pq}$ given by the differential $d_0$ for resolving markings is very close to the strands algebra $\cal{A}(2n)$.
They only differ at the relation of a double crossing which is set to be zero in the strands algebra.
\end{rmk}

Let $r(\mf{x} \xra{i} \mf{y})$ denote the class $[r(\mf{x} \xra{i,s_1,\mf{v}} \mf{y})]$ with $s_1=s_0(\mf{v})=0$ in the cohomology $H(R_n)$.

\begin{prop} \label{H(R_n)}
The cohomology $H(R_n)$ is an associative $t$-graded DG algebra with a trivial differential.
It has idempotents $e(\mf{x})$ for each vertex $\mf{x}$ in $\gn$, generators $\rho(\mf{x} \xra{i} \mf{x})$ for each loop $\mf{x} \xra{i} \mf{x}$ and $r(\mf{x} \xra{i} \mf{y})$ for each arrow $\mf{x} \xra{i,s_1,\mf{v}} \mf{y}$ with $s_1 + s_0(\mf{v})=0$ in $\gn$.
The relations consist of $4$ groups:

\n (i) idempotents:
\begin{gather*}
e(\mf{x}) \cdot e(\mf{y})=\delta_{\mf{x},\mf{y}}\cdot e(\mf{x}) ~\mbox{for all}~ \mf{x}, \mf{y}, \\
e(\mf{x}) \cdot \rho(\mf{x} \xra{i} \mf{x})=\rho(\mf{x} \xra{i} \mf{x}) \cdot e(\mf{x})= \rho(\mf{x} \xra{i} \mf{x}) ~\mbox{for all}~ \rho(\mf{x} \xra{i} \mf{x}), \\
e(\mf{x}) \cdot r(\mf{x} \xra{i} \mf{y})=r(\mf{x} \xra{i} \mf{y}) \cdot e(\mf{y})= r(\mf{x} \xra{i} \mf{y}) ~\mbox{for all}~ r(\mf{x} \xra{i} \mf{y});
\end{gather*}
\n (ii) unstackability relations (R1):
\begin{gather} \label{R1}\tag{R1}
\rho(\mf{x} \xra{i} \mf{x}) \cdot \rho(\mf{x} \xra{i} \mf{x})=0, \\
r(\mf{x} \xra{i} \mf{y}) \cdot r(\mf{y} \xra{i} \mf{z})=0 \tag{R1};
\end{gather}
\n (iii) commutativity relations (R2):
\begin{gather*} \label{R2}\tag{R2}
\rho(\mf{x} \xra{i'} \mf{x}) \cdot \rho(\mf{x} \xra{i} \mf{x})=\rho(\mf{x} \xra{i} \mf{x}) \cdot \rho(\mf{x} \xra{i'} \mf{x}) ~\mbox{if}~ i'\neq i, \\
\rho(\mf{x} \xra{i'} \mf{x}) \cdot r(\mf{x} \xra{i} \mf{y})=r(\mf{x} \xra{i} \mf{y}) \cdot \rho(\mf{y} \xra{i'} \mf{y}) ~\mbox{if}~ i' \neq i, \tag{R2}\\
r(\mf{x} \xra{i} \mf{y}) \cdot r(\mf{y} \xra{i'} \mf{z})=r(\mf{x} \xra{i'} \mf{w}) \cdot r(\mf{w} \xra{i} \mf{z}) ~\mbox{if}~ x_{i}<z_{i'};\tag{R2}
\end{gather*}
\n (iv) relation (R3) from the differential of a crossing:
\begin{gather*} \label{R3}\tag{R3}
\rho(\mf{x} \xra{i} \mf{x}) \cdot r(\mf{x} \xra{i} \mf{y}) \cdot r(\mf{y} \xra{i+1} \mf{z})=r(\mf{x} \xra{i} \mf{y}) \cdot r(\mf{y} \xra{i+1} \mf{z}) \cdot \rho(\mf{z} \xra{i+1} \mf{z}) ~\mbox{if}~ z_{i+1}=x_i.
\end{gather*}
\end{prop}

The $t$-graded DG algebra $R_n$ is formal since its cohomology $H(R_n)$ is concentrated along the line $\deh-\dt=0$ by Lemma \ref{H(R_n) proof}.
\begin{lemma} \label{quasirn}
The $t$-graded DG algebra $R_n$ is quasi-isomorphic to its cohomology $H(R_n)$.
\end{lemma}
\proof
A quasi-isomorphism is given by
$$\begin{array}{ccccc}
g_n: & R_n & \ra & H(R_n) &\\
& e(\mf{x}) & \mapsto & e(\mf{x});&\\
& \rho(\mf{x} \xra{i} \mf{x}) & \mapsto & \rho(\mf{x} \xra{i} \mf{x});&\\
& r(\mf{x} \xra{i,s_1,\mf{v}} \mf{y}) & \mapsto & r(\mf{x} \xra{i} \mf{y})& \mbox{if}~ s_1=s_0(\mf{v})=0 ;\\
& r(\mf{x} \xra{i,s_1,\mf{v}} \mf{y}) & \mapsto & 0& \mbox{otherwise}. \qed
\end{array}$$

\begin{defn}
(1) Let $DGP(R_n)$ be the smallest full subcategory of $DG(R_n)$ which contains the projective DG $R_n$-modules $\{P(\mf{x})=R_n \cdot e(\mf{x}) ~|~ \mf{x} \in \bn \}$ and is closed under the cohomological grading shift functor $[1]$, the $t$-grading shift functor $\{1\}$ and taking mapping cones.

\n(2) Let $DGP(H(R_n))$ be the smallest full subcategory of $DG(H(R_n))$ which contains the projective DG $R_n$-modules $\{PH(\mf{x})=H(R_n) \cdot e(\mf{x}) ~|~ \mf{x} \in \bn \}$ and is closed under the cohomological grading shift functor $[1]$, the $t$-grading shift functor $\{1\}$ and taking mapping cones.
\end{defn}

Let $HP(R_n)$ and $HP(H(R_n))$ denote $0$th homology categories of $DGP(R_n)$ and $DGP(H(R_n))$, respectively.
Then $HP(R_n)$ and $HP(H(R_n))$ are triangulated categories.
Since $R_n$ is formal, it is easy to see the following equivalence of triangulated categories.

\begin{lemma} \label{k0 vn}
The triangulated categories $HP(R_n)$ and $HP(H(R_n))$ are equivalent.
Hence there are isomorphisms of $\zt$-modules:
$K_0(HP(R_n)) \cong K_0(HP(H(R_n))) \cong \zt\lan \bn \ran \cong V_1^{\ot n}.$
\end{lemma}

\subsection{The $t$-graded DG algebra $A \bt R_n$}
Ideally, we want to use $DGP(A \ot R_n)$ of DG projective $A \ot R_n$-modules to categorify $\ut \ot V_1^{\ot n}$ and construct a $(R_n, A \ot R_n)$-bimodule to categorify the $\ut$-action.
But this ideal approach does not work.
We have to modify $A \ot R_n$ to a DG algebra $A \bt R_n$ by adding an extra differential which enables us to construct the DG $(H(R_n), A \bt R_n)$-bimodule $C_n$ in Section 6.
We show that $\aor$ is formal, hence it is quasi-isomorphic to $A \ot H(R_n)$.
The definition of $A \bt R_n$ is rather technical and the reader can pretend it is $A \ot R_n$ at a first reading.

\begin{defn} \label{arn}
$A \bt R_n$ is an associative $t$-graded DG $\F$-algebra with a differential $d$ and a grading $\op{deg}=(\deh, \dt) \in \Z^2$.

\vspace{.2cm}
\n (A) $A \bt R_n$ has generators of $3$ types:

(1) $e(\g) \bt r$ and $a \bt e(\mf{x}),$
for $\g \in \cal{B}, r \in R_n, a \in A, \mf{x} \in \bn$;

(2) $\rho(I\mf{x} \xra{i} EF\mf{x}),$
for $1 \leq i \leq k<n$ and $\mf{x} \in \bnk$ such that $x_{i}=n-k+i$;

(3) $\rho(EF\mf{x} \xra{j} I\mf{x}),$
for $1 \leq j \leq k<n$ and $\mf{x} \in \bnk$ such that $x_j=j$.

\vspace{.2cm}
\n (B) The relations consist of $5$ groups:
\be
\item Relations from $A$ and $R_n$: for $\g \in \cal{B}, r_1,r_2 \in R_n, a_1,a_2 \in A, \mf{x} \in \bn$,
\begin{gather*}
e(\g) \bt (r_1+r_2)=e(\g) \bt r_1 + e(\g) \bt r_2, \quad e(\g) \bt (r_1r_2)=(e(\g)\bt r_1)\cdot (e(\g)\bt r_2);\\
(a_1+a_2) \bt e(\mf{x})=a_1 \bt e(\mf{x}) +a_2 \bt e(\mf{x}), \quad (a_1a_2) \bt e(\mf{x})=(a_1\bt e(\mf{x}))\cdot (a_2\bt e(\mf{x})).
\end{gather*}

\item Commutativity relation from $A \ot R_n$: $(a \bt e(\mf{x})) \cdot (e(\g_2) \bt r)=(e(\g_1) \bt r) \cdot (a \bt e(\mf{y})),$
for $e(\g_1) \cdot a \cdot e(\g_2)=a \in A, ~e(\mf{x}) \cdot r \cdot e(\mf{y})=r \in R_n$ except that
\begin{gather*}\label{$**$}\tag{$**$}
(a, r) = (\rho(I,EF), \rho(\mf{x} \xra{k} \mf{x})) ~\mbox{if}~ x_k=n;\qquad \mbox{or}\quad(\rho(EF,I), \rho(\mf{x} \xra{1} \mf{x})) ~\mbox{if}~ x_1=1.
\end{gather*}

\item Relations for $\rho(I\mf{x} \xra{i} EF\mf{x})$:
\begin{gather*}
(e(I) \bt e(\mf{x})) \cdot \rho(I\mf{x} \xra{i} EF\mf{x})=\rho(I\mf{x} \xra{i} EF\mf{x}) \cdot (e(EF) \bt e(\mf{x}))=\rho(I\mf{x} \xra{i} EF\mf{x});\\
\rho(I\mf{x} \xra{i} EF\mf{x}) \cdot (e(EF) \bt \rho(\mf{x} \xra{i'} \mf{x}))=(e(I) \bt \rho(\mf{x} \xra{i'} \mf{x})) \cdot \rho(I\mf{x} \xra{i} EF\mf{x}) ~\mbox{if}~ i\neq i'+1; \tag{$*$}\label{$*$}\\
\rho(I\mf{x} \xra{i} EF\mf{x}) \cdot (e(EF) \bt r(\mf{x} \xra{i',s_1,\mf{v}} \mf{y}))=(e(I) \bt r(\mf{x} \xra{i',s_1,\mf{v}} \mf{y})) \cdot \rho(I\mf{y} \xra{i} EF\mf{y}) ~\mbox{if}~ i'+s_1<i;\\
\rho(I\mf{x} \xra{i} EF\mf{x}) \cdot (\rho(EF,I) \bt e(\mf{x}))=0.
\end{gather*}

\item Relations for $\rho(EF\mf{x} \xra{j} I\mf{x})$:
\begin{gather*}
(e(EF) \bt e(\mf{x})) \cdot \rho(EF\mf{x} \xra{j} I\mf{x})=\rho(EF\mf{x} \xra{j} I\mf{x}) \cdot (e(I) \bt e(\mf{x}))=\rho(EF\mf{x} \xra{j} I\mf{x});\\
\rho(EF\mf{x} \xra{j} I\mf{x}) \cdot (e(I) \bt \rho(\mf{x} \xra{j'} \mf{x}))=(e(EF) \bt \rho(\mf{x} \xra{j'} \mf{x})) \cdot \rho(EF\mf{x} \xra{j} I\mf{x}) ~\mbox{if}~ j\neq j'-1;\\
\rho(EF\mf{x} \xra{j} I\mf{x}) \cdot (e(I) \bt r(\mf{x} \xra{j',s_1,\mf{v}} \mf{y}))=(e(EF) \bt r(\mf{x} \xra{j',s_1,\mf{v}} \mf{y})) \cdot \rho(EF\mf{x} \xra{j} I\mf{x}) ~\mbox{if}~ j<j'.
\end{gather*}

\item Relations for $\rho(I\mf{x} \xra{i} EF\mf{x})$ and $\rho(EF\mf{x} \xra{j} I\mf{x})$:
\begin{gather*}
\rho(I\mf{x} \xra{i} EF\mf{x}) \cdot \rho(EF\mf{x} \xra{j} I\mf{x})=0.
\end{gather*}
\ee

\n (C) The differential is defined on generators in the following and extended by the Leibniz rule.
\begin{align}
d(a \bt e(\mf{x}))=&0; \tag{1}\label{1}\\
d(e(\g) \bt r)=&e(\g) \bt d(r); \tag{2}\label{2}\\
d(\rho(I\mf{x} \xra{k} EF\mf{x}))=&(\rho(I,EF)\bt e(\mf{x}))\cdot(e(EF) \bt \rho(\mf{x} \xra{k} \mf{x}))  \tag{3}\label{3}\\
&+ (e(I) \bt \rho(\mf{x} \xra{k} \mf{x}))\cdot(\rho(I,EF)\bt e(\mf{x})); \notag  \\
d(\rho(I\mf{x} \xra{i} EF\mf{x}))=&\rho(I\mf{x} \xra{i+1} EF\mf{x})\cdot(e(EF) \bt \rho(\mf{x} \xra{i} \mf{x})) \tag{4}\label{4} \\
&+ (e(I) \bt \rho(\mf{x} \xra{i} \mf{x}))\cdot(\rho(I\mf{x} \xra{i+1} EF\mf{x})) ~\mbox{for}~ i<k; \notag \\
d(\rho(EF\mf{x} \xra{1} I\mf{x}))=&(\rho(EF,I)\bt e(\mf{x}))\cdot(e(I) \bt \rho(\mf{x} \xra{1} \mf{x})) \tag{5}\label{5}\\
&+ (e(EF) \bt \rho(\mf{x} \xra{1} \mf{x}))\cdot(\rho(EF,I)\bt e(\mf{x}));\notag
\end{align}
\begin{align}
d(\rho(EF\mf{x} \xra{j} I\mf{x}))=&\rho(EF\mf{x} \xra{j-1} I\mf{x}) \cdot (e(I) \bt \rho(\mf{x} \xra{j} \mf{x})) \tag{6}\label{6}\\
&+ (e(EF) \bt \rho(\mf{x} \xra{j} \mf{x}))\cdot \rho(EF\mf{x} \xra{j-1} I\mf{x}) ~\mbox{for}~ j>1. \notag
\end{align}

\n (D) The grading $\op{deg}=(\deh, \dt)$ is defined for $a \in A, r \in R_{n,k}$ and $\mf{x} \in \bnk$ by:
\begin{align*}
\deh(a \bt r)&= \deh(a) + \deh(r) + 2k\dt(a), \\
\dt(a\bt r)&=n\dt(a)+\dt(r), \\
 \op{deg}(\rho(I\mf{x} \xra{i} EF\mf{x}))&=(-2(k-i+1),-(k-i+1)), \\
 \op{deg}(\rho(EF\mf{x} \xra{j} I\mf{x}))&=(2k+1-2j,n-j).
\end{align*}
\end{defn}

\begin{rmk}
(1)  Equations (\ref{1}) and (\ref{2}) in part (C) are inherited from the differential on $A\ot R_n$.
Equations (\ref{3}) and (\ref{5}) correspond to the conditions in Equations (\ref{$**$}) in part (B-2) where the commutativity relation fails.
Equations (\ref{4}) and (\ref{6}) are higher homotopies.

\n(2) The condition $x_{i}=n-k+i$ for $\rho(I\mf{x} \xra{i} EF\mf{x})$, where $\mf{x} \in \bnk$ implies that $x_{j}=n-k+j$ for all $i \leq j \leq k$.
In other words, all states between $n-k+i$th state and the last state are $|1\ran$.
Similarly, the condition $x_{j}=j$ for $\rho(EF\mf{x} \xra{j} I\mf{x})$ implies that all states between the first state and $j$th state are $|1\ran$.
See the proof of Lemma \ref{haor} for an example.

\n(3) The new ingredients $\rho(I\mf{x} \xra{i} EF\mf{x}), \rho(EF\mf{x} \xra{j} I\mf{x})$ are only used in the construction of the $(H(R_n), \aor)$-bimodules $C_n$ in Section 6.4.
A contact topological interpretation is still missing.
\end{rmk}

\begin{lemma}
$d$ is well-defined: (1) $d$ preserves the relations; (2) $d^2=0$; and (3) $\op{deg}(d)=(1,0)$.
\end{lemma}
\begin{proof}
There is a decomposition of the differential $d=d_1+d_2$, where $d_1$ is defined by Equations (1) and (2) and extended to other cases by zero.
It is easy to see that $d_1^2=d_2^2=d_1d_2+d_2d_1=0$.
Since $d_1$ is inherited from the differential on $R_n$, it suffices to prove the lemma for $d_2$.

\vspace{.1cm}
\n(1) We verify that $d_2$ preserves Equation (\ref{$*$}) in Relation (B-3).
Fixing $\mf{x} \in \bnk$, we simplify the notation as
\begin{gather*}
\bar{\rho}_{k+1}:=\rho(I,EF)\bt e(\mf{x}), \quad \bar{\rho}_{i}:=\rho(I\mf{x} \xra{i} EF\mf{x}) ~~\mbox{if}~ x_i=n-k+i;\\
\rho(\g,i):=e(\g)\bt \rho(\mf{x} \xra{i} \mf{x}) ~~\mbox{for}~ \g=I,EF.
\end{gather*}
Then the differential is given by: $d_2(\bar{\rho}_{i})=\bar{\rho}_{i+1}\rho(EF,i)+\rho(I,i)\bar{\rho}_{i+1},$ for $i\leq k$.
Equation (\ref{$*$}) reads as: $\bar{\rho}_{i}\rho(EF,i')=\rho(I,i')\bar{\rho}_{i}$ if $i\neq i'+1$.
We apply $d_2$ to both sides:
\begin{align*}
d_2(\bar{\rho}_{i}\rho(EF,i'))&=\bar{\rho}_{i+1}\rho(EF,i)\rho(EF,i')+\rho(I,i)\bar{\rho}_{i+1}\rho(EF,i') ;\\
d_2(\rho(I,i')\bar{\rho}_{i})&=\rho(I,i')\bar{\rho}_{i+1}\rho(EF,i)+\rho(I,i')\rho(I,i)\bar{\rho}_{i+1}.
\end{align*}

\n(Case 1) If $i \neq i'$, then $\bar{\rho}_{i+1}\rho(EF,i')=\rho(I,i')\bar{\rho}_{i+1}$ from Equation (\ref{$*$}).
We have $$d_2(\bar{\rho}_{i}\rho(EF,i'))=\bar{\rho}_{i+1}\rho(EF,i)\rho(EF,i')+\rho(I,i)\rho(I,i')\bar{\rho}_{i+1}=d_2(\rho(I,i')\bar{\rho}_{i}).$$
(Case 2) If $i=i'$, then $\rho(EF,i)\rho(EF,i')=\rho(I,i')\rho(I,i)=0$ from Relation (B-1).
We have $$d_2(\bar{\rho}_{i}\rho(EF,i'))=\rho(I,i)\bar{\rho}_{i+1}\rho(EF,i)=d_2(\rho(I,i')\bar{\rho}_{i}).$$
In either case, we proved that $d_2$ preservers Equation (\ref{$*$}).

The proofs for other relations are similar and we leave it to the reader.

\vspace{.1cm}
\n(2) It suffices to show that $d_2^2(r)=0$ for any generator $r \in \aor$.
The computation is similar to that in (1) above.

\vspace{.1cm}
\n(3) By definition, $\op{deg}(e(\g)\bt \rho(\mf{x} \xra{j} \mf{x}))=(-1,-1)$ for $\g=I,EF$ from Definition \ref{rnk}.
For $\rho(I\mf{x} \xra{i} EF\mf{x})$, $i<k$ we have:
$$\op{deg}(d)=(-1,-1)+\op{deg}(\rho(I\mf{x} \xra{i+1} EF\mf{x}))-\op{deg}(\rho(I\mf{x} \xra{i} EF\mf{x}))=(-1,-1)+(2,1)=(1,0).$$
We have similar computations for other generators.
\end{proof}

We compute the cohomology $H(\aor)$ and show that $\aor$ is formal.
\begin{lemma} \label{haor}
The $t$-graded DG algebra $\aor$ is quasi-isomorphic to its cohomology $A \ot H(R_n)$.
\end{lemma}
\begin{proof}
(1) We first compute $H(\aor)$.
Note that $\aor$ is a finite double complex with respect to $d=d_1+d_2$ in previous lemma.
Since the cohomology $H_{d_1}(A \ot R_n)$ with respect to $d_1$ is $A \ot H(R_n)$, the following claim implies that the cohomology $H_d(\aor)$ is $A \ot H(R_n)$ .

\vspace{.1cm}
\n{\bf Claim:} the cohomology $H_{d_2}(\aor)$ with respect to $d_2$ is $A \ot R_n$.

\vspace{.1cm}
\n There is a decomposition
$$\aor=\bigoplus\limits_{\g_1,\g_2 \in \cal{B}, \mf{x},\mf{y} \in \bn}L(\g_1,\g_2; \mf{x},\mf{y})=\bigoplus\limits_{\g_1,\g_2 \in \cal{B}, \mf{x},\mf{y} \in \bn}(e(\g_1)\bt e(\mf{x}))\cdot\aor\cdot(e(\g_2)\bt e(\mf{y})).$$
Since $d_2$ is nontrivial only on $\rho(I\mf{x} \xra{i} EF\mf{x}), \rho(EF\mf{x} \xra{j} I\mf{x})$, it suffices to prove the claim for summands $L(\g_1,\g_2; \mf{x},\mf{y})$ where $(\g_1, \g_2)=(I, EF), (EF, I)$ and $\mf{x}=\mf{y}$.

We compute $H_{d_2}(L)$, where $L=L(\g_1,\g_2; \mf{x},\mf{y})$ for $(\g_1, \g_2)=(EF, I)$, $\mf{x}=\mf{y}=|1101\ran$.
By definition there exist $\bar{\rho}_2:=\rho(EF|1101\ran \xra{2} I|1101\ran)$ and $\bar{\rho}_1:=\rho(EF|1101\ran \xra{1} I|1101\ran)$. Let
$$\bar{\rho}_0:=\rho(EF,I)\bt e(|1101\ran); \quad \rho(\g, i):=e(\g) \bt \rho(|1101\ran \xra{i} |1101\ran),$$
for $\g=I, EF$ and $i=1,2,3$.
The nontrivial differential on the generators is given by:
$$d_2(\bar{\rho}_2)=\rho(EF,2)\bar{\rho}_1 + \bar{\rho}_1\rho(I,2), \qquad
d_2(\bar{\rho}_1)=\rho(EF,1)\bar{\rho}_0 + \bar{\rho}_0\rho(I,1).$$
From relations in (B-4), we have
\begin{align*}
\rho(EF,i)\bar{\rho}_2 = &\bar{\rho}_2\rho(I,i) ~~\mbox{for}~~i=1,2,3; \\
\rho(EF,i)\bar{\rho}_1 = &\bar{\rho}_1\rho(I,i) ~~\mbox{for}~~i=1,3; \\
\rho(EF,i)\bar{\rho}_0 = &\bar{\rho}_0\rho(I,i) ~~\mbox{for}~~i=2,3.
\end{align*}
Then $L$ is a complex $L_2 \xra{d_2} L_1 \xra{d_2} L_0$, where $L_i=(\aor)\bar{\rho}_i(\aor)$.
A direct computation shows that $L_2$ is a $8$-dimensional $\F$-vector space, $L_1$ and $L_0$ are $16$-dimensional $\F$-vector spaces.
Moreover, the complex is exact except at $L_0$ and $H_{d_2}(L)$ is isomorphic to $\rho(EF,I) \ot e(\mf{x})Ae(\mf{x})$ in $A\ot R_n$ for $\mf{x}=|1101\ran$.

The proof of the claim in general is similar and we leave it to the reader.

\vspace{.2cm}
\n(2) It is easy to see that the following map gives a quasi-isomorphism:
$$
\begin{array}{ccc}
\aor & \ra & A \ot R_n \\
 e(\g) \bt r & \mapsto & e(\g) \ot r, \\
a \bt e(\mf{x}) & \mapsto & a \ot e(\mf{x}),\\
\rho(I\mf{x} \xra{i} EF\mf{x})& \mapsto & 0,\\
\rho(EF\mf{x} \xra{i} I\mf{x}) & \mapsto & 0.
\end{array}
$$
Then $\aor$ is quasi-isomorphic to $A\ot H(R_n)$ since $R_n$ is formal.
\end{proof}

Consider projective DG $\aor$-modules
$\{P(\g, \mf{x})=(A \bt R_n) (e(\g)\bt e(\mf{x})) ~|~ \g \in \cal{B}, \mf{x} \in \bn \},$
and projective DG $A \ot H(R_n)$-modules
$\{PH(\g, \mf{x})=(A \ot H(R_n)) (e(\g) \ot e(\mf{x})) |~ \g \in \cal{B}, \mf{x} \in \bn \}.$

\begin{defn}
(1) Let $DGP(\aor)$ be the smallest full subcategory of $DG(\aor)$ which contains the projective DG $\aor$-modules $\{P(\g, \mf{x}) ~|~ \g \in \cal{B}, \mf{x} \in \bn \}$ and is closed under the cohomological grading shift functor $[1]$, the $t$-grading shift functor $\{1\}$ and taking mapping cones.

\n(2) Let $DGP(A \ot H(R_n))$ be the smallest full subcategory of $DG(A \ot H(R_n))$ which contains the projective DG $A \ot H(R_n)$-modules $\{PH(\g, \mf{x}) ~|~ \g \in \cal{B}, \mf{x} \in \bn \}$ and is closed under the cohomological grading shift functor $[1]$, the $t$-grading shift functor $\{1\}$ and taking mapping cones.
\end{defn}
Let $HP(\aor)$ and $HP(A \ot H(R_n))$ denote $0$th homology categories of $DGP(\aor)$ and $DGP(A \ot H(R_n))$, respectively.
Then $HP(\aor)$ and $HP(A \ot H(R_n))$ are triangulated categories.
Since $\aor$ is formal, we have the following equivalence of triangulated categories.

\begin{lemma} \label{k0 ut vn}
The triangulated categories $HP(\aor)$ and $HP(A \ot H(R_n))$ are equivalent.
Hence, there are isomorphisms of $\zt$-modules:
$$ K_0(HP(\aor)) \cong K_0(HP(A \ot H(R_n))) \cong \ut \ot_{\{T=t^n\}} V_1^{\ot n}.$$
\end{lemma}
\proof
It is easy to see that $K_0(HP(\aor))$ is isomorphic to a quotient of
$$\zT\lan \cal{B} \ran \times \zt\lan \bn \ran$$
by the relation $(\g \cdot T, \mf{x})=(\g, t^n \mf{x})$ for $\g \in \cal{B}$ and $\mf{x} \in \bn$ from the $t$-grading in $\aor$:
$$\dt(a\bt r)=n\dt(a)+\dt(r). \qed$$

\begin{defn} \label{chi}
Define a tensor product functor
$$
\begin{array}{cccccc}
 \chi_n: & HP(A) & \times & HP(H(R_n)) & \ra & HP(A \ot H(R_n)) \\
       &   M    &     ,   &   M'    & \mapsto & M \ot M',
\end{array}
$$
where the grading of $M \ot M'$ is given by:
\begin{align*}
 \dt(m\ot m')&=n\dt(m)+\dt(m'), \\
 \deh(m \ot m')&= \deh(m) + \deh(m') + 2k\dt(m),
\end{align*}
for $m \in M$ and $m' \in M'$ in $HP(H(R_{n,k})).$
\end{defn}

\begin{rmk}
The grading on $M \ot M'$ makes it into a $t$-graded DG $A \ot H(R_n)$-module.
\end{rmk}

\section{The $t$-graded DG $(H(R_n), A \bt R_n)$-bimodule $C_n$}

In order to define a functor $DGP(A \bt R_n) \ra DGP(H(R_n))$, we construct the $t$-graded DG $(H(R_n), A \bt R_n)$-bimodule $C_n$ in 4 steps:
\be
\item We define the first part of the left $H(R_n)$-module $C_n$ corresponding to the categorical action of $I,E,F$ on the objects of $DGP(R_n)$ in Section 6.1.
\item We define the first part of the right $A \bt R_n$-module structure on $C_n$ corresponding to the categorical action of $I,E,F$ on the morphisms of $DGP(R_n)$ in Section 6.2.
\item We finish the construction of the left $H(R_n)$-module $C_n$ corresponding to the action of $EF$ in Section 6.3.
\item We finish the definition of the right $A \bt R_n$-module structure on $C_n$ corresponding to the action of $EF$ in Section 6.4.
\ee
The algebraic construction is quite technical, but the geometric presentation in terms of decorated rook diagrams is easy to follow.

\subsection{The left DG $H(R_n)$-module $C_n$, Part I}
As a left DG $H(R_n)$-module,
$$C_n=\bigoplus \limits_{\g \in \cal{B}, \mf{x}\in \bn}C_n(\g, \mf{x})$$
In this subsection we define $C_n(\g, \mf{x})$ for $\g \in \{I, E, F\}$ and $\mf{x} \in \bn$.
We fix some $n>0$ throughout this section and omit the subscript $n$.

\subsubsection{The case $\g=I$}
Define $C(I, \mf{x})=PH(\mf{x}) \in DGP(H(R_n))$ for all $\mf{x} \in \bn$.

\subsubsection{The case $\g=F$}
For $\mf{x} \in \bnk$, recall the linear action $F(\mf{x})$ from Lemma \ref{utvn}.
Define
$$C(F, \mf{x})=\bigoplus \limits_{j=1}^{n-k}C_j(F,\mf{x})=\bigoplus \limits_{j=1}^{n-k}PH((F\mf{x})_j)\{n-\bar{x}_j\}[\beta(\mf{x},\bar{x}_j)],$$
where $(F\mf{x})_j$ denotes $\mf{x}\sqcup \{\bar{x}_j\}$ in $\g_{n,k+1}$.
Define a differential $d(F, \mf{x})=\sum \limits_{j=2}^{n-k} d_{j}(F, \mf{x})$, where
$d_{j}(F, \mf{x}): C_j(F,\mf{x}) \xra{\cdot r_F(\mf{x}; j)} C_{j-1}(F,\mf{x})$
is defined below for $2 \leq j \leq n-k$.

\vspace{.2cm}
Let $q_j=\left| \{l\in \{1,\dots,k\} ~|~ x_l < \bar{x}_j\} \right|,$
for $1 \leq j \leq n-k$.
Then the number of states $|1\ran$ between $\bar{x}_{j-1}$ and $\bar{x}_j$ is $q_j-q_{j-1}$.
Recall that $\mf{x} \xra{i} \mf{y}$ is the shorthand for the arrow $\mf{x} \xra{i,s_1,\mf{v}} \mf{y}$ with $s_1=s_0(\mf{v})=0$ in the quiver $Q_n$.
Then there exists a path:
$$(F\mf{x})_j \xra{q_{j-1}+1} \mf{z}^1 \xra{q_{j-1}+2} \cdots \xra{q_j} \mf{z}^{q_j-q_{j-1}} \xra{q_j +1} (F\mf{x})_{j-1}.$$
Define $r_F(\mf{x}; j) \in H(R_{n,k+1})$ as the product of the corresponding $q_j-q_{j-1}+1$ generators:
$$r_F(\mf{x}; j)=r((F\mf{x})_j \xra{q_{j-1}+1} \mf{z}^1) \cdots r(\mf{z}^{q_j-q_{j-1}} \xra{q_j +1} (F\mf{x})_{j-1}).$$
See Figure \ref{5-1-2} for an example.
\begin{figure}[h]
\begin{overpic}
[scale=0.25]{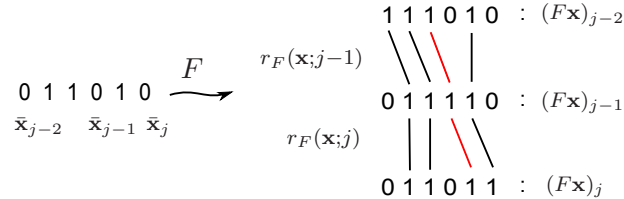}
\end{overpic}
\put(-193,25){${\scriptstyle \bar{\mf{x}}_{j-2}}$}
\put(-165,25){${\scriptstyle\bar{\mf{x}}_{j-1}}$}
\put(-143,25){${\scriptstyle\bar{\mf{x}}_{j}}$}
\put(-130,45){$F$}
\put(-100,50){${\scriptstyle r_F(\mf{x}; j-1)}$}
\put(-90,20){${\scriptstyle r_F(\mf{x}; j)}$}
\put(6,67){${\scriptstyle(F\mf{x})_{j-2}}$}
\put(6,34){${\scriptstyle(F\mf{x})_{j-1}}$}
\put(8,2){${\scriptstyle(F\mf{x})_{j}}$}
\caption{A local diagram of $r_F(\mf{x}; j) \cdot r_F(\mf{x}; j-1)=0$ since a product of two red strands is zero in $H(R_n)$.}
\label{5-1-2}
\end{figure}

\begin{defn} \label{Def df}
The differential $d_{j}(F, \mf{x}): C_j(F,\mf{x}) \xra{\cdot r_F(\mf{x}; j)} C_{j-1}(F,\mf{x})$ is a map of left $H(R_n)$-modules defined by right multiplication with $r_F(\mf{x}; j)$:
$$d_{j}(F, \mf{x})(m((F\mf{x})_j))=r_F(\mf{x}; j) \cdot m((F\mf{x})_{j-1}).$$
Here $m((F\mf{x})_j) \in C_j(F,\mf{x})=PH((F\mf{x})_j)\{n-\bar{x}_j\}[\beta(\mf{x},\bar{x}_j)]$ is the generator of the left projective $H(R_n)$-module for $1\leq j \leq n-k$.
\end{defn}

\begin{rmk} \label{Fx left}
The definition of the left $H(R_n)$-module $C(F, \mf{x})$ comes from a projective resolution of the left $H(R_n)$-module which corresponds to the dividing set $F \cdot \mf{x} \in \tC _n$.
See Figures \ref{1-1}, \ref{1-2}, \ref{1-3} in Section 1.3 for the dividing sets.
For instance, there is a distinguished triangle $F|010\ran \ra |011\ran \ra |110\ran$ in $\tC_3$ as in Figure \ref{5-1-1} which gives an isomorphism $F|010\ran \cong (|011\ran \ra |110\ran)$.
\begin{figure}[h]
\begin{overpic}
[scale=0.25]{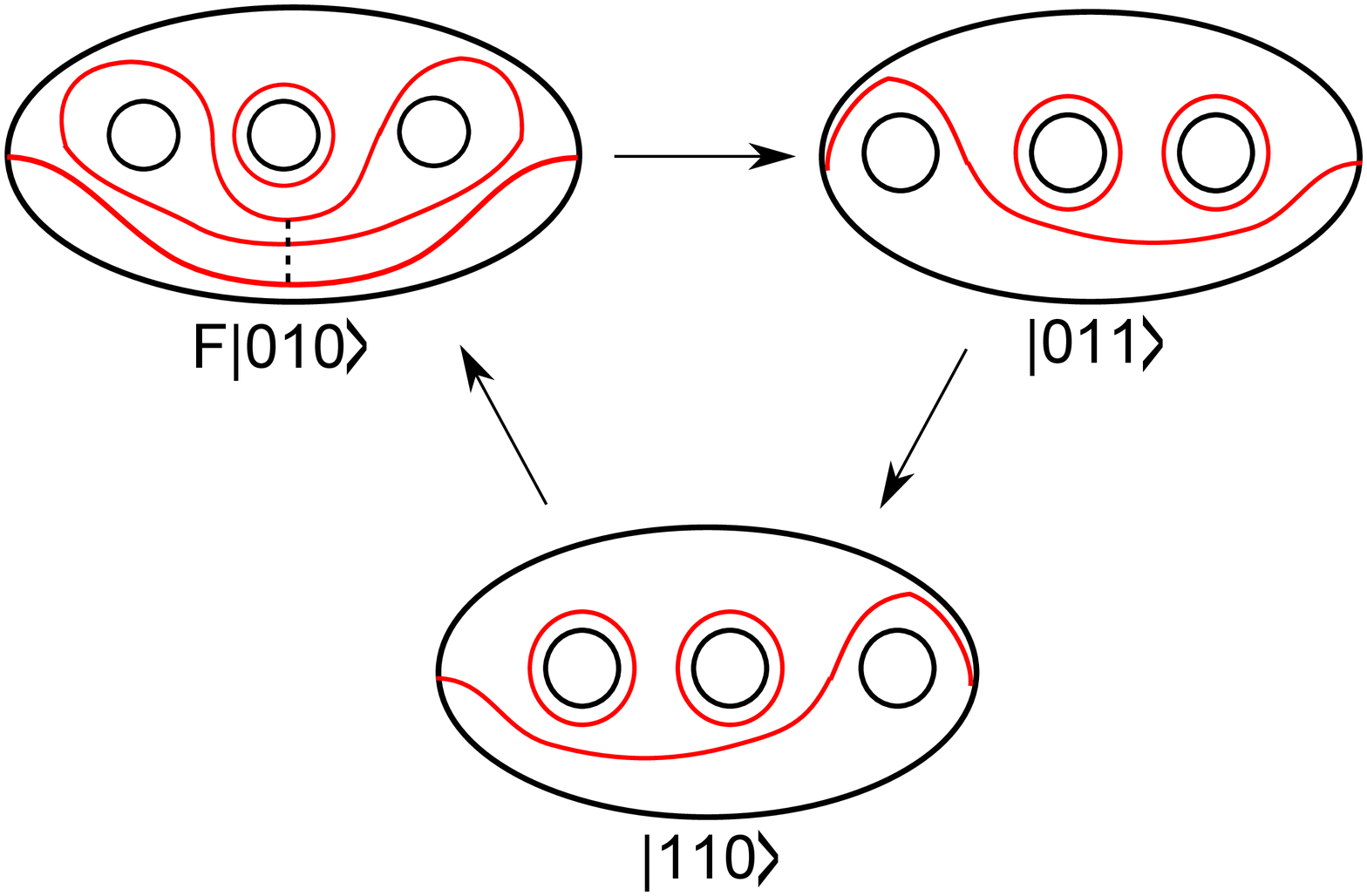}
\end{overpic}
\caption{}
\label{5-1-1}
\end{figure}
\end{rmk}

\begin{lemma}
$d_{j}(F, \mf{x})$ is a map of degree $(1,0)$ and $d_{j-1} \circ d_{j}=0.$
\end{lemma}
\begin{proof}
\n (1) The degrees of the generators are as follows:
\begin{gather*}
\op{deg}(m((F\mf{x})_j))=(-\beta(\mf{x},\bar{x}_j),\bar{x}_j-n), \\
\op{deg}(m((F\mf{x})_{j-1}))=(-\beta(\mf{x},\bar{x}_{j-1}),\bar{x}_{j-1}-n), \\
\op{deg}(r_F(\mf{x}; j))=(q_j-q_{j-1}+1, q_j-q_{j-1}+1).
\end{gather*}
Since $\beta(\mf{x},\bar{x}_{j-1})-\beta(\mf{x},\bar{x}_j)=q_j-q_{j-1}$ and $(n-\bar{x}_{j-1})-(n-\bar{x}_j)=q_j-q_{j-1}+1$,
we have $$\op{deg}(m((F\mf{x})_j))-\op{deg}(m((F\mf{x})_{j-1}))=\op{deg}(r_F(\mf{x}; j))-(1,0)$$ which implies that $d_{j}(F, \mf{x})$ is a map of degree $(1,0)$.

\vspace{.1cm}
\n (2) For a diagrammatic proof of $d_{j-1} \circ d_{j}=0$, see Figure \ref{5-1-2}.
The composition $d_{j-1} \circ d_{j}$ is right multiplication by
$r_F(\mf{x}; j) \cdot r_F(\mf{x}; j-1)$
and is induced by the following path:
$$
(F\mf{x})_j \xra{q_{j-1}+1} \mf{z}^1 \ra \cdots \ra \mf{z}^{q_j-q_{j-1}} \xra{q_j +1} (F\mf{x})_{j-1}
\xra{q_{j-2}+1} \mf{w}^1 \ra\cdots  \mf{w}^{q_{j-1}-q_{j-2}} \xra{q_{j-1} +1} (F\mf{x})_{j-2}.
$$
By using the commutation relation \ref{R2} to rearrange the arrows, the path above can be written as:
$$(F\mf{x})_j \xra{q_{j-2}+1} \mf{u}^{1}\ra \cdots \ra \mf{u}^{q_{j-1}-q_{j-2}} \xra{q_{j-1}+1} \mf{v}^{0} \xra{q_{j-1}+1} \mf{v}^{1} \ra \cdots \ra \mf{v}^{q_{j}-q_{j-1}} \xra{q_j+1} (F\mf{x})_{j-2}.$$
Hence,
$r_F(\mf{x}; j) \cdot r_F(\mf{x}; j-1)=\cdots r(\mf{u}^{q_{j-1}-q_{j-2}} \xra{q_{j-1}+1} \mf{v}^{0}) \cdot r(\mf{v}^{0} \xra{q_{j-1}+1} \mf{v}^{1}) \cdots=0$
by Relation \ref{R1}.
Here the product $r(\mf{u}^{q_{j-1}-q_{j-2}} \xra{q_{j-1}+1} \mf{v}^{0}) \cdot r(\mf{v}^{0} \xra{q_{j-1}+1} \mf{v}^{1})=0$ is given by a product of two red strands in Figure \ref{5-1-2}.
Then we have $d_{j-1} \circ d_{j}=0.$
\end{proof}

\subsubsection{The case $\g=E$}
For $\mf{x} \in \bnk$, recall the linear action $E(\mf{x})$ from Lemma \ref{utvn}.
Define
$$C(E, \mf{x})=\bigoplus \limits_{i=1}^{k}(C^i(E, \mf{x})\oplus C^i(E, \mf{x})')=\bigoplus \limits_{i=1}^{k} \left( PH((E\mf{x})^i)[1-i] \oplus PH((E\mf{x})^i)\{1\}[2-i] \right),$$
where $(E\mf{x})^i$ denotes $\mf{x}\backslash \{x_i\}$ in $\g_{n,k-1}$.
Define a differential $d(E, \mf{x})=\sum \limits_{i=1}^{k-1} d^{i}(E, \mf{x})$, where
$d^{i}(E, \mf{x}): C^i(E, \mf{x})\oplus C^i(E, \mf{x})' \ra C^{i+1}(E, \mf{x})\oplus C^{i+1}(E, \mf{x})'$
is defined below.

\vspace{.2cm}
Consider a path:
$(E\mf{x})^i \xra{i} \mf{z}^1 \xra{i} \cdots \xra{i} \mf{z}^{x_{i+1}-x_{i}-1} \xra{i} (E\mf{x})^{i+1}$ in $\g_{n,k-1}$.
Define $r_E(\mf{x}; i) \in H(R_{n,k-1})$ as the product of the generators corresponding to the $x_{i+1}-x_{i}$ arrows in the path and the $x_{i+1}-x_{i}-1$ loops attached at the vertices $\mf{z}^s$:
\begin{align*}
r_E(\mf{x}; i)=&r((E\mf{x})^i \xra{i} \mf{z}^1) \cdot \rho(\mf{z}^1 \xra{i} \mf{z}^1) \cdot r(\mf{z}^1 \xra{i} \mf{z}^2) \cdots  r(\mf{z}^{x_{i+1}-x_{i}-2} \xra{i} \mf{z}^{x_{i+1}-x_{i}-1})\\
&\cdot \rho(\mf{z}^{x_{i+1}-x_{i}-1} \xra{i} \mf{z}^{x_{i+1}-x_{i}-1}) \cdot r(\mf{z}^{x_{i+1}-x_{i}-1} \xra{i} (E\mf{x})^{i+1}).
\end{align*}
Define loops $$\theta(\mf{x};i)=\rho((E\mf{x})^i \xra{i} (E\mf{x})^i); \qquad \sigma(\mf{x};i)=\rho((E\mf{x})^{i+1} \xra{i} (E\mf{x})^{i+1}).$$
See Figure \ref{5-1-3} for an example.
Let $m((E\mf{x})^i) \in C^i(E, \mf{x})=PH(\mf{x}\backslash \{x_i\})[1-i]$ and $m'((E\mf{x})^i) \in C^i(E, \mf{x})'=PH(\mf{x}\backslash \{x_i\})\{1\}[2-i]$ be the generators of the left $H(R_n)$ modules for $1\leq i \leq k$.

\begin{defn} \label{Def de}
The differential $d^{i}(E, \mf{x}): C^i(E, \mf{x})\oplus C^i(E, \mf{x})' \ra C^{i+1}(E, \mf{x})\oplus C^{i+1}(E, \mf{x})'$ is a map of left $H(R_n)$-modules defined on the generators by:
\begin{align*}
d^{i}(E, \mf{x})(m((E\mf{x})^i)) = & \theta(\mf{x};i) \cdot r_E(\mf{x}; i) \cdot m((E\mf{x})^{i+1}) + r_E(\mf{x}; i) \cdot m'((E\mf{x})^{i+1}); \\
d^{i}(E, \mf{x})(m'((E\mf{x})^i))  = & \theta(\mf{x};i) \cdot r_E(\mf{x}; i) \cdot \sigma(\mf{x};i) \cdot m((E\mf{x})^{i+1}) + r_E(\mf{x}; i) \cdot \sigma(\mf{x};i) \cdot m'((E\mf{x})^{i+1}).
\end{align*}
\end{defn}

\begin{figure}[h]
\begin{overpic}
[scale=0.25]{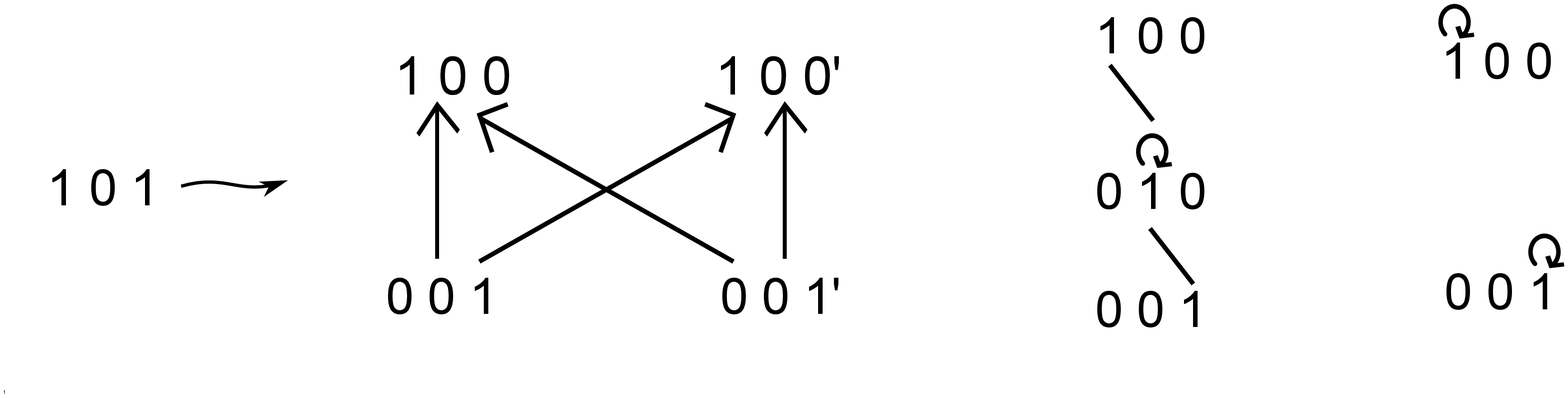}
\put(6,8){${\scriptstyle \mf{x}}$}
\put(14,16){${\scriptstyle E}$}
\put(25,0){${\scriptstyle (E\mf{x})^1}$}
\put(48,0){${\scriptstyle (E\mf{x})^1}$}
\put(25,25){${\scriptstyle (E\mf{x})^2}$}
\put(48,25){${\scriptstyle (E\mf{x})^2}$}
\put(68,0){${\scriptstyle r=r_E(\mf{x};1)}$}
\put(90,0){${\scriptstyle \theta=\theta(\mf{x};1)}$}
\put(90,16){${\scriptstyle \sigma=\sigma(\mf{x};1)}$}
\put(23,13){${\scriptstyle \theta r}$}
\put(35,8){${\scriptstyle r}$}
\put(34,16){${\scriptstyle \theta r \sigma}$}
\put(51,13){${\scriptstyle r \sigma}$}
\end{overpic}
\caption{$C(E,\mf{x})$ for $\mf{x}=|101\ran$, where $4$ upward arrows denote the differential.}
\label{5-1-3}
\end{figure}

\begin{lemma}
$d(E, \mf{x})$ is a map of degree $(1,0)$ and $d^{i+1} \circ d^{i}=0.$
\end{lemma}
\begin{proof}
\n (1) It is easy to verify that $d(E, \mf{x})$ is a map of degree $(1,0)$ since
\begin{gather*}
\op{deg}(r_E(\mf{x};i))=(1,1), \quad
\op{deg}(\theta(\mf{x};i))=(-1,-1), \quad
\op{deg}(\sigma(\mf{x};i))=(-1,-1).
\end{gather*}
\n (2) We show that $d^{i+1}(d^{i}(m((E\mf{x})^i)))=0$ and leave the case of $m'((E\mf{x})^i))$ to the reader.
\begin{align*}
& d^{i+1}(d^{i}(m((E\mf{x})^i))) \\
= & d^{i+1}\left(\theta(\mf{x};i) \cdot r_E(\mf{x}; i) \cdot m((E\mf{x})^{i+1}) + r_E(\mf{x}; i) \cdot m'((E\mf{x})^{i+1})\right) \\
= & \theta(\mf{x};i) \cdot r_E(\mf{x}; i) \cdot d^{i+1}(m((E\mf{x})^{i+1})) + r_E(\mf{x}; i) \cdot d^{i+1}(m'((E\mf{x})^{i+1})) \\
= & (\theta(\mf{x};i) \cdot r_E(\mf{x}; i) \cdot r_E(\mf{x}; i+1) + r_E(\mf{x}; i) \cdot r_E(\mf{x}; i+1) \cdot \sigma(\mf{x};i+1)) \cdot m'((E\mf{x})^{i+2}) \\
& + \theta(\mf{x};i) \cdot r_E(\mf{x}; i) \cdot \theta(\mf{x};i+1) \cdot r_E(\mf{x}; i+1) \cdot m((E\mf{x})^{i+2}) \\
& + r_E(\mf{x}; i) \cdot \theta(\mf{x};i+1) \cdot r_E(\mf{x}; i+1) \cdot \sigma(\mf{x};i+1) \cdot m((E\mf{x})^{i+2}).
\end{align*}

\begin{figure}[h]
\begin{overpic}
[scale=0.25]{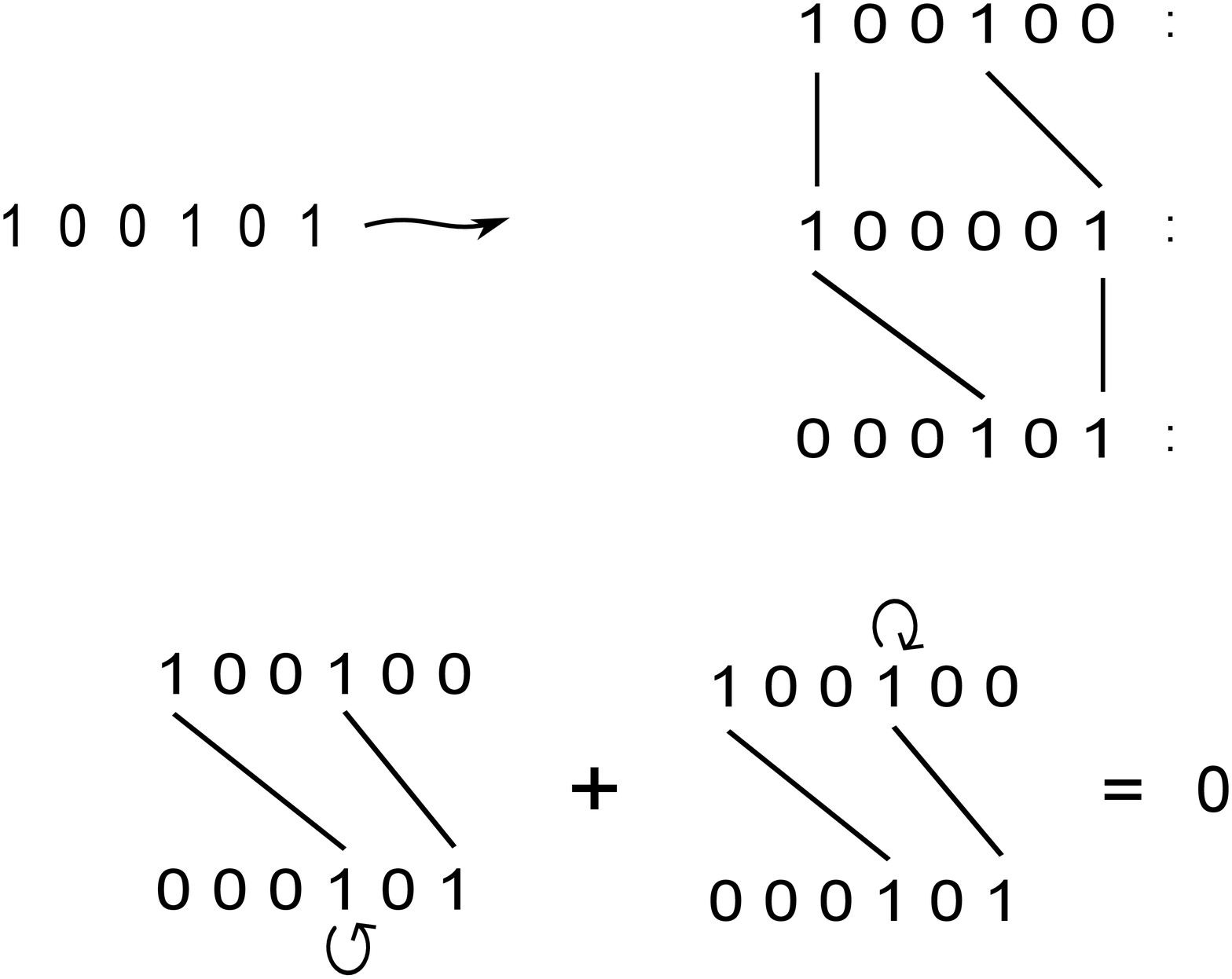}
\end{overpic}
\put(-206,110){${\scriptstyle \mf{x}_{i}}$}
\put(-175,110){${\scriptstyle \mf{x}_{i+1}}$}
\put(-152,110){${\scriptstyle \mf{x}_{i+2}}$}
\put(-136,130){$E$}
\put(-115,140){${\scriptstyle r_E(\mf{x}; i+1)}$}
\put(-105,100){${\scriptstyle r_E(\mf{x}; i)}$}
\put(0,158){${\scriptstyle (E\mf{x})^{i+2}}$}
\put(0,122){${\scriptstyle (E\mf{x})^{i+1}}$}
\put(2,88){${\scriptstyle (E\mf{x})^{i}}$}
\caption{The top diagram describes $(E\mf{x})^{i}$ and the bottom diagram represents the coefficient $\theta(\mf{x};i) \cdot r_E(\mf{x}; i) \cdot r_E(\mf{x}; i+1) + r_E(\mf{x}; i) \cdot r_E(\mf{x}; i+1) \cdot \sigma(\mf{x};i+1).$ Here, the diagrams representing $r_E(\mf{x}; i)$'s are defined in Figure \ref{4-1-5} in Section 5.1.4.}
\label{5-1-4}
\end{figure}
We compute the coefficient of $m'((E\mf{x})^{i+2})$ in Figure \ref{5-1-3}.
%$$\theta(\mf{x};i) \cdot r_E(\mf{x}; i) \cdot r_E(\mf{x}; i+1) + r_E(\mf{x}; i) \cdot r_E(\mf{x}; i+1) \cdot \sigma(\mf{x};i+1).$$
%Consider the path from $(E\mf{x})^i$ to $(E\mf{x})^{i+2}$ in the definition of the differential:
%$$(E\mf{x})^i \xra{i} \cdots \xra{i} (E\mf{x})^{i+1} \xra{i+1} \cdots \xra{i+1} (E\mf{x})^{i+2}.$$
%There exists another path by rearranging the arrows:
%$$(E\mf{x})^i \xra{i+1} \mf{u}^1 \xra{i+1} \cdots \mf{u}^{x_{i+2}-x_{i+1}-2} \xra{i+1} \mf{v}^{0} \xra{i} \mf{v}^1 \xra{i+1} \mf{v}^2 \xra{i} \cdots \mf{v}^{x_{i+1}-x_i} \xra{i} (E\mf{x})^{i+2}.$$
%Let
%\begin{gather*}
%r((E\mf{x})^i, \mf{v}^{0})=r(E\mf{x})^i \xra{i+1} \mf{u}^1) \cdot \rho(\mf{u}^1 \xra{i+1} \mf{u}^1) \cdots r(\mf{u}^{x_{i+2}-x_{i+1}-2} \xra{i+1} \mf{v}^{0}) \cdot \rho(\mf{v}^0 \xra{i+1} \mf{v}^0); \\
%r(\mf{v}^{2}, (E\mf{x})^{i+2})=\rho(\mf{v}^2 \xra{i} \mf{v}^2) \cdot r(\mf{v}^2 \xra{i} \mf{v}^3) \cdots \rho(\mf{v}^{x_{i+1}-x_i} \xra{i} \mf{v}^{x_{i+1}-x_i}) \cdot r(\mf{v}^{x_{i+1}-x_i} \xra{i} (E\mf{x})^{i+2}).
%\end{gather*}
%Then we have
%\begin{align*}
%& \theta(\mf{x};i) \cdot r_E(\mf{x}; i) \cdot r_E(\mf{x}; i+1) + r_E(\mf{x}; i) \cdot r_E(\mf{x}; i+1) \cdot \sigma(\mf{x};i+1) \\
%= & r((E\mf{x})^i, \mf{v}^{0}) \cdot \rho(\mf{v}^{0} \xra{i} \mf{v}^{0}) \cdot r(\mf{v}^{0} \xra{i} \mf{v}^{1}) \cdot r(\mf{v}^{1} \xra{i+1} \mf{v}^{2}) \cdot r(\mf{v}^{2}, (E\mf{x})^{i+2}) \\
%& + r((E\mf{x})^i, \mf{v}^{0}) \cdot r(\mf{v}^{0} \xra{i} \mf{v}^{1}) \cdot r(\mf{v}^{1} \xra{i+1} \mf{v}^{2}) \cdot \rho(\mf{v}^{2} \xra{i+1} \mf{v}^{2}) \cdot r(\mf{v}^{2}, (E\mf{x})^{i+2}) \\
%= & 0
%\end{align*}
The coefficient is zero by Relations \ref{R2} and \ref{R3}.
Similarly, we can prove that the coefficient of $m((E\mf{x})^{i+2})$ is zero.
Hence $d^{i+1}(d^{i}(m((E\mf{x})^i)))=0$.
\end{proof}

\subsection{The right $\aor$-module $C_n$, Part I}
In this subsection we define the right multiplication with the idempotents $e(\g) \bt e(\mf{x})$ and generators $e(\g) \bt \rho(\mf{x} \xra{i} \mf{x}), e(\g) \bt r(\mf{x} \xra{i,s_1,\mf{v}} \mf{y})$ of $\aor$ for $\g \in \{I, E, F\}$ and $\mf{x} \in \bn$.
Let $m \times r$ denote the right multiplication for $m \in C$ and generators $r \in \aor$.
The definition of right multiplication in general is extended by the associativity:
$m \times (r_1\cdot r_2)=(m \times r_1)\times r_2$.
We will check this is well-defined in Proposition \ref{welldef}.
The case-by-case definition will be labeled by (M$*$).
Let $m \cdot r'$ denote the multiplication in $H(R_n)$ for $m \in PH(\mf{x}) \subset R_n$ and $r' \in H(R_n)$.
Let $j(\mf{x},i)$ be the number in $\{0,1,\dots,n-k\}$ such that $\bar{x}_{j(\mf{x},i)} < x_i < \bar{x}_{j(\mf{x},i)+1}$ for $\mf{x}\in \bnk$.
Let $j_0$ denote $j(\mf{x},i)$ when $\mf{x}$ and $i$ are understood.
\subsubsection{Idempotents}
Let $a\bt r = e(\g) \bt e(\mf{x})$ be an idempotent.
Then define for $m \in C(\g',\mf{x'})$
\begin{gather}
m \times (e(\g) \bt e(\mf{x}))=\delta_{\g, \g'} ~ \delta_{\mf{x},\mf{x'}} ~ m.
\tag{M1}\label{M1}\end{gather}

\subsubsection{The case $a=e(I)$}
For $a\bt r = e(I) \bt r$ where $r \in \{\rho(\mf{x} \xra{i} \mf{x}), r(\mf{x} \xra{i,s_1,\mf{v}} \mf{y})\}$, define
\begin{gather}
m \times (e(I) \bt r)= \left\{
\begin{array}{cc}
m \cdot \rho(\mf{x} \xra{i} \mf{x}) & ~\mbox{if}~ r=\rho(\mf{x} \xra{i} \mf{x}),\\
m \cdot r(\mf{x} \xra{i} \mf{y}) & ~\mbox{if}~ r=r(\mf{x} \xra{i,s_1,\mf{v}} \mf{y}), s_1=s_0(\mf{v})=0,\\
0 & ~\mbox{otherwise,}~
\end{array}
\right. \tag{M2}\label{M2}
\end{gather}
for $m \in C(I, \mf{x})=PH(\mf{x})$.
We view $\bigoplus\limits_{\mf{x} \in \bn}C(I,\mf{x})=H(R_n)$ as an $(H(R_n), R_n)$-bimodule, where $R_n$ acts from right via the quasi-isomorphism $g_n: R_n \ra H(R_n)$ in Lemma \ref{quasirn}.

\subsubsection{The case $a=e(F)$}
$\mbox{}$ \\
\n (1) Let $a\bt r = e(F) \bt \rho(\mf{x} \xra{i} \mf{x})$.
The right multiplication is a map of left $H(R_n)$-modules $C(F, \mf{x}) \ra  C(F, \mf{x})$
defined on the generators by:
\begin{gather}
m((F\mf{x})_j) \times (e(F) \bt \rho(\mf{x} \xra{i} \mf{x}))=
\left\{
\begin{array}{cc}
\rho((F\mf{x})_j \xra{i} (F\mf{x})_j) \cdot m((F\mf{x})_j) & ~\mbox{if}~ j > j_0; \\
\rho((F\mf{x})_j \xra{i+1} (F\mf{x})_j) \cdot m((F\mf{x})_j) & ~\mbox{if}~ j \leq j_0.
\end{array}
\right.\tag{M3-1}\label{M3-1}
\end{gather}

\begin{rmk}
The morphism $e(F) \bt \rho(\mf{x} \xra{i} \mf{x}) \in \op{Hom}_{\tC _n}(F \cdot \mf{x}, F \cdot \mf{x})$ represents a tight contact structure from the dividing curve $F \cdot \mf{x}$ to itself.
Recall from Remark \ref{Fx left} that $(C(F, \mf{x}), d(F, \mf{x}))$ is the ``projective resolution" of $F \cdot \mf{x}$.
Then the right multiplication $C(F, \mf{x}) \ra  C(F, \mf{x})$ defined above is the corresponding morphism between their projective resolutions.
\end{rmk}

\begin{lemma} \label{rtF relation}
The right multiplication with $u=e(F) \bt \rho(\mf{x} \xra{i} \mf{x})$ is compatible with the relation in $\aor$:
$(m \times u) \times u=0=m \times (u \cdot u)=0.$
\end{lemma}
\begin{proof}
It follows from (\ref{M3-1}) and $\rho((F\mf{x})_j \xra{i} (F\mf{x})_j) \cdot \rho((F\mf{x})_j \xra{i} (F\mf{x})_j)=0$ for all $i$ and $j$.
\end{proof}

\begin{lemma} \label{rtFrho}
The right multiplication by $e(F) \bt \rho(\mf{x} \xra{i} \mf{x})$ commutes with the differential.
\end{lemma}
\begin{proof}
The commutativity for each square which is not in the diagram below follows from Commutativity Relation \ref{R2} since the corresponding decorated rook diagrams are disjoint.
$$
\xymatrix{
\cdots \ar[r] & PH((F\mf{x})_{j_0+1}) \ar[r]^{d} & PH((F\mf{x})_{j_0}) \ar[r] & \cdots\\
\cdots \ar[r] & PH((F\mf{x})_{j_0+1}) \ar[r]^{d} \ar[u]^{\cdot \rho((F\mf{x})_{j_0+1} \xra{i} (F\mf{x})_{j_0+1})} & PH((F\mf{x})_{j_0}) \ar[u]_{\cdot \rho((F\mf{x})_{j_0} \xra{i+1} (F\mf{x})_{j_0})} \ar[r] & \cdots
}$$

The commutativity for the square follows from Relation \ref{R3} since the sum of maps is a resolution of a crossing in $R_n$, hence zero in $H(R_n)$. See Figure \ref{5-2-1}.
\begin{figure}[h]
\begin{overpic}
[scale=0.28]{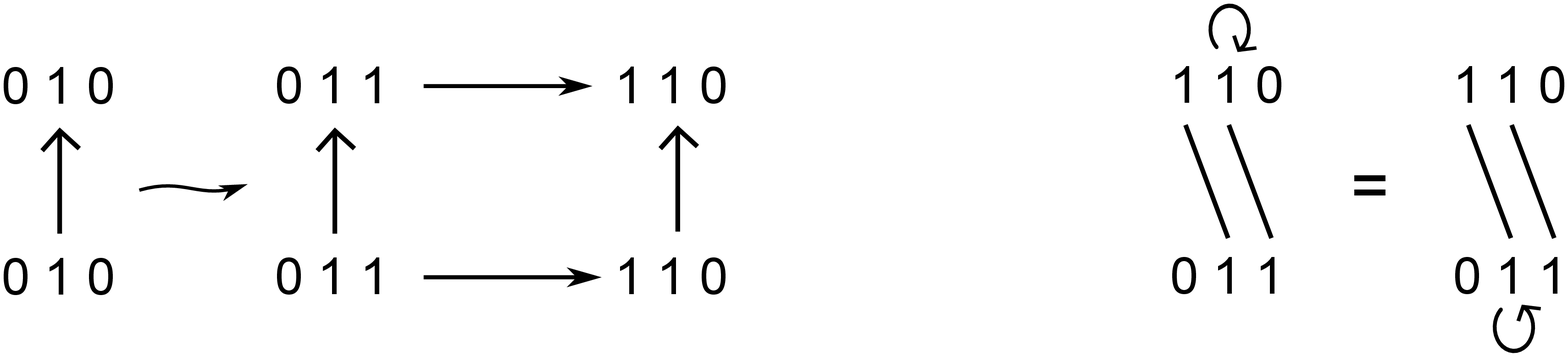}
\put(11,12){$F$}
\put(0,10){$\rho$}
\put(23,10){$\rho$}
\put(45,10){$\rho$}
\put(104,10){$\in H(R_n)$}
\put(28,6){${\scriptstyle r_F(\mf{x};2)}$}
\put(28,18){${\scriptstyle r_F(\mf{x};2)}$}
\put(3,0){${\scriptstyle \mf{x}}$}
\put(3,22){${\scriptstyle \mf{x}}$}
\put(18,0){${\scriptstyle (F\mf{x})_2}$}
\put(18,22){${\scriptstyle (F\mf{x})_2}$}
\put(40,0){${\scriptstyle (F\mf{x})_1}$}
\put(40,22){${\scriptstyle (F\mf{x})_1}$}
\end{overpic}
\caption{In the left-hand diagram, the upward arrow on the left is $\rho(\mf{x} \xra{1} \mf{x})$ for $\mf{x}=|010\ran$.
The horizontal arrows in $|011\ran \ra |110\ran$ are the differential in $C(F, |010\ran)$ given by the right multiplication with $r_F(\mf{x},2)$.
Two upward arrows between $|011\ran \ra |110\ran$'s are the right multiplication with $e(F) \bt \rho(\mf{x} \xra{i} \mf{x})$. The right-hand diagram shows that the right multiplication and differential commute.}
\label{5-2-1}
\end{figure}
\end{proof}

\begin{rmk}
The right multiplication with $e(F)\bt \rho(\mf{x} \xra{i} \mf{x})$ as in the left-hand diagram in Figure \ref{5-2-1} can be viewed as the functor $F$ applying to the morphism $\rho(\mf{x} \xra{i} \mf{x})$ in $\op{Hom}(\mf{x},\mf{x})$.
From now on, we will omit labels $\mf{x}, (F\mf{x})_j, r_F(\mf{x};i)$'s in this type of diagrams to express right multiplications.
\end{rmk}

\n (2) Let $a\bt r = e(F) \bt r(\mf{x} \xra{i,s_1,\mf{v}} \mf{y})$ with $s_1=s_0(\mf{v})=0$ for $\mf{x}, \mf{y} \in \bnk$.
Note that
$$\bar{y}_{j_0}=\bar{x}_{j_0}+1; \qquad \bar{y}_{j}=\bar{x}_{j} ~\mbox{for}~ j \neq j_0.$$
Then we have $(F\mf{x})_{j_0}=(F\mf{y})_{j_0} \in \g_{n,k+1}$ and there exist arrows in $\g_{n,k+1}$:
$$(F\mf{x})_{j} \xra{i} (F\mf{y})_{j} ~\mbox{for}~ j > j_0; \qquad (F\mf{x})_{j} \xra{i+1} (F\mf{y})_{j} ~\mbox{for}~ j < j_0.$$
The right multiplication is a map of left $H(R_n)$-modules: $C(F, \mf{x}) \ra  C(F, \mf{y})$ defined on the generators by:
\begin{gather}
m((F\mf{x})_{j}) \times (e(F) \bt r(\mf{x} \xra{i,s_1,\mf{v}} \mf{y}))=
\left\{
\begin{array}{cc}
r((F\mf{x})_{j} \xra{i} (F\mf{y})_{j}) \cdot m((F\mf{y})_{j}) & ~\mbox{if}~ j > j_0; \\
m((F\mf{y})_{j}) & ~\mbox{if}~ j = j_0; \\
r((F\mf{x})_{j} \xra{i+1} (F\mf{y})_{j}) \cdot m((F\mf{y})_{j}) & ~\mbox{if}~ j < j_0.
\end{array}
\right.\tag{M3-2}\label{M3-2}
\end{gather}
See the left-hand diagram in Figure \ref{5-2-2} for an example.
\begin{figure}[h]
\begin{overpic}
[scale=0.25]{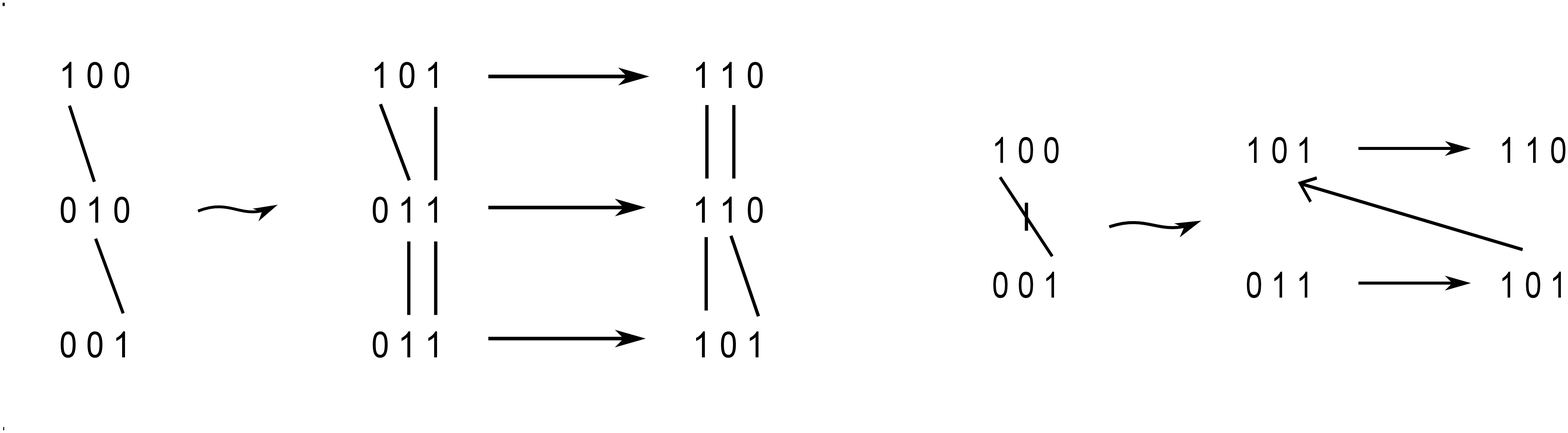}
\put(14,15){$F$}
\put(90,13){$id$}
\put(72,15){$F$}
\put(60,13){$r_0$}
\put(0,8){$r_1$}
\put(0,18){$r_2$}
\put(34,7){$d$}
\put(34,15){$d$}
\put(34,24){$d$}
\put(89,6){$d$}
\put(89,20){$d$}
\put(24,0){${\scriptstyle j=2}$}
\put(45,0){${\scriptstyle j=1}$}
\put(80,5){${\scriptstyle j=2}$}
\put(97,5){${\scriptstyle j=1}$}
\end{overpic}
\caption{The right multiplications with $e(F) \bt r_i$ for $i=1,2$ on the left, and for $i=0$ on the right.}
\label{5-2-2}
\end{figure}

\begin{lemma} \label{rtFr1}
The right multiplication with $e(F) \bt r(\mf{x} \xra{i,s_1,\mf{v}} \mf{y})$ such that $s_1=s_0(\mf{v})=0$ commutes with the differential.
\end{lemma}
\begin{proof}
We have the following diagram
$$
\xymatrix{
 PH((F\mf{y})_{j_0+1}) \ar[r]^{d} & PH((F\mf{y})_{j_0}) \ar[r]^{d} & PH((F\mf{y})_{j_0-1})  \\
 PH((F\mf{x})_{j_0+1}) \ar[r]^{d} \ar[u]^{\cdot r((F\mf{x})_{j_0+1} \xra{i} (F\mf{y})_{j_0+1})} & PH((F\mf{x})_{j_0}) \ar[r]^{d} \ar[u]_{id} & PH((F\mf{x})_{j_0-1}) \ar[u]_{\cdot r((F\mf{x})_{j_0-1} \xra{i+1} (F\mf{y})_{j_0-1})}
}$$
The commutativity follows from the commutativity relation \ref{R2}.
\end{proof}

\n (3) Let $a\bt r = e(F) \bt r(\mf{x} \xra{i,s_1,\mf{v}} \mf{y})$ with $s_1=0, s_0(\mf{v})>0$ for $\mf{x}, \mf{y} \in \bnk$.
Let $s_0$ denote $s_0(\mf{v})$ for simplicity.
Note that $(F\mf{x})_{j_0-s_0}=(F\mf{y})_{j_0} \in \g_{n,k+1}$.
Then the right multiplication is a map of left $H(R_n)$-modules defined on the generators by:
\begin{gather}
m((F\mf{x})_{j}) \times (e(F) \bt r(\mf{x} \xra{i,s_1,\mf{v}} \mf{y}))=
\left\{
\begin{array}{cc}
m((F\mf{y})_{j+s_0}) & ~\mbox{if}~ j = j_0-s_0; \\
0 & \mbox{otherwise}.
\end{array}
\right.\tag{M3-3}\label{M3-3}
\end{gather}
Note that $m((F\mf{y})_{j+s_0})=m((F\mf{y})_{j_0})=(F\mf{x})_{j}$ if $j = j_0-s_0$.
See the right-hand diagram in Figure \ref{5-2-2} for an example when $s_0=1$.

We verify that the definition is compatible with the DG structure on $\aor$.
\begin{lemma} \label{rtFr2}
$d(m \times r)=dm \times r + m \times dr$ holds for $r=e(F) \bt r(\mf{x} \xra{i,0,\mf{v}} \mf{y})$.
\end{lemma}
\begin{proof}
\n (1) The case for $s_0=0$ is proved in Lemma \ref{rtFr1}.

\vspace{.1cm}
\n (2) If $s_0=1$, there is only one marking in $r=e(F) \bt r(\mf{x} \xra{i,0,(1)} \mf{y})$.
Let $dr=e(F) \bt r(\mf{x} \xra{i,0,(0)} \mf{z}) \cdot e(F) \bt r(\mf{z} \xra{i,0,(0)} \mf{y})$ for some $\mf{z}$.

\n(2-1) We first discuss $m \in C_j(F,\mf{x})$ for $j< j_0-1$ or $j>j_0$, where the $j$th vertical strand is disjoint with the marking.
We have $m\times r=dm \times r=0$ from (\ref{M3-3}).
The lemma follows from
\begin{gather*}
m \times dr=
\left\{
\begin{array}{ll}
m\cdot r((F\mf{x})_{j} \xra{i} (F\mf{z})_{j}) \cdot r((F\mf{z})_{j} \xra{i} (F\mf{y})_{j})=m\cdot 0=0 & ~\mbox{if}~ j > j_0; \\
m\cdot r((F\mf{x})_{j} \xra{i+1} (F\mf{z})_{j}) \cdot r((F\mf{z})_{j} \xra{i+1} (F\mf{y})_{j})=m\cdot 0=0 & ~\mbox{if}~ j < j_0-1.
\end{array}
\right.
\end{gather*}

\n(2-2) From (2-1) we can reduce to the local model: $j=j_0-1, j_0$ where the marking lives.
A diagrammatic proof for $r_0=e(F) \bt r(\mf{x} \xra{1,0,(1)} \mf{y})$ is given in Figure \ref{5-2-2}, where $\mf{x}=(3)=|001\ran, \mf{y}=(1)=|100\ran \in \cal{B}_{3,1}$.
Recall that $$d(r_0)=e(F) \bt r(\mf{x} \xra{1,0,(0)} \mf{z}) \cdot e(F) \bt r(\mf{z} \xra{1,0,(0)} \mf{y})=r_1 \cdot r_2,$$
where $\mf{z}=(2)=|010\ran$.
In Figure \ref{5-2-2}, the right multiplications by $r_1, r_2$ and $r_0$ are given in the left-hand and right-hand diagrams, respectively.

We verify the equation for $m=m(|011\ran) \in C(F, |001\ran)$ by chasing the diagrams and leave other cases to the reader.
The right-hand side of the equation is zero since
\begin{align*}
d(m(|011\ran)) \times r_0 = & r(|011\ran \xra{1} |101\ran) \cdot \left(m(|101\ran) \times r_0\right) \\
= & r(|011\ran \xra{1} |101\ran) \cdot m(|101\ran) \in C(F, |100\ran), \\
m(|011\ran) \times d(r_0)= & (m(|011\ran) \times r_1)\times r_2 = m(|011\ran) \times r_2 \\
= & r(|011\ran \xra{1} |101\ran) \cdot m(|101\ran) \in C(F, |100\ran).
\end{align*}
The left-hand side is obviously zero since $m(|011\ran) \times r_0=0$.
Hence we proved $$d(m(|011\ran) \times r_0)=dm(|011\ran) \times dr_0)+dm(|011\ran) \times r_0.$$

\vspace{.1cm}
\n (3) If $s_0>1$, for $m \in C_j(F,\mf{x})$, $j<j_0-s_0$ or $j>j_0$ we have $m \times r=dm \times r=0$ and $m \times dr=0$ from (\ref{M3-3}) since $dr$ contains at least one marking.

Hence we reduce to the local model: $j_0-s_0\leq j \leq j_0$, i.e., $r_0=e(F) \bt r(\mf{x} \xra{1,0,(s_0)} \mf{y})$, where $\mf{x}=(s_0+2)=|0\cdots 01\ran, \mf{y}=(1)=|10\cdots 0\ran \in \cal{B}_{s_0+2,1}$.
A diagrammatic proof for $s_0=2$ is given in Figure \ref{5-2-0} by chasing the diagram.
The proof for the case $s_0>1$ in general is similar and we leave it to the reader.
\begin{figure}[h]
\begin{overpic}
[scale=0.2]{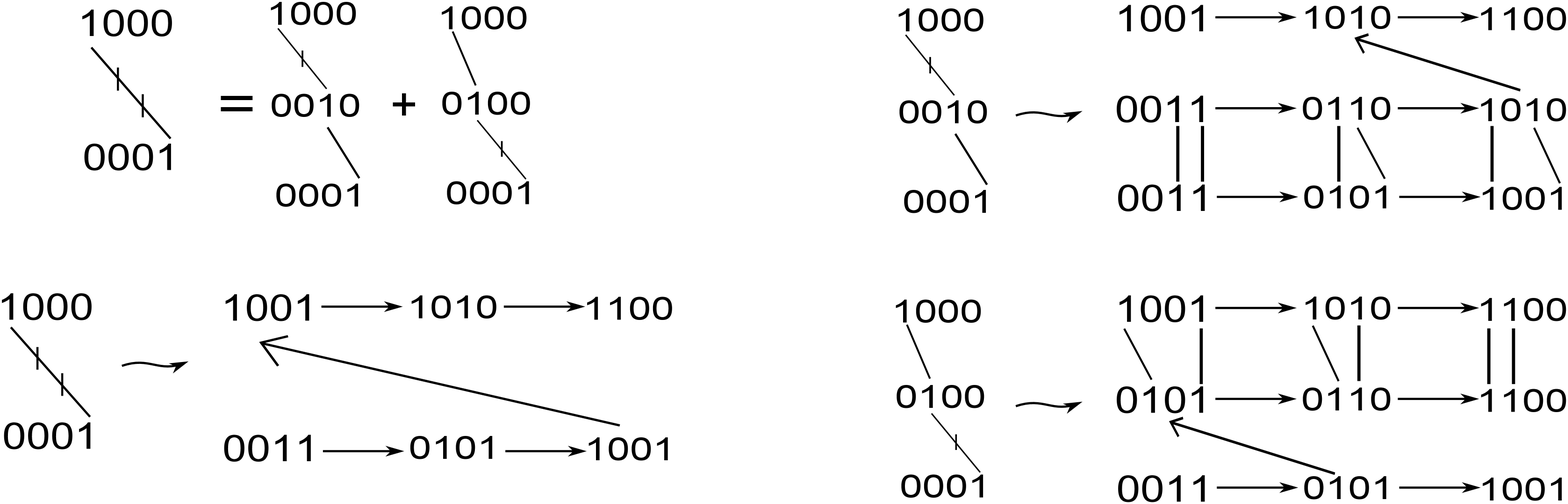}
\put(3,25){$d$}
\put(8,18){$r_0$}
\put(20,16.5){$r_1$}
\put(31,16.5){$r_2$}
\put(28,7){$id$}
\put(80,3){$id$}
\put(92,27){$id$}
\put(9,10){$F$}
\put(66,26){$F$}
\put(66,8){$F$}
\put(34,13){$d$}
\put(22,13){$d$}
\put(33,4){$d$}
\put(22,4){$d$}
\put(16,0){${\scriptstyle j=3}$}
\put(27,0){${\scriptstyle j=2}$}
\put(39,0){${\scriptstyle j=1}$}
\put(73,15){${\scriptstyle j=3}$}
\put(84,15){${\scriptstyle j=2}$}
\put(95,15){${\scriptstyle j=1}$}
\end{overpic}
\caption{The right multiplications with $e(F) \bt r_i$ for $i=0$ on the left, and for $i=1,2$ on the right.}
\label{5-2-0}
\end{figure}
\end{proof}
%Note that
%$$(F\mf{y})_{j_0-1}=(F\mf{z})_{j_0-1}, \quad (F\mf{z})_{j_0}=(F\mf{x})_{j_0}, \quad (F\mf{y})_{j_0}=(F\mf{x})_{j_0-1} \in \cal{B}_{n,k+1}.$$
%We need to verify the equation for the generators $m=m((F\mf{x})_{j})$ in 3 cases:
%\be
%\item $j=j_0-1$;
%\item $j=j_0$;
%\item $j \neq j_0-1, j_0$.
%\ee
%We prove it for Case (1) and leave Cases (2) and (3) to the reader.
%Suppose $j=j_0-1$. The left-hand side of equation is
%\begin{align*}
  %d(m((F\mf{x})_{j_0-1}) \times (e(F) \bt r(\mf{x} \xra{i,0,(1)} \mf{y})))
%= & d(m((F\mf{y})_{j_0})) \\
%=& r_F(\mf{y}, j_0) \cdot m((F\mf{y})_{j_0-1}) \\
%=& r((F\mf{y})_{j_0} \xra{i+1} (F\mf{y})_{j_0-1}) \cdot m((F\mf{y})_{j_0-1}),
%\end{align*}
%and the right-hand side is
%\begin{align*}
%& d(m((F\mf{x})_{j_0-1})) \times (e(F) \bt r(\mf{x} \xra{i,0,(1)} \mf{y})) + m((F\mf{x})_{j_0-1}) \times d(e(F) \bt r(\mf{x} \xra{i,0,(1)} \mf{y})) \\
%= & 0 + m((F\mf{x})_{j_0-1}) \times (r(\mf{x} \xra{i,0,(0)} \mf{z}) \cdot r(\mf{z} \xra{i,0,(0)} \mf{y})) \\
%=& (m((F\mf{x})_{j_0-1}) \times r(\mf{x} \xra{i,0,(0)} \mf{z})) \times r(\mf{z} \xra{i,0,(0)} \mf{y}) \\
%=& (r((F\mf{x})_{j_0-1} \xra{i+1} (F\mf{z})_{j_0-1}) \cdot m((F\mf{z})_{j_0-1})) \times r(\mf{z} \xra{i,0,(0)} \mf{y}) \\
%=& r((F\mf{x})_{j_0-1} \xra{i+1} (F\mf{z})_{j_0-1}) \cdot (m((F\mf{z})_{j_0-1}) \times r(\mf{z} \xra{i,0,(0)} \mf{y})) \\
%=& r((F\mf{x})_{j_0-1} \xra{i+1} (F\mf{z})_{j_0-1}) \cdot (m((F\mf{y})_{j_0-1}) \\
%=& r((F\mf{y})_{j_0} \xra{i+1} (F\mf{y})_{j_0-1}) \cdot m((F\mf{y})_{j_0-1}),
%\end{align*}
%which agrees with the left-hand side.

\n (4) Let $a\bt r = e(F) \bt r(\mf{x} \xra{i,s_1,\mf{v}} \mf{y})$ with $s_1>0$ for $\mf{x}, \mf{y} \in \bnk$.
The right multiplication is defined as the zero map:
\begin{gather}
m \times (e(F) \bt r(\mf{x} \xra{i,s_1,\mf{v}} \mf{y}))=0.\tag{M3-4}\label{M3-4}
\end{gather}

\subsubsection{The case $a=e(E)$}
$\mbox{}$ \\
\n (1) Let $a\bt r = e(E) \bt \rho(\mf{x} \xra{i_0} \mf{x})$.
The right multiplication is a map of left $H(R_n)$-modules: $C(E, \mf{x}) \ra  C(E, \mf{x})$
defined on the generators by:
\begin{gather}
m((E\mf{x})^{i}) \times (e(E) \bt \rho(\mf{x} \xra{i_0} \mf{x}))=
\left\{
\begin{array}{cc}
\rho((E\mf{x})^{i} \xra{i_0-1} (E\mf{x})^{i}) \cdot m((E\mf{x})^{i}) & ~\mbox{if}~ i<i_0, \\
m'((E\mf{x})^{i}) & ~\mbox{if}~ i=i_0, \\
\rho((E\mf{x})^{i} \xra{i_0} (E\mf{x})^{i}) \cdot m((E\mf{x})^{i}) & ~\mbox{if}~ i>i_0;
\end{array}
\right. \notag \\
m'((E\mf{x})^{i}) \times (e(E) \bt \rho(\mf{x} \xra{i_0} \mf{x}))=
\left\{
\begin{array}{cc}
\rho((E\mf{x})^{i} \xra{i_0-1} (E\mf{x})^{i}) \cdot m'((E\mf{x})^{i}) & ~\mbox{if}~ i<i_0, \\
0 & ~\mbox{if}~ i=i_0, \\
\rho((E\mf{x})^{i} \xra{i_0} (E\mf{x})^{i}) \cdot m'((E\mf{x})^{i}) & ~\mbox{if}~ i>i_0.
\end{array}
\right.\tag{M4-1}\label{M4-1}
\end{gather}
An example is given in Figure \ref{5-2-3}, where $|011\ran$ denotes the first summand $PH((E|111\ran)^1)$ and $|011\ran'$ denotes the second summand $PH((E|111\ran)^1)\{1\}[1]$ in $C(E, |111\ran)$.
\begin{figure}[h]
\begin{overpic}
[scale=0.3]{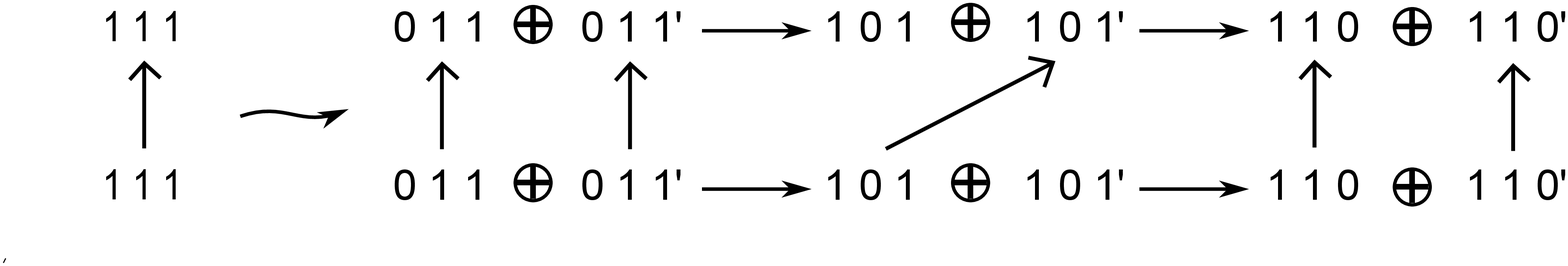}
\put(17,11){$E$}
\put(7,10){$\rho$}
\put(26,10){$\rho$}
\put(38,10){$\rho$}
\put(62,10){$id$}
\put(85,10){$\rho$}
\put(98,10){$\rho$}
\put(47,6){$d$}
\put(75,6){$d$}
\put(47,16){$d$}
\put(75,16){$d$}
\put(32,0){${\scriptstyle i=1}$}
\put(60,0){${\scriptstyle i=2}$}
\put(88,0){${\scriptstyle i=3}$}
\end{overpic}
\caption{The $5$ upward arrows on the right are the right multiplication with $e(E) \bt \rho(\mf{x} \xra{2} \mf{x})$ for $\mf{x}=|111\ran$.}
\label{5-2-3}
\end{figure}

\begin{lemma} \label{rtE relation}
The right multiplication with $u=e(E) \bt \rho(\mf{x} \xra{i_0} \mf{x})$ is compatible with the relation in $\aor$:
$(m \times u) \times u=0=m \times (u \cdot u).$
\end{lemma}
\begin{proof}
For $m \in C^i(E,\mf{x})$, it follows from (\ref{M4-1}) and $\rho((E\mf{x})_i \xra{i'} (E\mf{x})_i) \cdot \rho((E\mf{x})_i \xra{i'} (E\mf{x})_i)=0$ if $i \neq i_0$.
If $i=i_0$, it follows from the calculation of a local model as in Figure \ref{5-2-3}.
\end{proof}

\begin{lemma} \label{rtErho}
The right multiplication by $e(E) \bt \rho(\mf{x} \xra{i_0} \mf{x})$ commutes with the differential.
\end{lemma}
\begin{proof}
It suffices to prove that
$$d(m \times (e(E) \bt \rho(\mf{x} \xra{i_0} \mf{x})))=d(m) \times (e(E) \bt \rho(\mf{x} \xra{i_0} \mf{x})),$$
for $m=m((E\mf{x})^{i}), m'((E\mf{x})^{i})$ for $1 \leq i \leq k$.
The equation is obviously true for $i \neq i_0-1, i_0$ from the commutativity relation \ref{R2} since the decorated rook diagrams in the differential and right multiplication are disjoint.

We verify the equation for $m=m((E\mf{x})^{i_0})$:
\begin{align*}
& d\left(m((E\mf{x})^{i_0}) \times (e(E) \bt \rho(\mf{x} \xra{i_0} \mf{x}))\right) \\
= & d\left(m'((E\mf{x})^{i_0})\right) \\
= & \theta(\mf{x};i_0) \cdot r_E(\mf{x}; i_0) \cdot \sigma(\mf{x};i_0) \cdot m((E\mf{x})^{i_0+1}) + r_E(\mf{x}; i_0) \cdot \sigma(\mf{x};i_0) \cdot m'((E\mf{x})^{i_0+1}) \\
= & d\left(m((E\mf{x})^{i_0})\right) \cdot \rho((E\mf{x})^{i_0+1} \xra{i_0} (E\mf{x})^{i_0+1}) \\
= & d\left(m((E\mf{x})^{i_0})\right) \times \left(e(E) \bt \rho(\mf{x} \xra{i_0} \mf{x})\right).
\end{align*}
The proof for other cases is similar and we leave it to the reader.
\end{proof}

\n (2) Let $a\bt r = e(E) \bt r(\mf{x} \xra{i_0,s_1,\mf{v}} \mf{y})$ with $s_1=s_0(\mf{v})=0$ for $\mf{x}, \mf{y} \in \bnk$.
Note that
$$y_{i_0}=x_{i_0}-1; \qquad y_{i}=x_{i} ~\mbox{for}~ i \neq i_0.$$
We have $(E\mf{x})^{i_0} = (E\mf{y})^{i_0} \in \g_{n,k-1}$ and there exist arrows:
\begin{gather*}
(E\mf{x})^{i} \xra{i_0-1} (E\mf{y})^{i} ~\mbox{for}~ i < i_0, \qquad (E\mf{x})^{i} \xra{i_0} (E\mf{y})^{i} ~\mbox{for}~ i > i_0.
\end{gather*}
Then the right multiplication is a map of left $H(R_n)$-modules defined on the generators by:
\begin{gather}
m((E\mf{x})^{i}) \times (e(E) \bt r(\mf{x} \xra{i_0,s_1,\mf{v}} \mf{y}))=
\left\{
\begin{array}{cc}
r((E\mf{x})^{i} \xra{i_0-1} (E\mf{y})^{i}) \cdot m((E\mf{y})^{i}) & ~\mbox{if}~ i<i_0; \\
0 & ~\mbox{if}~ i=i_0; \\
r((E\mf{x})^{i} \xra{i_0} (E\mf{y})^{i}) \cdot m((E\mf{y})^{i}) & ~\mbox{if}~ i>i_0.
\end{array}
\right.\notag \\
m'((E\mf{x})^{i}) \times (e(E) \bt r(\mf{x} \xra{i_0,s_1,\mf{v}} \mf{y}))=
\left\{
\begin{array}{cc}
r((E\mf{x})^{i} \xra{i_0-1} (E\mf{y})^{i}) \cdot m'((E\mf{y})^{i}) & ~\mbox{if}~ i<i_0; \\
m((E\mf{y})^{i}) & ~\mbox{if}~ i=i_0; \\
r((E\mf{x})^{i} \xra{i_0} (E\mf{y})^{i}) \cdot m'((E\mf{y})^{i}) & ~\mbox{if}~ i>i_0.
\end{array}
\right.\tag{M4-2}\label{M4-2}
\end{gather}
An example is given in Figure \ref{5-2-4}.
\begin{figure}[h]
\begin{overpic}
[scale=0.3]{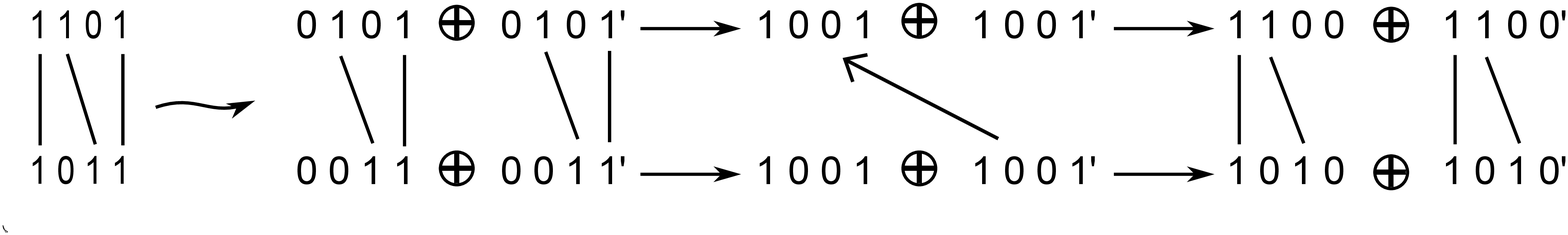}
\put(12,10){$E$}
\put(58,8){$id$}
\put(44,5){$d$}
\put(73,5){$d$}
\put(44,14){$d$}
\put(73,14){$d$}
\put(28,0){${\scriptstyle i=1}$}
\put(58,0){${\scriptstyle i=2}$}
\put(88,0){${\scriptstyle i=3}$}
\end{overpic}
\caption{The $5$ diagrams on the right represent the right multiplication with $e(E) \bt r(|1011\ran \xra{2,0,(0)} |1101\ran)$.}
\label{5-2-4}
\end{figure}
\begin{lemma} \label{rtEr1}
The right multiplication by $e(E) \bt r(\mf{x} \xra{i_0,s_1,\mf{v}} \mf{y})$ such that $s_1=s_0(\mf{v})=0$ commutes with the differential.
\end{lemma}
\begin{proof}
It suffices to prove that
$$d(m \times (e(E) \bt r(\mf{x} \xra{i_0,s_1,\mf{v}} \mf{y})))=d(m) \times (e(E) \bt r(\mf{x} \xra{i_0,s_1,\mf{v}} \mf{y})),$$
where $m=m((E\mf{x})^{i}), m'((E\mf{x})^{i})$ for $1 \leq i \leq k$.
The equation is obviously true for $i \neq i_0-1, i_0$ from the commutativity Relation \ref{R2} since the decorated rook diagrams in the differential and right multiplication are disjoint.

We verify the equation for $m=m'((E\mf{x})^{i_0})$:
\begin{align*}
& d\left(m'((E\mf{x})^{i_0}) \times (e(E) \bt r(\mf{x} \xra{i_0,s_1,\mf{v}} \mf{y}))\right) \\
= & d\left(m((E\mf{y})^{i_0})\right) \\
= & \theta(\mf{y};i_0) \cdot r_E(\mf{y}; i_0) \cdot m((E\mf{y})^{i_0+1}) + r_E(\mf{y}; i_0) \cdot m'((E\mf{y})^{i_0+1}) \\
= & d\left(m'((E\mf{x})^{i_0})\right) \cdot r((E\mf{x})^{i_0+1} \xra{i_0} (E\mf{y})^{i_0+1}) \\
= & d\left(m'((E\mf{x})^{i_0})\right) \times \left(e(E) \bt r(\mf{x} \xra{i_0,s_1,\mf{v}} \mf{y})\right).
\end{align*}
The proof for other cases is similar and we leave it to the reader.
\end{proof}

\n (3) Let $a\bt r = e(E) \bt r(\mf{x} \xra{i_0,s_1,\mf{v}} \mf{y})$ with $s_1>0, s_0(\mf{v})=0$ for $\mf{x}, \mf{y} \in \bnk$.
Note that $(E\mf{x})^{i_0+s_1}=(E\mf{y})^{i_0} \in \g_{n,k-1}$.
The right multiplication is a map of left $H(R_n)$-modules defined on the generators by:
\begin{gather}
m \times (e(E) \bt r(\mf{x} \xra{i,s_1,\mf{v}} \mf{y}))=
\left\{
\begin{array}{cc}
m((E\mf{y})^{i-s_1}) & ~\mbox{if}~ m=m'((E\mf{x})^{i_0+s_1}); \\
0 & \mbox{otherwise}.
\end{array}
\right.\tag{M4-3}\label{M4-3}
\end{gather}
See the right-hand diagram in Figure \ref{5-2-5} for an example.
\begin{figure}[h]
\begin{overpic}
[scale=0.25]{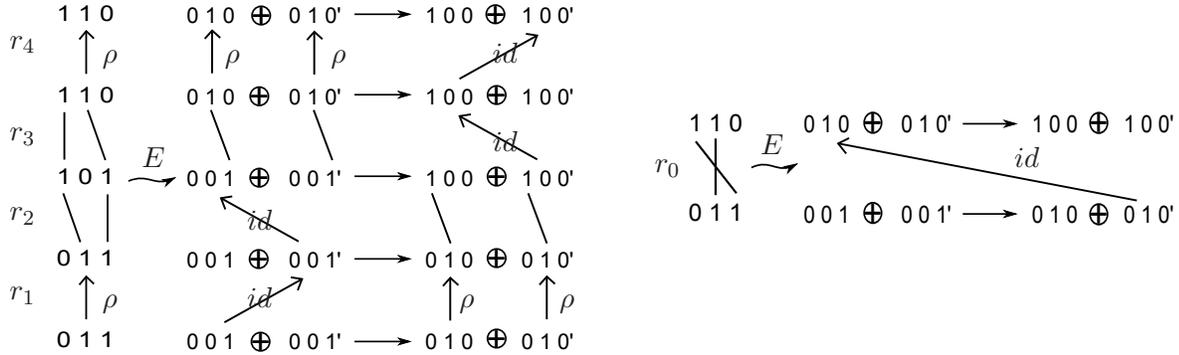}
\put(4,4){$\rho$}
\put(36,4){$\rho$}
\put(45,4){$\rho$}
\put(4,26){$\rho$}
\put(15,26){$\rho$}
\put(24.5,26){$\rho$}
\put(17,11){$id$}
\put(17,4){$id$}
\put(39,26){$id$}
\put(39,18){$id$}
\end{overpic}
\put(-196,68){$r_0$}
\put(-156,75){$E$}
\put(-60,70){$id$}
\put(-390,70){$E$}
\put(-440,20){$r_1$}
\put(-440,50){$r_2$}
\put(-440,80){$r_3$}
\put(-440,115){$r_4$}
\caption{The right multiplication with $e(E)\bt r_i$ for $i=1,2,3,4$ on the left and for $i=0$ on the right, where $dr_0=r_1r_2r_3+r_2r_3r_4$.}
\label{5-2-5}
\end{figure}

We verify that the definition is compatible with the DG structure on $\aor$.
\begin{lemma} \label{rtEr2}
$d(m \times r)=dm \times r + m \times dr$ holds for $r=e(E) \bt r(\mf{x} \xra{i_0,s_1,\mf{v}} \mf{y})$ with $s_0(\mf{v})=0$.
\end{lemma}
\begin{proof}
\n (1) The case for $s_1=0$ is proved in Lemma \ref{rtEr1}.

\vspace{.1cm}
\n (2) Suppose $s_1>0$.
For $m \in C^i(E,\mf{x})$ or $C^i(E,\mf{x})'$, $m \times r=dm \times r=m \times dr=0$ if $i<i_0-s_1$ or $i>i_0$ from (\ref{M4-2}) and (\ref{M4-3}).

If $i_0-s_1 \leq i \leq i_0$, we reduce to the local model:
$r=e(E) \bt r(\mf{x} \xra{1,s_1,\mf{v}(s_1)} \mf{y})$, where $\mf{x}=(2,\dots, s_1+2)=|01\cdots 1\ran, ~\mf{y}=(1,\dots, s_1+1)=|1\cdots 10\ran \in \cal{B}_{s_1+2,s_1+1}$, and $\mf{v}(s_1)=(0,\dots,0) \in \N^{s_1+1}$.
A diagrammatic proof for $s_1=1$ is given in Figure \ref{5-2-5}. Recall that $d(r_0)=r_1 \cdot r_2 \cdot r_3 + r_2 \cdot r_3 \cdot r_4,$
where
\begin{align*}
r_1=e(E) \bt \rho(\mf{x} \xra{1} \mf{x}), & \quad r_2=e(E) \bt r(\mf{x} \xra{1,0,(0)} \mf{z}), \\
r_4=e(E) \bt \rho(\mf{y} \xra{2} \mf{y}), & \quad r_3=e(E) \bt r(\mf{z} \xra{2,0,(0)} \mf{y}),
\end{align*}
and $\mf{z}=(1,3)=|101\ran$.
In Figure \ref{5-2-5}, the right multiplications by $r_1, r_2, r_3, r_4$ and $r_0$ are given in the left-hand and right-hand diagrams, respectively.

We verify the equation for $m=m(|001\ran) \in C(E, |011\ran)$ by chasing the diagrams and leave other cases to the reader.
The right-hand side of the equation is zero since
\begin{align*}
& d(m(|001\ran)) \times r_0 \\
= & \left(\rho(|001\ran \xra{1} |001\ran) \cdot r(|001\ran \xra{1} |010\ran) \cdot m(|010\ran) + r(|001\ran \xra{1} |010\ran) \cdot m'(|010\ran)\right) \times r_0 \\
= & r(|001\ran \xra{1} |010\ran) \cdot (m'(|010\ran) \times r_0) \\
= & r(|001\ran \xra{1} |010\ran) \cdot m(|010\ran) \in C(E, |110\ran),
\end{align*}
is the same as
\begin{align*}
 m(|001\ran) \times d(r_0) =& m(|001\ran) \times (r_1 \cdot r_2 \cdot r_3 + r_2 \cdot r_3 \cdot r_4) = ((m(|001\ran) \times r_1) \times r_2) \times r_3 \\
= & r(|001\ran \xra{1} |010\ran) \cdot m(|010\ran) \in C(E, |110\ran).
\end{align*}
The left-hand side is obviously zero since $m(|001\ran) \times r_0=0$.

The proof for the case $s_1>1$ is similar.
\end{proof}

\n (4) Let $a\bt r = e(E) \bt r(\mf{x} \xra{i,s_1,\mf{v}} \mf{y})$ with $s_0(\mf{v})>0$ for $\mf{x}, \mf{y} \in \bnk$.
The right multiplication is defined as the zero map:
\begin{gather}
m \times (e(E) \bt r(\mf{x} \xra{i,s_1,\mf{v}} \mf{y}))=0.\tag{M4-4}\label{M4-4}
\end{gather}

\subsection{The left DG $H(R_n)$-module $C_n$, Part II}
We finish the definition of the left $H(R_n)$-module structure on $C(\g, \mf{x})$ for $\g=EF$ and $\mf{x} \in \bnk$.
The module $C(EF,\mf{x})$ is constructed through the action of $E$ on the module $C(F, \mf{x})$.
Let $(EF\mf{x})^{i}_{j}$ denote $(E(F\mf{x})_{j})^{i} \in \bnk$, i.e.,
$$(EF\mf{x})^{i}_{j}=
\left\{
\begin{array}{cc}
\mf{x}\sqcup \{\bar{x}_j\} \backslash \{x_i\} & ~\mbox{if}~ i < q_j+1; \\
\mf{x} & ~\mbox{if}~ i=q_j+1; \\
\mf{x}\sqcup \{\bar{x}_j\} \backslash \{x_{i-1}\} & ~\mbox{if}~ i> q_j+1.
\end{array}
\right.$$
Define
\begin{align*}
C(EF, \mf{x}) = & \bigoplus \limits_{j=1}^{n-k}  C(E, (F\mf{x})_j)\{n-\bar{x}_j\}[\beta(\mf{x},\bar{x}_j)] \\
= & \bigoplus \limits_{j=1}^{n-k} \bigoplus \limits_{i=1}^{k+1}  (PH((EF\mf{x})^{i}_{j})\{n-\bar{x}_j\}[\beta(\mf{x},\bar{x}_j)+1-i] \\
 & \qquad\quad \oplus PH((EF\mf{x})^{i}_{j})\{n-\bar{x}_j+1\}[\beta(\mf{x},\bar{x}_j)+2-i]).
\end{align*}

Recall that $r_F(\mf{x}; j) \in H(R_{n,k+1})$ is given by the path from $(F\mf{x})_j$ to $(F\mf{x})_{j-1}$.
It can also be viewed as an element in $R_{n,k+1}$ which is still denoted by $r_F(\mf{x}; j)$.
The right multiplication with $e(E) \bt r_F(\mf{x}; j)$ defines a chain map:
$$\times (e(E) \bt r_F(\mf{x}; j)): C(E, (F\mf{x})_j) \ra C(E, (F\mf{x})_{j-1}).$$
We view $C(EF, \mf{x})$ as a double complex with $(i,j)$-th entry $C_{j}^{i}(EF, \mf{x})$ equal to:
$$PH((EF\mf{x})^{i}_{j})\{n-\bar{x}_j\}[\beta(\mf{x},\bar{x}_j)+1-i] \oplus PH((EF\mf{x})^{i}_{j})\{n-\bar{x}_j+1\}[\beta(\mf{x},\bar{x}_j)+2-i].$$
Let $m((EF\mf{x})^{i}_{j})$ and $m'((EF\mf{x})^{i}_{j})$ be the generators of the first and the second summand of $C_{j}^{i}(EF, \mf{x})$.

\begin{defn} \label{Def def}
The differential $d(EF, \mf{x})$ is defined by
\begin{align*}
d(EF, \mf{x})=  \sum \limits_{j=1}^{n-k}\sum \limits_{i=1}^{k+1} d_{j}^{i}(EF, \mf{x})
=  \sum \limits_{j=1}^{n-k}\sum \limits_{i=1}^{k+1} \left( d_{j}^{i}|_{ver}(EF, \mf{x})+d_{j}^{i}|_{hor}(EF, \mf{x}) \right),
\end{align*}
where
$$d_{j}^{i}|_{ver}(EF, \mf{x})=d^{i}(E, (F\mf{x})_j): C_{j}^{i}(EF, \mf{x}) \ra C_{j}^{i+1}(EF, \mf{x});$$
$$d^{i}_{j}|_{hor}(EF, \mf{x})= (\times e(E) \bt r_F(\mf{x}; j)): C_{j}^{i}(EF, \mf{x}) \ra C_{j-1}^{i}(EF, \mf{x}).$$
\end{defn}

We have the following double complex $(C(EF, \mf{x}), d(EF, \mf{x}))$:
\begin{gather} \tag{D}\label{double EFx}
\xymatrix{
C_{n-k}^{k+1}(EF, \mf{x}) \ar[r] &  C_{n-k-1}^{k+1}(EF, \mf{x}) \ar[r] & \cdots \ar[r] & C_{1}^{k+1}(EF, \mf{x}) \\
\vdots \ar[r] \ar[u] & \vdots \ar[u] \ar[r] & \vdots \ar[u] \ar[r] & \vdots \ar[u] \\
C_{n-k}^{1}(EF, \mf{x}) \ar[r] \ar[u]^{d^{1}(E, (F\mf{x})_{n-k})} & C_{n-k-1}^{1}(EF, \mf{x}) \ar[r] \ar[u]^{d^{1}(E, (F\mf{x})_{n-k-1})} & \cdots \ar[u] \ar[r] & C_{1}^{1}(EF, \mf{x}) \ar[u]^{d^{1}(E, (F\mf{x})_{1})}
}
\end{gather}

\begin{lemma}
The differential $d(EF, \mf{x})$ is well-defined.
\end{lemma}
\begin{proof}
Since $r_F(\mf{x}; j)$ is a product of the generators which satisfy $s_1=s_0(\mf{v})=0$, the horizontal differential $d_{j}^{i}|_{hor}(EF, \mf{x})$, i.e., the right multiplication by $e(E) \bt r_F(\mf{x}; j)$, commutes with the vertical differential $d_{j}^{i}|_{ver}(EF, \mf{x})$ by Lemma \ref{rtEr1}.
Hence $d(EF, \mf{x}) \circ d(EF, \mf{x})=0$.
\end{proof}

This concludes the construction of the left $H(R_n)$-module $C_n$.

The following lemma is immediate:
\begin{lemma} \label{k0 utvn}
If $\g \in \cal{B}$ and $\mf{x} \in \bn$, then $[C(\g, \mf{x})]=\g(\mf{x})$ when viewed as elements in $$K_0(HP(H(R_n))) \cong V_1^{\ot n}.$$
\end{lemma}

\subsection{The right $\aor$-module $C_n$, Part II}
We finish the definition of the right multiplication with generators in
\begin{gather*}
\{ e(EF) \bt \rho(\mf{x} \xra{i} \mf{x}),~ e(EF) \bt r(\mf{x} \xra{i,s_1,\mf{v}} \mf{y})\},\\
\{\rho(I,EF) \bt e(\mf{x}),~ \rho(EF,I) \bt e(\mf{x}),\}, \\
\{\rho(I\mf{x} \xra{i} EF\mf{x}), \quad \rho(EF\mf{x} \xra{j} I\mf{x})\}.
\end{gather*}

\subsubsection{The right multiplication by $e(EF) \bt r$}
The right multiplication by $e(EF) \bt r$ corresponds to apply the functor $EF$ to $r$ as a morphism in $DGP(R_n)$.
It can be decomposed into two steps: first applying $F$ to $r$ to obtain a morphism, then applying $E$ to the resulting morphism.

\vspace{.2cm}
\n (1) Let $a\bt r = e(EF) \bt \rho(\mf{x} \xra{i} \mf{x})$.
The right multiplication is a map $C(EF, \mf{x}) \ra  C(EF, \mf{x})$ of left $H(R_n)$-modules.
Recall that $$C(EF, \mf{x})  =  \bigoplus \limits_{j=1}^{n-k}  C(E, (F\mf{x})_j)\{n-\bar{x}_j\}[\beta(\mf{x},\bar{x}_j)],$$
and the right multiplication
$$\times (e(F) \bt \rho(\mf{x} \xra{i} \mf{x})): PH((F\mf{x})_j) \ra PH((F\mf{x})_j)$$
is given in (\ref{M3-1}).
Then the right multiplication by $e(EF) \bt \rho(\mf{x} \xra{i} \mf{x})$ is defined by:
\begin{gather}
m \times \left(e(EF) \bt \rho(\mf{x} \xra{i} \mf{x})\right)=
\left\{
\begin{array}{cc}
m \times \left(e(E) \bt \rho((F\mf{x})_j \xra{i} (F\mf{x})_j)\right) & ~\mbox{if}~ j > j_0, \\
m \times \left(e(E) \bt \rho((F\mf{x})_j \xra{i+1} (F\mf{x})_j)\right) & ~\mbox{if}~ j \leq j_0,
\end{array}
\right. \tag{M5-1}\label{M5-1}
\end{gather}
where $m \in C(E, (F\mf{x})_j)\{n-\bar{x}_j\}[\beta(\mf{x},\bar{x}_j)] \subset C(EF, \mf{x})$.

\vspace{.2cm}
\n (2) Let $a\bt r = e(EF) \bt r(\mf{x} \xra{i,s_1,\mf{v}} \mf{y})$.
Recall that the right multiplication
$$\times (e(F) \bt r(\mf{x} \xra{i,s_1,\mf{v}} \mf{y})): PH((F\mf{x})_j) \ra PH((F\mf{y})_{j+s_0})$$
is given in (\ref{M3-2}), (\ref{M3-3}) and (\ref{M3-4}).

\vspace{.1cm}
If $s_1+s_0(\mf{v})=0$, the right multiplication by $e(EF) \bt r(\mf{x} \xra{i,0,(0)} \mf{y})$ is defined by:
\begin{gather}
m \times \left(e(EF) \bt r(\mf{x} \xra{i,0,(0)} \mf{y})\right)=
\left\{
\begin{array}{cc}
m \times \left(e(E) \bt r((F\mf{x})_j \xra{i,0,(0)} (F\mf{y})_j)\right) & ~\mbox{if}~ j > j_0, \\
m \times \left(e(E) \bt e((F\mf{y})_j)\right) & ~\mbox{if}~ j = j_0, \\
m \times \left(e(E) \bt r((F\mf{x})_j \xra{i+1,0,(0)} (F\mf{y})_j)\right) & ~\mbox{if}~ j < j_0,
\end{array}
\right.\tag{M5-2}\label{M5-2}
\end{gather}
where $m \in C(E, (F\mf{x})_j)\{n-\bar{x}_j\}[\beta(\mf{x},\bar{x}_j)] \subset C(EF, \mf{x})$.

\vspace{.1cm}
If $s_1=0$ and $s_0(\mf{v})>0$, the right multiplication by $e(EF) \bt r(\mf{x} \xra{i,0,(s_0)} \mf{y})$ is defined by:
\begin{gather}
m \times \left(e(EF) \bt r(\mf{x} \xra{i,0,(s_0)} \mf{y})\right)=
\left\{
\begin{array}{cc}
m \times \left(e(E) \bt e((F\mf{y})_{j+s_0})\right) & ~\mbox{if}~ j = j_0-s_0; \\
0 & \mbox{otherwise}.
\end{array}
\right.\tag{M5-3}\label{M5-3}
\end{gather}
where $m \in C(E, (F\mf{x})_j)\{n-\bar{x}_j\}[\beta(\mf{x},\bar{x}_j)] \subset C(EF, \mf{x})$.

\vspace{.1cm}
If $s_1>0$, the right multiplication by $e(EF) \bt r(\mf{x} \xra{i,s_1,\mf{v}} \mf{y})$ is defined as the zero map:
\begin{gather}
m \times \left(e(EF) \bt r(\mf{x} \xra{i,0,(s_0)} \mf{y})\right)=0.
\tag{M5-4}\label{M5-4}
\end{gather}
\subsubsection{The right multiplication by $\rho(I,EF) \bt e(\mf{x})$}
We discuss $(EF\mf{x})_{n-k}^{k+1} \in \bnk$ shown in the top left corner $C_{n-k}^{k+1}(EF, \mf{x})$ of the double complex (\ref{double EFx}) depending on $\bar{x}_{n-k}=n$ or $\bar{x}_{n-k}<n$.

\vspace{.2cm}
\n (Case 1) Suppose $\bar{x}_{n-k}=n$, i.e., the last state is $|0\ran$.
Then we have $(EF\mf{x})_{n-k}^{k+1}=\mf{x}$, $\beta(\mf{x},\bar{x}_{n-k})=k$ and
$C_{n-k}^{k+1}(EF, \mf{x})=PH(\mf{x})\oplus PH(\mf{x})\{1\}[1].$
The right multiplication
\begin{gather}\times (\rho(I,EF) \bt e(\mf{x})):C(I, \mf{x}) \ra  C(EF, \mf{x})\tag{M6-1-1}\label{M6-1-1}
\end{gather}
is defined by the identity map from $PH(\mf{x})=C(I, \mf{x})$ to $PH(\mf{x}) \subset C_{n-k}^{k+1}(EF, \mf{x})$.
An example is given in Figure \ref{5-4-1}.
\begin{figure}[h]
\begin{overpic}
[scale=0.3]{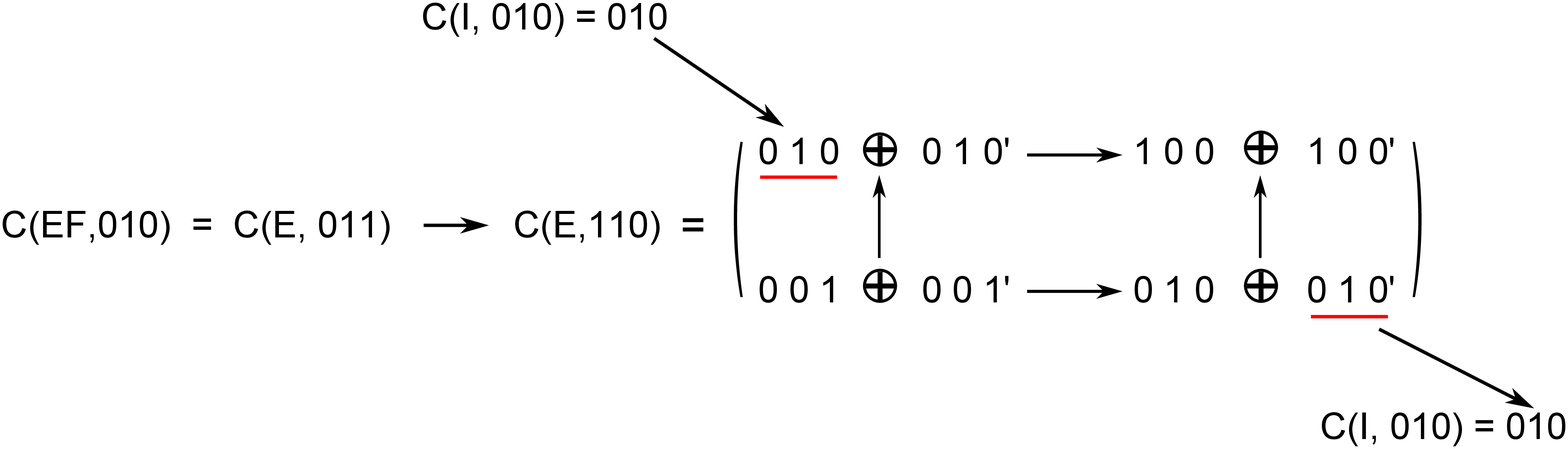}
\put(46,24){$id$}
\put(93,6){$id$}
\put(57,14){$d|_{ver}$}
\put(81,14){$d|_{ver}$}
\put(66,20){$d|_{hor}$}
\put(66,11){$d|_{hor}$}
\put(94,10){${\scriptstyle i=1}$}
\put(94,18){${\scriptstyle i=2}$}
\put(54,6){${\scriptstyle j=2}$}
\put(78,6){${\scriptstyle j=1}$}
\put(18,10){${\scriptstyle j=2}$}
\put(36,10){${\scriptstyle j=1}$}
\end{overpic}
\caption{The identity map from $C(I,010)$ to the top left corner is the right multiplication by $\rho(I,EF) \bt e(\mf{x})$ when $\bar{x}_{n-k}=n$. The identity map from the bottom right corner to $C(I,010)$ is the right multiplication by $\rho(EF,I) \bt e(\mf{x})$ when $\bar{x}_{1}=1$.}
\label{5-4-1}
\end{figure}
\begin{lemma} \label{rtIEF1}
The right multiplication by $\rho(I,EF) \bt e(\mf{x})$ in Case (1) commutes with $d$.
\end{lemma}
\begin{proof}
Since the differential on $C(I, \mf{x})$ is trivial, it suffices to prove that
\begin{align*}
d\left(m(I\mf{x}) \times (\rho(I,EF) \bt e(\mf{x}))\right)=& d^{k+1}_{n-k}|_{hor}(EF, \mf{x})(m((EF\mf{x})_{n-k}^{k+1}))\\
=& m((EF\mf{x})_{n-k}^{k+1}) \times (e(E) \bt r_F(\mf{x}; n-k))=0.
\end{align*}

Recall from Section 6.1.2 that
$r_F(\mf{x}; n-k)=r_0 \cdot r(\mf{z} \xra{q_{n-k} +1} (F\mf{x})_{n-k-1})$
for some $r_0 \in R_{n,k+1}$ and $\mf{z} \in \cal{B}_{n,k+1}$, where $q_{n-k}=\left| l\in \{1,...,k\} ~|~ x_l < \bar{x}_{n-k}=n\} \right|=k$.
Moreover,
$(E\mf{z})^{k+1}=(E(F\mf{x})_{n-k-1})^{k+1} \in \bnk.$
Hence, there exists $r_1 \in R_{n,k}$ such that
\begin{align*}
  & m((EF\mf{x})_{n-k}^{k+1}) \times (e(E) \bt r_F(\mf{x}; n-k)) \\
= & \left(m((EF\mf{x})_{n-k}^{k+1}) \times (e(E) \bt r_0)\right) \times \left(e(E) \bt r(\mf{z} \xra{k +1} (F\mf{x})_{n-k-1})\right) \\
= & \left(r_1 \cdot m((E\mf{z})^{k+1})\right) \times \left(e(E) \bt r(\mf{z} \xra{k +1} (F\mf{x})_{n-k-1})\right) \\
= & r_1 \cdot \left(m((E\mf{z})^{k+1}) \times (e(E) \bt r(\mf{z} \xra{k +1} (F\mf{x})_{n-k-1}))\right) \\
= & r_1 \cdot 0 =0,
\end{align*}
where the last step $m((E\mf{z})^{k+1}) \times (e(E) \bt r(\mf{z} \xra{k +1} (F\mf{x})_{n-k-1}))=0$ is from (\ref{M4-2}).
\end{proof}

\n (Case 2) Suppose $\bar{x}_{n-k}<n$. Then $\beta(\mf{x},\bar{x}_{n-k})=k+n-\bar{x}_{n-k}$ and $C_{n-k}^{k+1}(EF, \mf{x})$ is $$PH((EF\mf{x})_{n-k}^{k+1})\{n-\bar{x}_{n-k}\}[n-\bar{x}_{n-k}]\oplus PH((EF\mf{x})_{n-k}^{k+1})\{n-\bar{x}_{n-k}+1\}[n-\bar{x}_{n-k}+1].$$
Note that $(EF\mf{x})_{n-k}^{q_{n-k}+1}=\mf{x}$ and there is a path from $\mf{x}$ to $(EF\mf{x})_{n-k}^{k+1}$ in $\gnk$:
$$\mf{x}=(EF\mf{x})_{n-k}^{q_{n-k}+1} \xra{q_{n-k}+1} (EF\mf{x})_{n-k}^{q_{n-k}+2} \xra{q_{n-k}+2} \cdots \xra{k} (EF\mf{x})_{n-k}^{k+1}.$$
Let $r_{I,EF}(\mf{x})$ be a product of the corresponding $n-\bar{x}_{n-k}=k-q_{n-k}$ generators in $H(R_{n,k})$.
The right multiplication is a map $C(I, \mf{x}) \ra  C(EF, \mf{x})$ of left $H(R_n)$-modules defined on the generators by
\begin{gather}
m(I\mf{x}) \times (\rho(I,EF) \bt e(\mf{x}))=r_{I,EF}(\mf{x}) \cdot m((EF\mf{x})_{n-k}^{k+1}).\tag{M6-1-2}\label{M6-1-2}
\end{gather}
An example is given in Figure \ref{5-4-2}.
\begin{figure}[h]
\begin{overpic}
[scale=0.3]{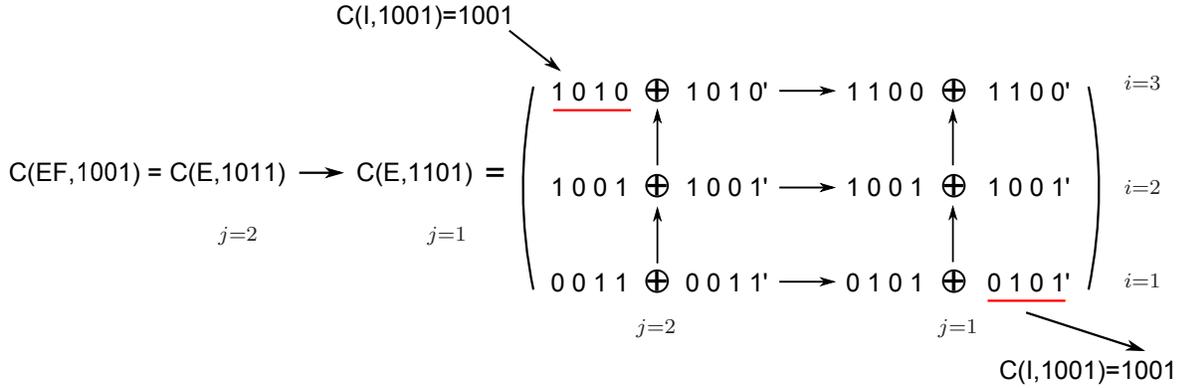}
\put(96,8){${\scriptstyle i=1}$}
\put(96,16){${\scriptstyle i=2}$}
\put(96,25){${\scriptstyle i=3}$}
\put(54,4){${\scriptstyle j=2}$}
\put(80,4){${\scriptstyle j=1}$}
\put(18,12){${\scriptstyle j=2}$}
\put(36,12){${\scriptstyle j=1}$}
\end{overpic}
\caption{The map from $C(I,1010)$ to the top left corner is the right multiplication by $\rho(I,EF) \bt e(\mf{x})$ when $\bar{x}_{n-k}<n$. The map from the bottom right corner to $C(I,1010)$ is the right multiplication by $\rho(EF,I) \bt e(\mf{x})$ when $\bar{x}_{1}>1$.}
\label{5-4-2}
\end{figure}

\begin{lemma} \label{rtIEF2}
The right multiplication by $\rho(I,EF) \bt e(\mf{x})$ in Case (2) commutes with $d$.
\end{lemma}
\begin{proof}
Since the differential on $C(I, \mf{x})$ is zero, it suffices to prove that
\begin{align*}
& d^{k+1}_{n-k}|_{hor}(EF, \mf{x})(m({I\mf{x}}) \times (\rho(I,EF) \bt e(\mf{x}))) \\
= & (r_{I,EF}(\mf{x}) \cdot m((EF\mf{x})_{n-k}^{k+1})) \times (e(E) \bt r_F(\mf{x}; n-k)) \\
= & r_{I,EF}(\mf{x}) \cdot r_{EF}(\mf{x}; n-k,k+1) = 0,
\end{align*}
where $r_{EF}(\mf{x}; n-k,k+1)$ is a product of $q_{n-k}-q_{n-k-1}+1$ generators in $H(R_{n,k})$ induced by the path in $\gnk$:
$(EF\mf{x})_{n-k}^{k+1} \xra{q_{n-k-1}+1}  \cdots \xra{q_{n-k}+1} (EF\mf{x})_{n-k-1}^{k+1}.$
Then $r_{I,EF}(\mf{x}) \cdot r_{EF}(\mf{x}; n-k,k+1)$ is induced by the concatenation of the two paths:
\begin{gather*}
\mf{x} \xra{q_{n-k}+1} \cdots \xra{k} (EF\mf{x})_{n-k}^{k+1} \xra{q_{n-k-1}+1}  \cdots \xra{q_{n-k}+1} (EF\mf{x})_{n-k-1}^{k+1}.
\end{gather*}
By using the commutation relation \ref{R2} to rearrange the arrows, the path above can be written as:
$$\mf{x} \xra{q_{n-k-1}+1} \cdots \xra{q_{n-k}} \mf{z}^0 \xra{q_{n-k}+1} \mf{z}^1 \xra{q_{n-k}+1} \mf{z}^2 \xra{q_{n-k}+2} \cdots \xra{k} (EF\mf{x})_{n-k-1}^{k+1}.$$
Hence,
$r_{I,EF}(\mf{x}) \cdot r_{EF}(\mf{x}; n-k,k+1)= \cdots r(\mf{z}^0 \xra{q_{n-k}+1} \mf{z}^1) \cdot r(\mf{z}^1 \xra{q_{n-k}+1} \mf{z}^2) \cdots =0$
from Relation \ref{R1}.
\end{proof}

\subsubsection{The right multiplication by $\rho(I\mf{x} \xra{i} EF\mf{x})$}
Because the right multiplication with
\begin{gather*}
\rho(I,EF)\bt e(\mf{x}))\cdot(e(EF) \bt \rho(\mf{x} \xra{k} \mf{x})) + (e(I) \bt \rho(\mf{x} \xra{k} \mf{x}))\cdot(\rho(I,EF)\bt e(\mf{x}))
\end{gather*}
is possibly nonzero, we represent the expression above as the differential of $\rho(I\mf{x} \xra{k} EF\mf{x})$ for $\mf{x} \in \bnk$ with $x_k=n$ in Definition \ref{arn}. See Figure \ref{5-4-4} for an example.

Recall from Definition \ref{arn} that $\rho(I\mf{x} \xra{i} EF\mf{x})$ exists if and only if $1\leq i \leq k <n$ and $\mf{x} \in \bnk$ such that $x_i=n-k+i$.
Note that the condition $x_i=n-k+i$ implies that $x_l=n-k+l$ for all $i \leq l \leq k$, i.e., each of the last $k-i+1$ states in $\mf{x}$ is $|1\ran$.
We are interested in $(i,n-k)$-th entry $C_{n-k}^{i}(EF, \mf{x})$ of the double complex (\ref{double EFx}):
$$PH((EF\mf{x})_{n-k}^{i})\{n-\bar{x}_{n-k}\}[n-\bar{x}_{n-k}+k-i+1]\oplus PH((EF\mf{x})_{n-k}^{i})\{n-\bar{x}_{n-k}+1\}[n-\bar{x}_{n-k}+k-i+2].$$
Note that $(EF\mf{x})_{n-k}^{q_{n-k}+1}=\mf{x}$ and $q_{n-k}+1 \leq i$.
Hence, there is a path from $\mf{x}$ to $(EF\mf{x})_{n-k}^{k+1}$ through $(EF\mf{x})_{n-k}^{i}$ in $\gnk$:
$$\mf{x}=(EF\mf{x})_{n-k}^{q_{n-k}+1} \xra{q_{n-k}+1} \cdots \xra{i-1} (EF\mf{x})_{n-k}^{i} \xra{i} \cdots \xra{k} (EF\mf{x})_{n-k}^{k+1}.$$
Let $r_{I,EF}(\mf{x};i)$ be the product of $i-q_{n,k}-1$ generators in $H(R_{n,k})$ corresponding to the path from $\mf{x}$ to $(EF\mf{x})_{n-k}^{i}$.
%$$r_{I,EF}(\mf{x})= r_{I,EF}(\mf{x};k) \cdot r((EF\mf{x})_{n-k}^{k} \xra{k} (EF\mf{x})_{n-k}^{k+1}).$$
Then the right multiplication is a map $C(I, \mf{x}) \ra  C(EF, \mf{x})$ of left $H(R_n)$-modules defined on the generators by
\begin{gather}
m(I\mf{x}) \times (\rho(I\mf{x} \xra{i} EF\mf{x}))=r_{I,EF}(\mf{x};i) \cdot m((EF\mf{x})_{n-k}^{i}),\tag{M6-2}\label{M6-2}
\end{gather}
where $m((EF\mf{x})_{n-k}^{i}) \in C_{n-k}^{i}(EF, \mf{x})$.

\begin{example}
Let $\mf{x}=(2,3)=|011\ran \in \cal{B}_{3,2}$, then $x_i=n-k+i$ for $i=1,2$ and $n=3, k=2$.
Hence there exist
$$r_1=\rho(I|011\ran \xra{1} EF|011\ran), \quad r_2=\rho(I|011\ran \xra{2} EF|011\ran), \quad r_3=\rho(I,EF)\bt e(|011\ran).$$
The right multiplications by $r_1, r_2$ and $r_3$ are described in Figure \ref{5-4-3}.
More precisely,
\begin{align*}
m(I|011\ran) \times r_3 &= r(|011\ran \xra{1} |101\ran) \cdot r(|101\ran \xra{2} |110\ran) \cdot m(|110\ran) \in C(EF, |011\ran), \\
m(I|011\ran) \times r_2 &= r(|011\ran \xra{1} |101\ran) \cdot m(|101\ran) \in C(EF, |011\ran), \\
m(I|011\ran) \times r_1 &= m(|011\ran) \in C(EF, |011\ran).
\end{align*}
\begin{figure}[h]
\begin{overpic}
[scale=0.3]{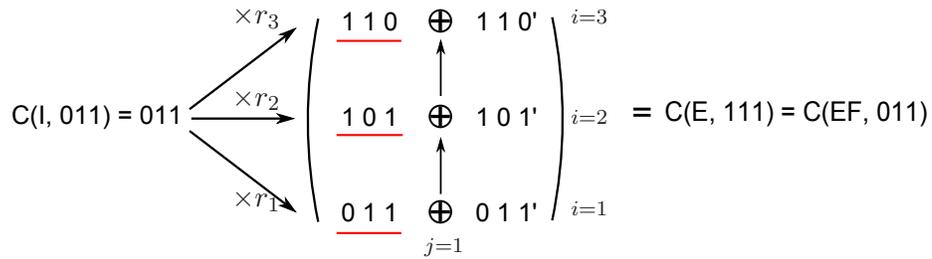}
\put(61,2){${\scriptstyle i=1}$}
\put(61,12){${\scriptstyle i=2}$}
\put(61,23){${\scriptstyle i=3}$}
\put(45,-2){${\scriptstyle j=1}$}
\put(24,23){$\times r_3$}
\put(24,14){$\times r_2$}
\put(24,3){$\times r_1$}
\end{overpic}
\caption{The right multiplication with $r_i$ from $C(I,|011\ran)$ to the summands of $C(EF,|011\ran)$ with red lines underlying}
\label{5-4-3}
\end{figure}
\end{example}

We verify that the definitions (\ref{M6-1-1}), (\ref{M6-1-2}) and (\ref{M6-2}) are compatible with the DG structure on $\aor$.
\begin{lemma} \label{rtIEFrho}
For $1\leq i \leq k$ and $\mf{x} \in \bnk$ with $x_i=n-k+i$,
$$d(m(I\mf{x}) \times (\rho(I\mf{x} \xra{i} EF\mf{x})))=m(I\mf{x}) \times d(\rho(I\mf{x} \xra{k} EF\mf{x})).$$
\end{lemma}
\begin{proof}
For $i=k$, we can reduce the case in general to the local model.
We give a diagrammatic proof for the local model: $r_1=\rho(I|01\ran \xra{1} EF|01\ran)$ in Figure \ref{5-4-4}.
Let $\rho=\rho(|01\ran \xra{1} |01\ran)$ and $r_0=r(|01\ran \xra{1} |10\ran)$.
Recall that $d(r_1)=r_2 \cdot r_3 + r_4 \cdot r_2,$ here
$$r_2=\rho(I,EF)\bt e(|01\ran), \quad r_3=e(EF)\bt \rho, \quad r_4=e(I)\bt \rho.$$
In Figure \ref{5-4-4}, the right multiplications by $r_1, r_2$ and $r_3, r_4$ are given in the left-hand and right-hand diagrams, respectively.

\begin{figure}[h]
\begin{overpic}
[scale=0.3]{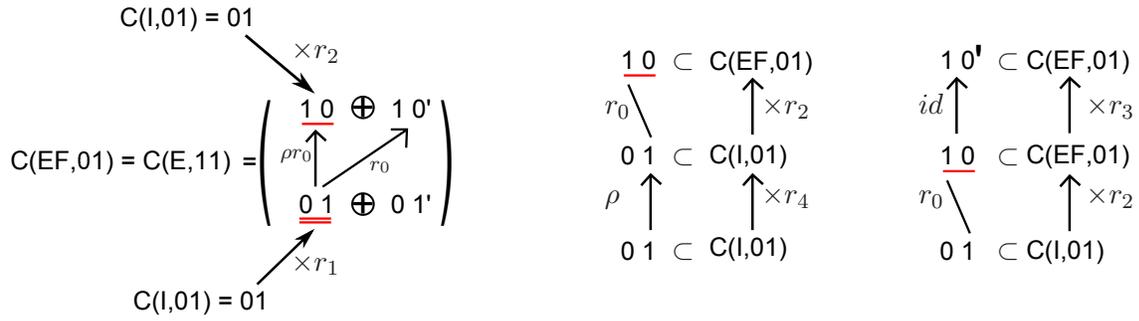}
\put(81,18){${id}$}
\put(81,10){${r_0}$}
\put(25,4){$\times r_1$}
\put(25,23){$\times r_2$}
\put(24,14){${\scriptstyle {\rho r_0}}$}
\put(32,13){${\scriptstyle r_0}$}
\put(67,18){$\times r_2$}
\put(67,10){$\times r_4$}
\put(96,18){$\times r_3$}
\put(96,10){$\times r_2$}
\put(53,10){${\rho}$}
\put(53,18){${r_0}$}
\put(59,5){$\subset$}
\put(88,5){$\subset$}
\put(59,13.5){$\subset$}
\put(88,13.5){$\subset$}
\put(59,22){$\subset$}
\put(88,22){$\subset$}
\end{overpic}
\caption{The map from $C(I,01)$ to the bottom left corner is the right multiplication by $\rho(I|01\ran \xra{1} EF|01\ran)$. The right-hand diagram gives the right multiplications by $d(\rho(I|01\ran \xra{1} EF|01\ran))$.}
\label{5-4-4}
\end{figure}

The right-hand side of the equation is
\begin{align*}
 m(I|01\ran) \times d(r_1)=& (m(I|01\ran) \times r_4) \times r_2 + (m(I|01\ran) \times r_2) \times r_3 \\
= & \rho \cdot r_0 \cdot m(|10\ran) + r_0 \cdot m'(|10\ran) \in C(EF, |01\ran),
\end{align*}
which agrees with the left-hand side: $d(m(I|01\ran) \times r_1)=d(m(|01\ran)) \in C(EF, |01\ran)$.

The proof for the case $i<k$ is similar.
\end{proof}

\subsubsection{The right multiplication by $\rho(EF,I) \bt e(\mf{x})$}
The construction in this subsection is dual to that in Section 6.4.2.
We discuss $(EF\mf{x})_{1}^{1} \in \bnk$ shown in the bottom right corner $C_{1}^{1}(EF, \mf{x})$ of the double complex (\ref{double EFx}) depending on $\bar{x}_{1}=1$ or $\bar{x}_{1}>1$.

\vspace{.1cm}
\n (Case 1) Suppose $\bar{x}_{1}=1$, i.e., the first state in the tensor product presentation is $|0\ran$.
Then we have $(EF\mf{x})_{1}^{1}=\mf{x}$, $\beta(\mf{x},\bar{x}_{1})=2k$ and
$C_{1}^{1}(EF, \mf{x})=PH(\mf{x})\{n-1\}[2k]\oplus PH(\mf{x})\{n\}[2k+1].$
The right multiplication
\begin{gather}
\times (\rho(EF,I) \bt e(\mf{x})):C(EF, \mf{x}) \ra  C(I, \mf{x})\tag{M7-1-1}\label{M7-1-1}
\end{gather}
is defined by the identity map from $PH(\mf{x})\{n\}[2k+1] \subset C_{1}^{1}(EF, \mf{x})$ to $PH(\mf{x})=C(I, \mf{x})$.
See Figure \ref{5-4-1} for an example.

%\begin{lemma} \label{rtEFI1}
%The right multiplication by $\rho(EF,I) \bt e(\mf{x})$ in Case (1) commutes with the differential.
%\end{lemma}
%\begin{proof}
%Since the differential on $C(I, \mf{x})$ is zero, it suffices to prove that
%$$d^{1}_{2}|_{hor}(m)=m \times (e(E) \bt r_F(\mf{x};2)) \in PH(\mf{x})\{n-1\}[2k]$$ for $m \in C^{1}_{2}(EF, \mf{x})$.
%Recall that
%$r_F(\mf{x}; 2)=r((F\mf{x})_{2} \xra{q_{1} +1} \mf{z}) \cdot r_0$
%for some $r_0 \in R_{n,k+1}$ and $\mf{z} \in \cal{B}_{n,k+1}$, where $q_{1}=\left| l\in \{1,...,k\} ~|~ x_l < \bar{x}_{1}=1\} \right|=0$.
%Moreover,
%$(E(F\mf{x})_{2})^{1}=(E\mf{z})^{1} \in \bnk.$
%Hence, there exist $r_1 \in R_{n,k}$ and $r_2 \in H(R_{n,k})$ such that
%\begin{align*}
%  m \times (e(E) \bt r_F(\mf{x};2))
%= & (m \times (e(E) \bt r((F\mf{x})_{2} \xra{1} \mf{z}))) \times (e(E) \bt r_1) \\
%= & (r_2 \cdot m((E\mf{z})^{1})) \times (e(E) \bt r_1) \in PH(\mf{x})\{n-1\}[2k].
%\end{align*}
%from the definition of right multiplication in Section 6.2.4 (2).
%\end{proof}

\vspace{.1cm}
\n (Case 2) Suppose $\bar{x}_{1}>1$, then $\beta(\mf{x},\bar{x}_{1})=2k-\bar{x}_{1}+1$ and $C_{1}^{1}(EF, \mf{x})$ is
$$PH((EF\mf{x})_{1}^{1})\{n-\bar{x}_{1}\}[2k-\bar{x}_{1}+1]\oplus PH((EF\mf{x})_{1}^{1})\{n-\bar{x}_{1}+1\}[2k-\bar{x}_{1}+2].$$
Note that $(EF\mf{x})_{1}^{q_{1}+1}=\mf{x}$ and there is a path from $(EF\mf{x})_{1}^{1}$ to $\mf{x}$ in $\gnk$:
$$(EF\mf{x})_{1}^{1} \xra{1} (EF\mf{x})_{1}^{2} \xra{2} \cdots \xra{q_{1}} \mf{x}.$$
Let $r_{EF,I}(\mf{x})$ be a product of the corresponding $q_1=\bar{x}_{1}-1$ generators in $H(R_{n,k})$.
The right multiplication is a map of left $H(R_n)$-modules: $C(EF, \mf{x}) \ra  C(I, \mf{x})$ defined on the generators by
\begin{gather}
m \times (\rho(EF,I) \bt e(\mf{x}))=
\left\{
\begin{array}{cc}
r_{EF,I}(\mf{x}) \cdot m(I\mf{x}) & \mbox{if}~ m=m'((EF\mf{x})_{1}^{1});\\
0 & \mbox{otherwise}.
\end{array}
\right.
\tag{M7-1-2}\label{M7-1-2}
\end{gather}
An example is given in Figure \ref{5-4-2}.

\begin{lemma} \label{rtEFI}
In both cases the right multiplication by $\rho(EF,I) \bt e(\mf{x})$ commutes with $d$.
\end{lemma}
\begin{proof}
The proof is similar to those of Lemmas \ref{rtIEF1} and \ref{rtIEF2}.
\end{proof}

\subsubsection{The right multiplication by $\rho(EF\mf{x} \xra{j} I\mf{x})$}
The construction in this subsection is dual to that in Section 6.4.3.
Recall from Definition \ref{arn} that $\rho(EF\mf{x} \xra{j} I\mf{x})$ exists if and only if $1\leq j \leq k <n$ and $\mf{x} \in \bnk$ such that $x_j=j$.
Note that the condition $x_j=j$ implies that $x_l=l$ for all $1 \leq l \leq j$, i.e., each of the first $j$ states in $\mf{x}$ is $|1\ran$.
We are interested in $(j+1,1)$-th entry $C_{1}^{j+1}(EF, \mf{x})$ of the double complex (\ref{double EFx}):
$$PH((EF\mf{x})_{1}^{j+1})\{n-\bar{x}_{1}\}[2k-\bar{x}_{1}+j-1]\oplus PH((EF\mf{x})_{1}^{j+1})\{n-\bar{x}_{1}+1\}[2k-\bar{x}_{1}+j].$$
Note that $(EF\mf{x})_{1}^{q_{1}+1}=\mf{x}$.
Hence, there is a path from $(EF\mf{x})_{1}^{1}$ to $\mf{x}$ through $(EF\mf{x})_{1}^{j+1}$ in $\gnk$:
$$(EF\mf{x})_{1}^{1} \xra{1} \cdots \xra{j} (EF\mf{x})_{1}^{j+1} \xra{j+1} \cdots \xra{q_{1}} (EF\mf{x})_{1}^{q_{1}+1}=\mf{x}.$$
Let $r_{EF,I}(\mf{x};j+1)$ be a product of $q_1-j$ generators in $H(R_{n,k})$ corresponding to the path from $(EF\mf{x})_{1}^{j+1}$ to $\mf{x}$.
Then the right multiplication is a map $C(EF, \mf{x}) \ra  C(I, \mf{x})$ of left $H(R_n)$-modules defined on the generators by
\begin{gather}
m \times (\rho(EF\mf{x} \xra{j} I\mf{x}))=
\left\{
\begin{array}{cc}
r_{EF,I}(\mf{x};j+1) \cdot m(I\mf{x}) & \mbox{if}~ m=m'((EF\mf{x})_{1}^{j+1});\\
0 & \mbox{otherwise}.
\end{array}
\right.
\tag{M7-2}\label{M7-2}
\end{gather}

\begin{lemma} \label{rtEFIrho}
The right multiplication is compatible with the DG structure on $\aor$:
$$d(m'((EF\mf{x})_{1}^{j+1})) \times (\rho(EF\mf{x} \xra{j} I\mf{x}))=m'((EF\mf{x})_{1}^{j+1}) \times d((\rho(EF\mf{x} \xra{j} I\mf{x}))),$$
for $1\leq j \leq k <n$ and $\mf{x} \in \bnk$ such that $x_j=j$.
\end{lemma}
\begin{proof}
The proof is similar to that of Lemma \ref{rtIEFrho}.
\end{proof}

This concludes the definition of the right $\aor$-module structure on $C_n$.

\begin{prop} \label{welldef}
The definitions of the right multiplications by $\aor$ is well-defined.
\end{prop}
\begin{proof}
\n (1) For the relatioons in $\aor$, we need to verify that
$$
(m\times r_1)\times r_2 = (m\times r_1') \times r_2',
$$
if $r_1 \cdot r_2 = r_1' \cdot r_2' \in \aor$ for $m \in C$ and generators $r_1, r_2, r_1', r_2' \in \aor$.
We checked the relations $(e(\g)\bt \rho(\mf{x} \xra{i} \mf{x}))^2=0$ for $\g=E, F$ in Lemmas \ref{rtF relation} and \ref{rtE relation}.
The commutation relations which come from isotopies of stackings of disjoint rook diagrams are easily verified since the definition of the right multiplication only depends on local properties of the rook diagrams.

\vspace{.1cm}
\n (2) For the DG structure, we need to verify that
$$d(m \times r)=dm \times r + m \times dr,$$
for $m \in C$ and any generator $r \in \aor$.
We proved it when
\begin{itemize}
\item $d(r)=0$ in Lemmas \ref{rtFrho}, \ref{rtFr1}, \ref{rtErho}, \ref{rtEr1}, \ref{rtIEF1}, \ref{rtIEF2}, and \ref{rtEFI};
\item $r= e(F) \bt r(\mf{x} \xra{i,s_1,\mf{v}} \mf{y})$ with $s_1=0$ in Lemma \ref{rtFr2};
\item $r= e(E) \bt r(\mf{x} \xra{i,s_1,\mf{v}} \mf{y})$ with $s_0(\mf{v})=0$ in Lemma \ref{rtEr2};
\item $r = \rho(I\mf{x} \xra{i} EF\mf{x}))$ and $\rho(EF\mf{x} \xra{j} I\mf{x})$ in Lemmas \ref{rtIEFrho} and \ref{rtEFIrho}, respectively.
\end{itemize}
The proof for other cases is similar and we leave it to the reader.
\end{proof}

Since the right multiplications are defined as maps of left $H(R_n)$-modules, we have
$$a \cdot (m \times r)=(a \cdot m) \times r,$$
for $a \in H(R_n), r \in \aor$ and $m \in C_n$.
Hence $C_n$ is a $t$-graded DG $(H(R_n), \aor)$-bimodule.

\section{The categorical action of $HP(A)$ on $HP(H(R_n))$}
In this section, we use the $(H(R_n), \aor)$-bimodule $C_n$ to categorify the action of $\ut$ on $V_1^{\ot n}$.
Let $\eta_n: DGP(\aor) \xra{C_n \otimes_{\aor} -} DGP(H(R_n))$ be a functor of tensoring with the DG $(H(R_n), \aor)$-bimodule $C_n$ over $\aor$.

\begin{lemma} \label{etan}
For all $\g \in \cal{B}$ and $\mf{x} \in \bn$, $\eta_n(P(\g, \mf{x}))=C_n(\g, \mf{x}) \in DGP(H(R_n)).$
\end{lemma}

\begin{proof}
The proof is similar to that of Lemma \ref{eta}.
\end{proof}

There is an induced exact functor $\eta_n: HP(\aor) \xra{C_n \otimes_{\aor} -} HP(H(R_n))$ between the $0$th homology categories.
We choose an equivalence $\cal{F}_n: HP(A \ot H(R_n)) \ra HP(\aor)$ of triangulated categories from Lemma \ref{k0 ut vn}.
Let $\cal{M}_n=\eta_n \circ \cal{F}_n \circ \chi_n$ be a composition:
$$HP(A) \times HP(H(R_n)) \xra{\chi_n}  HP(A \ot H(R_n)) \xra{\cal{F}_n} HP(\aor) \xra{\eta_n} HP(H(R_n)),$$
where $\chi_n$ is given in Definition \ref{chi} which induces the tensor product on the Grothendieck groups.

\begin{proof} [Proof of Theorem \ref{thm-utvn}]
We use $\{[P(\g, \mf{x})]\}$ as a basis of $K_0(HP(\aor))$ to compute $K_0(\eta_n)$.
By Lemma \ref{k0 utvn},
$$K_0(\eta_n)(\g \ot \mf{x})=  K_0(\eta_n)([P(\g, \mf{x})])
=  [C \ot P(\g, \mf{x})]
=  [C(\g, \mf{x})]
=  \g(\mf{x}) \in V_1^{\ot n},$$
Hence, the $\zt$-linear map
$K_0(\cal{M}_n): K_0(HP(A)) \times K_0(HP(H(R_n))) \ra K_0(HP(H(R_n)))$
agrees with the action of $\ut$ on $V_1^{\ot n}$: $\ut \times V_1^{\ot n} \ra V_1^{\ot n}$.
\end{proof}

\begin{rmk}
It is natural to ask whether the categorical action is associative up to equivalence:
$$
\xymatrix{
HP(A) \times HP(A) \times HP(H(R_n)) \ar[r]^-{id \times \cal{M}_n} \ar[d]^{\cal{M} \times id} & HP(A) \times HP(H(R_n)) \ar[d]^{\cal{M}_n} \\
HP(A) \times HP(H(R_n)) \ar[r]^{\cal{M}_n} & HP(H(R_n)).
}$$
The question is equivalent to verifying some associativity relation on various DG bimodules.
The computation is quite technical and we leave it to future work.
\end{rmk}

\end{document}